\AtBeginDocument{}
\AtBeginDocument{\renewcommand{\bT}{\mathbf{T}}}
\documentclass[12pt]{amsart}
\usepackage{macros}

\begin{document}

\title{Derived Manifolds as Differential Graded Manifolds}

\author{David Carchedi}

\maketitle

\begin{abstract}
On one hand, together with Pelle Steffens, we recently characterized the $\i$-category of derived manifolds up to equivalence by a universal property. On the other hand, it is shown in \cite{dgder} that the category of differential graded manifolds admits a homotopy theory. In this paper, we prove that the associated $\i$-category of dg-manifolds is equivalent to that of derived manifolds. This allows the well developed theory of dg-geometry to live in symbiosis with $\i$-categorical approach to derived differential geometry. As a consequence, we prove that a dg-$\Ci$-algebra is homotopically finitely presented if and only if it is quasi-isomorphic to smooth functions on a dg-manifold.
\end{abstract}

\tableofcontents
\part*{Introduction}
	In this paper, we show that differential graded manifolds can be used to model derived manifolds. Differential graded manifolds (dg-manifolds) generalize the notion of smooth manifolds by allowing their algebras of smooth functions to be differential graded commutative algebras of a specific local form. This can be articulated using the language of supergeometry. They arose out of ideas developed in physics to deal with degenerate action functionals, specifically, in the BRST and BV quantization schemes. Kontsevich promoted dg-manifolds in the 90s as a formulation of derived geometry in the context of deformation theory, and they were adapted in the algebro-geometric context to define dg-schemes \cite{DGk1,mod1,mod2,Kai1,Kai2}. Unfortunately, dg-schemes turned out to be inadequate for the purposes of derived algebraic geometry, since for example, one should be able to homotopy-coherently glue together structure sheaves along quasi-isomorphisms, and they also should satisfy higher descent. This is one of many reasons for the higher categorical definition of the modern-day derived scheme. We show that in the $\Ci$-setting, these homotopical deficiencies are not present, and all that one must do in order to model derived manifolds with dg-manifolds is to formally invert the quasi-isomorphisms (in the $\i$-categorical sense).
	\section*{Organization and main results}
	The organization of this paper is slightly non-standard in that it is divided into two parts. Part \ref{part:background}, containing Sections \ref{sec:dggeom} --- \ref{sec:sch}, provides the reader with extensive background material. It may be skipped by experts. Part \ref{sec:sch} contains the main results of the paper.
	
	Section \ref{sec:history} is an overview of the current state of affairs of derived differential geometry, provided to orient the reader and place this paper within the existing literature.
	
	Part \ref{part:background} starts with Section \ref{sec:dggeom}, which provides a detailed review of graded supergeometry for the unacquainted reader. Experts may wish to briefly skim this section, as we will switch back and forth between equivalent formulations of the theory of dg-manifolds in the rest of the paper. Specifically, we discuss the following three equivalent categories: the category of dg-manifolds, the category of supermanifolds equipped with an action of $\End\left(\R^{0|1}\right),$ and the theory of bundles of $L_\i\left[1\right]$-algebras.
	
	Section \ref{sec:cialg} reviews the concept of $\Ci$-algebras (a.k.a. $\Ci$-rings) and their homotopical versions coming from simplicial $\Ci$-algebras and differential graded $\Ci$-algebras. It also recalls the definition of a derived manifold (in Section \ref{sec:dmfd}). Finally, the theory of Koszul complexes is explained in this setting, paying careful attention to the case of odd regular sequences.
	
	Section \ref{sec:sch} reviews the theory of derived $\Ci$-schemes.
	
	Part \ref{part:new} is where the main new results are. It starts with Section \ref{sec:ii} in which we connect the theory of dg-manifolds and the theory of derived $\Ci$-schemes. The main result is the following theorem:
	
	\begin{theorem*}(Theorem \ref{thm:sch}).
Let $\left(M,\O_\M,\D\right)$ be a dg-manifold. Then there is a canonical equivalence
$$\Speci\left(\Ci\left(\M\right),D\right) \simeq \left(\pi_0\left(\M\right),\O_\M|_{\pi_0\left(M\right)}\right),$$
where $\Speci\left(\Ci\left(\M\right),D\right)$ is the derived affine $\Ci$-scheme associated to dg-$\Ci$-algebra of smooth functions on $\M.$
\end{theorem*}

In section \ref{sec:catfib}, we review the category of fibrant objects structure on bundles of curved $L_\i\left[1\right]$-algebras developed in \cite{dgder}, and translate it into the language of dg-manifolds. We also prove the following crucial conjecture from that paper:

\begin{theorem*}(Theorem \ref{thm:weqi}.)
Let $f:\left(M,L,\lambda\right) \to \left(N,E,\mu\right)$ a morphism of bundles of curved $L_\i\left[1\right]$-algebras with corresponding dg-manifolds $\left(\M,D\right)$ and $\left(\cN,D'\right)$ respectively. Then $f$ is a weak equivalence if and only if the induced map $\Ci\left(\cN\right) \to \Ci\left(\M\right)$ is a quasi-isomorphism.
\end{theorem*}

In Section \ref{sec:fibprod} we explain how to compute a special case of derived fiber product of dg-manifolds, and use it to prove that every dg-manifold can be constructed through a finite iteration of derived fibered products, starting with ordinary manifolds (Theorem \ref{thm:decomp}.)

Section \ref{sec:eq} contains the main result of this paper:

\begin{theorem*}(Theorem \ref{thm:main}.)
There is a canonical equivalence of $\i$-categories
$$\dgMani \stackrel{\sim}{\longrightarrow} \DMfd$$
between the $\i$-category of dg-manifolds and the $\i$-category of derived manifolds.
\end{theorem*}

\section*{Conventions and notation}
\begin{itemize}
\item By an $\i$-category we mean $\left(\i,1\right)$-category or quasi-category. $\Spc$ will denote the $\i$-category of spaces.
\item We will use the following font choices to differentiate different objects:

A graded manifold or dg-manifold will be written with this font: $\M$

A supermanifold will be written with this font: $\sM.$

If $\M$ is a dg-manifold, we will denote the corresponding supermanifold with an $\End\left(\R^{0|1}\right)$-action as $\sM,$ and conversely.

\item We will use cohomological grading conventions, and Einstein summation convention throughout.
\end{itemize}

\subsection*{Acknowledgments}
We are particularly indebted to Dmitry Roytenberg, who took us under their differentially graded wings in graduate school, and has continued to provide support until present day.

\section{A (biased) history of derived differential geometry}\label{sec:history}

\subsection{Motivation for derived manifolds, and literature review.}
Recall that two closed submanifolds $M \hookrightarrow L$ and $N \hookrightarrow L$ of a smooth manifold $L,$ are said to intersect transversely if for all points $p \in M \cap N,$ if the two subspaces $T_pM$ and $T_pN$ of $T_p L$ collectively span $T_pL.$ In such a situation, the intersection is also a closed submanifold. More generally, given two smooth maps $f:M \to L$ and $g:N \to L,$ which are not necessarily embeddings, \emph{$f$ is transverse to $g$} if for all points $x \in M$ and $y \in N$ such that $f\left(x\right)=g\left(y\right),$ letting $p=f\left(x\right),$
$$f_*\left(T_xM\right)+g_*\left(T_yN\right)=T_pL.$$ In such a situation, the subspace 
\begin{equation}\label{eq:pbspace}
\left\{\left(x,y\right)\mspace{3mu}|\mspace{3mu} f\left(x\right)=g\left(y\right)\right\} \subseteq M \times N
\end{equation}
is a closed submanifold, and is the categorical fibered product $M \times_{L} N.$ A special case of this is when $N=\{p\},$ with $p \in L.$ Then $f$ is transverse to $g$ if there is a neighborhood $U$ of $f^{-1}\{p\}$ such that $f|_U$ is a submersion, and this recovers the well-known fact that the fibers of any submersion are a closed submanifold.

In the absence of transversality, the situation can get arbitrary bad, in the following precise sense. If $f:M \to \R$ is any smooth map, and $p \in \R,$ then by continuity $f^{-1}\left(\{p\}\right)$ is a closed subset of $M.$ Conversely, due to the existence of bump functions, \emph{any} closed subset of $M$ is the fiber over $\{0\}$ of some smooth function $f:M \to \R.$ In particular, one can realize the Cantor set as the zeros of a smooth function $f \in \Ci\left(M\right).$ This creates many problems, for example, one may want to construct a moduli space by considering the vanishing locus of a smooth section $s$ of a vector bundle $V \to M$ over a manifold, i.e. the subspace of $M$ on which $s$ is zero. If $s$ fails to meet the zero-section transversely, this intersection may not be a manifold, and the moduli space will become singular. Many famous moduli spaces arise this way, at least locally. For example, any point in a moduli space of $J$-holomorphic curves has a Kuranishi neighborhood, which essentially means that the neighborhood can be realized as the quotient of the vanishing locus of a section of a vector bundle over a manifold (with corners) by a finite group. These moduli spaces have important applications to Gromov-Witten theory.

In any theory of derived manifolds, given two smooth maps $f:M \to L$ and $g:N \to L,$ the (derived) fibered product $M \times_L N$ exists as a smooth object, and agrees with the ordinary fibered product when the maps are transverse. Such derived manifolds (and those locally of this form) are called \emph{quasi-smooth}. Of course, for any such theory to be useful, it must behave well in that derived intersections of arbitrary maps should enjoy the same nice properties as transversal intersections. In the pioneering work of Spivak, he lists several desiderata that a (simplicial) category of derived manifolds should satisfy. One of these is that for any manifold $M$ there exists a derived cobordism ring $$\Omega^\emph{der}\left(M\right)$$ which at the level of cocycles can be modeled by compact (quasi-smooth) derived manifold mapping to $M,$ in which the cup-product formula holds on the nose:
$$\left[\mathcal{N}_1\right] \smile \left[\mathcal{N}_2\right]=\left[\mathcal{N}_1 \times_M \mathcal{N}_2\right],$$ and
the induced map $$\Omega^\emph{der}\left(M\right) \to \Omega\left(M\right)$$ to the classical cobordism ring is an injection. He constructs an explicit (simplicial) category for which this map is in fact an isomorphism.

The reader may wonder about parenthetical insertion of the adjective simplicial. Spivak gives a very elegant proof (\cite[Proposition 1.10]{spivak}) that in order for a theory of derived manifolds to satisfy his desiderata, its objects cannot form a $1$-category, but must form a higher category. (There is a slight error in the proof, but it can be rectified by replacing $\mathbb{R}$ with $S^1.$) Thus, from purely geometric considerations, one must pursue higher category theory to find a theory of derived manifolds. There exists to date several approaches to derived geometry in the $C^\i$-setting, all of them higher categorical in nature e.g. \cite{spivak,joyce,derivjustden,dg2,Nuiten,univ,dgder,pelle}. Joyce's model, whose objects he calls $d$-manifolds, is in fact $2$-categorical. Although this theory loses some homotopical information (it is essentially a $2$-categorical shadow of the full theory \cite{Borisov2}), its applications have been highly impactful. In particular, Joyce proved that any moduli space of $\mathcal{J}$-holomorphic curves admits the structure of a $d$-orbifold with corners \cite[Theorem 16.2]{joyce2}.

For some applications, it is important to be able to take derived intersections \emph{of} derived intersections, and so forth. For example, suppose that $\pi_1:E_1 \to M$ and $\pi_2:E_2 \to M$ are fiber bundles and $$H:\Gamma_{M}\left(E_1\right) \to \Gamma_{M}\left(E_2\right)$$ is a non-linear differential operator of degree $k.$ Given $\psi \in \Gamma_M\left(E_2\right),$ one may be interested in those $\varphi \in \Gamma_M\left(E_1\right)$ for which $H\varphi = \psi.$ As a first approximation, one may want to study the \emph{formal solution space} as follows. Consider the $k^{th}$-order jet bundle 
$$\pi_1^k:\mathcal{J}^k_{M} E_1 \to M.$$ Then $H$ corresponds to a map of fiber bundles $\hat H:\mathcal{J}^k_{M}E_1 \to E_2,$ such that $H\left(\varphi\right)=\hat H\left(j^k\left(\varphi\right)\right).$ So one could first ask for the subspace $\mathit{FSol_\psi}$ of $\Gamma_M\left(\mathcal{J}^k_M E_1\right)$ on those sections $\lambda$ of $\mathcal{J}^k_{M}E_1$ such that $\hat H\left(\lambda\right)=\psi$ (and then from there try to figure out which ones arise as prolongations, i.e. which are holonomic sections). If transversality were not an issue, one could form the fibered product
$$\xymatrix{\mathcal{R} \ar[r]^{pr_2} \ar[d]_-{pr_1} & M \ar[d]^-{\psi}\\
\mathcal{J}^k E_1 \ar[r]_-{\hat H} & E_2.}$$
And then from there the fibered product
$$\xymatrix{\mathcal{S} \ar[r] \ar[d]_-{p} & \mathcal{R} \ar[d]^-{\langle \pi_1^k \circ pr_1,pr_2\rangle}\\
M \ar[r]_-{\Delta_M} & M \times M.}$$
Then one would have $\mathit{FSol_\psi}=\Gamma_{M} \left(\mathcal{S}\right).$

A common thread of all of the existing approaches to derived manifolds, with the notable exception of \cite{dgder}, is the prominent role played by $\Ci$-algebras (a.k.a. $\Ci$-rings). Roughly speaking, a $\Ci$-algebra is a commutative $\R$-algebra $\A$ with the additional structure of an induced $n$-ary operation on $\A$ for each smooth function $f:\R^n \to \R,$ natural in $f.$ The prototypical example of a $\Ci$-algebra is the ring of smooth functions on a manifold $M,$ and $\Ci\left(\R^n\right)$ is the free $\Ci$-algebra on $n$-generators. The category of $\Ci$-algebra posses a tensor product $\oinfty$ (coproduct) such that for any two manifolds $M$ and $N,$ one has
$$\Ci\left(M\right) \underset{\R} \otimes \Ci\left(N\right) \subsetneq\Ci\left(M\right) \oinfty \Ci\left(N\right) \cong \Ci\left(M \times N\right).$$ In fact, the functor given by taking smooth functions produces a fully-faithful functor from the category of manifolds to the opposite of the category of $\Ci$-algebras which \emph{preserves transverse pullbacks}. This fact, together with the fact that $\A/I$ is a $\Ci$-algebra for any ideal $I$ of $\A,$ is one of the main reasons these algebras form the cornerstone for many approaches to derived differential geometry. Indeed, \cite{Nuiten,pelle} develop the theory of derived algebraic geometry over homotopical $\Ci$-algebras, and Joyce's approach is a $2$-categorical enhancement of the ``classical'' theory of algebraic geometry over $\Ci$-algebras.

With so many theories of derived manifolds out there, one may wonder about the unicity of such a theory. Pelle Steffens and I rectify this in \cite{univ}, by classifying the $\i$-category of derived manifolds up to equivalence by a universal property. We then show that Spivak's model embeds fully faithfully onto the subcategory on the quasi-smooth objects. Here is a summary: whatever the $\i$-category $\mathsf{DMfd}$ is, it should receive a fully faithful functor from the category of smooth manifolds $$i:\Mfd \hookrightarrow \mathsf{DMfd}$$ and $i$ should preserve transverse pullbacks and the terminal object. Moreover, $\mathsf{DMfd}$ should be in a suitable sense closed under taking fibered products and retracts. One way to phrase this is to ask for $\mathsf{DMfd}$ to have finite limits and be idempotent complete (for $\i$-categories, having finite limits does not imply that idempotents split). Finally, one should make sense of the idea that derived manifolds are completely determined by how they are built out of manifolds using fibered products and retracts. This leads to the following definition:

\begin{definition}
The $\i$-category  $\mathsf{DMfd}$ of derived manifolds is the (essentially unique) idempotent complete $\i$-category with finite limits, together with a functor $i:\mathsf{Mfd} \to \mathsf{DMfd}$ which preserves transverse pullbacks and the terminal object, such that for all other idempotent complete $\i$-categories $\sC$ with finite limits, composition along $i$ induces an equivalence of $\i$-categories
$$\Fun^{\mathbf{lex}}\left(\mathsf{DMfd},\sC\right) \stackrel{\sim}{\longlongrightarrow} \Fun^\pitchfork\left(\Mfd,\sC\right)$$
between the $\i$-category of finite limit preserving functors from $\mathsf{DMfd}$ to $\sC$ and the $\i$-category of functors from $\Mfd$ to $\sC$ which preserves transverse pullbacks and the terminal object.
\end{definition}

\begin{remark}
It turns out one does not need to demand that $i$ be fully faithful, and that it is automatically satisfied.
\end{remark}

One can show that abstractly such an $\i$-category must exist, but of course, such a definition is not very tractable. However there is the following remarkable connection between derived manifolds and $\Ci$-algebras:

\begin{theorem} \cite[Theorem 5.3]{univ}
For all idempotent complete $\i$-categories $\sC$ with finite limits, composition with the canonical functor $\Ci \to \Mfd \to \mathsf{DMfd}$ induces an equivalence of $\i$-categories
$$\Fun^{\mathbf{lex}}\left(\mathsf{DMfd},\sC\right) \stackrel{\sim}{\longlongrightarrow} \Alg_{\Ci}\left(\sC\right),$$
between the $\i$-category of finite-limit preserving functor from $\DMfd$ to $\sC$ and the $\i$-category of $\Ci$-algebra objects in $\sC.$
\end{theorem}

The abstract formalism of algebraic theories implies the following corollary, which provides a concrete model for derived manifolds:

\begin{corollary}\label{cor:dmfduniv}
There is a canonical equivalence of $\i$-categories
$$\DMfd\simeq \left(\Alg_{\Ci}\left(\Spc\right)^{\mathbf{fp}}\right)^{op}$$ between the $\i$-category of derived manifolds and the opposite of the $\i$-category of homotopically finitely presented $\Ci$-algebras in spaces, i.e. the full subcategory on the compact objects.
\end{corollary}

This connects directly with the approach of Borisov-Noel \cite{derivjustden} based on simplicial $\Ci$-rings, as the model category thereof is a presentation for the $\i$-category $\Alg_{\Ci}\left(\Spc\right).$ There is also a spectrum functor $\Speci$ producing from a homotopical $\Ci$-algebra $\A$ a homotopically $\Ci$-ringed space $\Speci\left(\A\right),$ and from here, one can develop a good notion of derived $\Ci$-schemes. In scheme language, the above theorem states that $\DMfd$ is equivalent to the $\i$-category of affine $\Ci$-schemes of finite type. Thus, the theory of $\DMfd$ is also subsumed by the theories of derived algebraic geometry over homotopical $\Ci$-algebras developed in \cite{Nuiten,pelle}. In fact, this is also where the connection with Spivak's model lies, as his $\i$-category of derived manifolds can be identified almost on the nose with quasi-smooth derived $\Ci$-schemes (locally) of finite type.

\subsection{Dg-manifolds}
The odd theory out is the theory of derived manifolds developed in \cite{dgder}, based upon \emph{differential graded manifolds} (but using the alternative language of bundles of curved $L_\i\left[1\right]$-algebras), and is thus a theory developed using the formalism of differential-graded geometry--- which is based on supergeometry. The central objects of study are called dg-manifolds, which are also known as $Q$-supermanifolds, especially in the physics literature. Dg-manifolds are very tractable objects on which one can do differential geometry almost as usual, even in local coordinates. See Section \ref{sec:dggeom} for a detailed review of this theory. 

Dg-geometry stems from homological tools developed by  physicists several decades ago. For instance, when computing the Feynman diagrams of gauge theories, like those involved in quantum chromodynamics, physicists discovered that they had to add ``ghost fields'' to extract sensible answers. In a similar vein, physicists introduced ``antifields'' to partner with every field, as part of the Batalin-Vilkovisky (BV) formalism developed to deal with supergravity. Later, Kontsevich \cite{K} suggested how these ideas allow one to define a virtual fundamental class, a key technical tool in mirror symmetry. These ideas were developed further in the algebro-geometric to define \emph{differential graded schemes} in work of Ciocan-Fontanine and Kapranov \cite{DGk1,mod1,mod2} and subsequent work of Behrend \cite{Kai1,Kai2}, which were precursors to the modern-day derived scheme.

Dmitry Roytenberg and I have been working on developing a model for derived manifolds based on dg-manifolds for more than a decade. In fact, it was through these efforts that we defined and constructed a model category of dg-$\Ci$-algebras; the prototypical example of a dg-$\Ci$-algebra is the algebra of smooth functions on a dg-manifold. It was on this notion of homotopical $\Ci$-algebra that the thesis \cite{Nuiten} was based, where it was also shown the the associated $\i$-category $\dgci$ is equivalent to $\Alg_{\Ci}\left(\Spc\right)$ (and thus to the model using simplicial $\Ci$-rings).

The idea of an approach to derived manifolds based upon dg-manifolds is that dg-manifolds can be used to replace any smooth map with a submersion, up to quasi-isomorphism of dg-manifolds. For example, if $E\to M$ is a vector bundle and $s$ is a section, then insertion $\iota_s$ can be viewed as a derivation of the graded algebra $$\Gamma_{M}\left(\Lambda^\bullet\left(E^*\right)\right).$$ In the language of graded differential geometry, this graded algebra is the algebra of smooth functions on the underlying graded manifold of the vector bundle $E\left[-1\right] \to M$ obtained by shifting the fibers of $E$ into degree $1,$ and $\iota_s$ becomes a \emph{cohomological vector field} on this graded manifold, making $\left(E\left[-1\right],\iota_s\right)$ into a dg-manifold. It can be shown that if $s$ meets the zero-section $Z$ transversely, then the dg-algebra of smooth functions on $\left(E\left[-1\right],\iota_s\right)$ is quasi-isomorphic to $\Ci\left(s \cap Z\right).$ Using this, and the existence of tubular neighborhoods, one can easily produce from a smooth function $f:M \to N,$ a factorization $f=f' \circ \lambda,$ where
$$\xymatrix{& \mathcal{M} \ar[ld]_-{\pi} \ar[rd]^{f'} & \\ M & & N }$$
$\lambda$ is a section of $\pi,$ $\pi$ and $f'$ are submersions of graded manifolds, and $\pi$ is moreover a quasi-isomorphism. In particular, $f'$ will be transverse in the suitable graded sense to any smooth map to $N.$

For over a decade, Dmitry Roytenberg publicly conjectured the existence of the structure of a \emph{category of fibrant objects} on the category of non-positively graded dg-manifolds $\dgMan$ (a weakening of a model structure that only defines fibrations and weak equivalence, and does not require the existence of small (or even finite) limits and colimits) where:

\begin{itemize}
\item $f:\M \to \cN$ is a \textbf{fibration} if it is a submersion of graded manifolds.
\item $f:\M \to \cN$ is a \textbf{weak equivalence} if $f^*:\Ci\left(\cN\right) \to \Ci\left(\M\right)$ is a quasi-isomorphism.
\end{itemize}

Given any category $\sC$ with a set of morphisms $\mathcal{W}$ called the \emph{weak equivalences,} one can formally homotopically invert the morphisms in $\mathcal{W}$ to produce an $\i$-category $\sC\left[\mathcal{W}^{-1}\right]_\i$ such that given any $\i$-category $\sD$ and any functor 
$$F:\sC \to \sD,$$
for which $F\left(f\right)$ is an equivalence in $\sD$ for all $f \in \mathcal{W},$ there is a (homotopically) unique extension $$\tilde F:\sC\left[\mathcal{W}^{-1}\right]_\i \to \sD.$$ One can model this as a simplicial category simply by Dwyer-Kan localization, and for $\sC$ a model category and $\mathcal{W}$ its class of weak equivalences, this is the usual construction. For a general relative category, one has little information on how to compute limits and colimits in $\sC\left[\mathcal{W}^{-1}\right]_\i,$ if they even exist at all. However, for a category of fibrant objects, $\sC\left[\mathcal{W}^{-1}\right]_\i$ always has finite limits, and homotopy pullbacks can be computed simply be replacing one of the maps with a fibration. Since we are concerned with taking derived fibered products, and we want this to be just ordinary pullbacks in the $\i$-category of derived manifolds, this is exactly the structure one needs for concrete calculations.

Using ideas from AKSZ theory and homological perturbation theory, the authors of \cite{dgder} established a category of fibrant objects structure on dg-manifolds nearly the same as the one we proposed. The only subtlety is their definition of weak equivalence is instead that $f$ induces a bijection
$$\pi_0\left(\M\right) \to \pi_0\left(\cN\right)$$
on the \emph{underlying sets} of classical loci, and also for each classical point $x \in \pi_0\left(\M\right),$ the induced map on tangent complexes
$T_x \M \to T_{f(x)} \cN$ is a quasi-isomorphism. In this paper, we rectify this by proving an inverse-function theorem for dg-manifolds; namely we show these two notions of weak equivalence agree.

The main result of this paper is that
\begin{theorem}
There is a canonical equivalence of $\i$-categories
$$\dgMan\left[\mathbf{qi}^{-1}\right]_\i \stackrel{\sim}{\longlongrightarrow} \DMfd$$
between the $\i$-category associated to the category of fibrant objects structure on dg-manifolds, and the $\i$-category of derived manifolds.
\end{theorem}

\begin{remark}
It should be noted that we are using cohomological grading conventions, and restricting to dg-manifolds whose algebra of functions are concentrated in non-positive degrees. Dg-manifolds whose algebras are concentrated in \emph{non-negative} degrees behave much differently, and are closely related to Lie algebroids and deformation theory. See for example the thesis of Joost Nuiten \cite{Nuiten}.
\end{remark}


\part{Background}\label{part:background}
\section{A Review of Differential Graded Geometry}\label{sec:dggeom}
In this section, we will give a mostly self-contained introduction to graded and dg-geometry. We find this necessary firstly, as this language is likely unfamiliar to many homotopy theorists. Secondly, we will switch back and forth between many different equivalent categories of dg-manifolds in our arguments in the rest of the paper, so we provide a detailed explanation about how these different points of view are related for the reader's convenience. The content is mostly standard. There are many additional sources. Here is an incomplete list: \cite{Sch93,AKSZdima,Cattaneo,Maxim,Teichner,Fairon2017,ref1}

\subsection{$\Z_2$-graded geometry}
\begin{definition}
A \textbf{$\Z_2$-graded vector space} (a.k.a super vector space) $\V$ over $\R$ is simply a pair $\left(V_0,V_1\right)$ of real vector spaces. Elements of $V_0$ are said to be \textbf{even} and elements of $V_1$ are said to be \textbf{odd}. For $v \in V_i,$ we write $|v|=i \in \Z_2.$ A morphism $f:\V \to \W$ is a pair of linear maps, $f_0:V_0 \to W_0$ and and $f_1:V_1 \to W_1.$ I.e., the category $\SVect$ is the category $\Fun\left(\Z_2,\Vect\right).$ For $V$ an ordinary (i.e. non-graded) vector space, we will identify it with the $\Z_2$-vector space $\left(V,0\right).$ A $\Z_2$-graded vector space of this form will be called \textbf{purely even}. We write the $\Z_2$-graded vector space $\left(\R^n,\R^m\right)$ as $\R^{n|m},$ with $n$ the \textbf{even dimension} and $m$ the \textbf{odd dimension}, and $\left(n|m\right)$ is the \textbf{super dimension}.

The tensor product is defined by
$$\left(\V \otimes \W\right)_0=\left(V_0 \otimes W_0\right)\oplus \left(V_1 \otimes W_1\right),$$
$$\left(\V \otimes \W\right)_1=\left(V_0 \otimes W_1\right)\oplus \left(V_1 \otimes W_0\right).$$ We regard $\SVect$ as the symmetric monoidal category $\left(\SVect,\otimes,\tau\right)$ where $\tau$ is the Koszul braiding, defined on homogeneous elements by
\begin{eqnarray*}
\tau:\V \otimes \W &\stackrel{\sim}{\longrightarrow}& \W \otimes \V\\
 v \otimes w &\mapsto& \left(-1\right)^{|v|\cdot |w|} w \otimes v.
 \end{eqnarray*}
\end{definition}

\begin{remark}
Alternatively, many authors define a $\Z_2$-graded vector space as a vector space $V$ together with a direct sum decomposition $$V=V_0 \oplus V_1,$$ and define morphisms to be those which preserve degree. This category is canonically isomorphic to $\SVect.$ The only difference to keep in mind is, in our approach, one cannot sum an element of $V_0$ with one of $V_1,$ i.e. all elements of a $\Z_2$-graded vector space in the definition above are homogeneous. Nonetheless, these two approaches are categorically equivalent.
\end{remark}

\begin{definition}
There is a canonical idempotent automorphism $\Pi$ of the category $\SVect,$ called the \textbf{parity reversal} functor defined by $$\Pi=\R^{0|1} \otimes \left(\blank\right).$$ A $\Z_2$-graded vector space of the form $\Pi V,$ with $V$ purely even is called \textbf{purely odd}.
\end{definition}

Notice that for any $\Z_2$-graded vector space $\V=\left(V_0,V_1\right),$ we have $\R^{0|1} \otimes \V=\left(V_1,V_0\right).$ In particular $\R^{0|1} \otimes \R^{0|1}=\R.$ It follows that $\Pi\left(\Pi\left(\V\right)\right)=\V,$ for all $\V.$

The $\left(\SVect,\otimes,\tau\right)$ is symmetric monoidal closed with
$$\underline{\Hom}\left(\V,\W\right)=\left(\Hom\left(\V,\W\right),\Hom\left(\V,\Pi \W\right)\right).$$ Accordingly, the \textbf{dual} of $\V$ is defined as
$$\V^*:=\underline{\Hom}\left(\V,\R\right).$$

\begin{definition}
The category $\SCom$ of \textbf{super-commutative $\R$-algebras} is the category of commutative monoid objects in the symmetric monoidal category $\SVect.$ Explicitly, these are $\Z_2$-graded $\R$-algebras $\A=\left(A_0,A_1\right)$ such that $$a \cdot b = \left(-1\right)^{|a||b|}b \cdot a.$$
\end{definition}

\begin{definition}
The parity reversal functor $\Pi$ induces an automorphism $\Pi_\A$ for every super-commutative $\R$-algebra defined by $$\Pi_\A\left(x\right)=\left(-1\right)^{|x|}\cdot x.$$ This is called the \textbf{parity involution} automorphism.
\end{definition}

There is a forgetful functor $U:\SCom \to \SVect,$ which has a left adjoint $\Sym,$ where $\Sym\left(\V\right)$ is the $\Z_2$-graded symmetric algebra. Notice that for any $\V,$ 
\begin{eqnarray*}
\V&=&\left(V_0,0\right) \oplus \left(0,V_1\right) \\
&=& V_0 \oplus \Pi V_1.
\end{eqnarray*}
Since $\Sym$ is a left adjoint, it is uniquely determined by the fact that for $V$ purely even, $\Sym\left(V\right)$ is the usual (ungraded) symmetric algebra concentrated in degree 0, and that $$\Sym\left(\Pi V\right)=\Lambda\left(V\right),$$ where $\Lambda\left(V\right)$ is the exterior algebra (with its $\Z$-grading reduced modulo $2$). That is to say, for a general $\V,$
$$\Sym\left(\V\right)=\Sym\left(V_0\right) \underset{\R} \otimes \Lambda\left(V_1\right).$$ (Just as in the non-graded commutative case, the tensor product becomes the coproduct for super-commutative algebras).

Suppose that we have a basis $\left(x_1,\ldots,x_n\right)$ for $V_0$ and $\left(y_1,\ldots,y_m\right)$ for $V_1.$ Then elements of $\Sym\left(\V\right)$ are polynomials in the $n$ commuting variables $x_1,\ldots, x_n$ and the $m$ anti-commuting variables $y_1,\ldots,y_m.$

Let $\V=\left(V_0,V_1\right)$ be a finite dimensional $\Z_2$-graded vector space (a.k.a a super vector space) over $\R.$ Associated to $\V$ there is a ringed space $\left(V_0,\O_\V\right),$ where the $\Z_2$-graded structure sheaf is defined as
\begin{eqnarray*}
\O_\V\left(U\right)&=&\Ci\left(U\right) \underset{\Sym\left(\left(V_0\right)^*\right)} \otimes \Sym\left(\V^*\right)\\
&\cong& \Ci\left(U\right) \underset{\R} \otimes \Lambda\left(\left(V_1\right)^*\right).
\end{eqnarray*}

\begin{definition}
A \textbf{superdomain} is pair $\left(X,\O_X\right)$ of a topological space and a sheaf $\O_X$ of super-commutative $\R$-algebras which is isomorphic as a ringed space to $\left(V_0,\O_\V\right)$ for a finite dimensional $\Z_2$-graded vector space.
\end{definition}

\begin{definition}
A \textbf{supermanifold} $\sM=\left(M,\O_{\sM}\right)$ is a second countable Hausdorff space $M$ with a sheaf of super-commutative $\R$-algebras $\O_{\sM}$ for which there exists a cover $\left(U_\alpha \hookrightarrow M\right)_\alpha,$ such that for all $\alpha,$ $\left(U_\alpha,\O_{\sM}|_{U_\alpha}\right)$ is a superdomain. We denote the super-commutative $\R$-algebra $\Gamma_{M}\left(\O_{\sM}\right)$ by $\Ci\left(\sM\right).$
\end{definition}

\begin{definition}
For $\V$ a finite dimensional $\Z_2$-graded vector space, we write $\V=\left(V_0,\O_{\V}\right),$ meaning we regard $\left(V_0,\O_{\V}\right)$ as the geometric incarnation of $\V$ as a supermanifold, just as we can write $\R^n$ both for the $n$-dimensional vector space, and the smooth manifold. This defines a natural functor $\SVect \to \SMfd.$ Morphisms in the image of this functor are called \textbf{linear}.
\end{definition}

\begin{remark}
Any supermanifold $\sM$ is locally diffeomorphic to $\R^{p|q}$ for some non-negative $p$ and $q.$
\end{remark}

\begin{example}
Any smooth manifold $M$ is a super manifold since $\R^n$ is a purely even $\Z_2$-graded vector space, and $\left(\R^n,\O_{\R^n}\right)$ is just $\R^n$ with its standard structure sheaf of smooth functions. A supermanifold of this form is called \textbf{purely even}.
\end{example}

\begin{definition}
For $\sM=\left(M,\O_{\sM}\right)$ a supermanifold, let $\mathcal{I}_\sM$ be the (homogeneous) ideal sheaf locally generated by odd-elements. Then $|\sM|:=\left(M,\left(\O_{\sM}\right)/\mathcal{I}_{\sM}\right)$ is a smooth manifold (regarded as a locally ringed space.) It is called the \textbf{core} of $\sM.$ The functor $|\blank|:\SMfd \to \Mfd$ is left adjoint to the inclusion of manifolds into supermanifolds.
\end{definition}

\begin{remark}
Just as $\O_M$ equips $M$ with local coordinate systems, $\O_{\sM}$ equips $\sM$ with local coordinate systems of the form $\left(x_1,\ldots,x_n;\xi_1,\ldots,\xi_m\right),$ where $\xi_1,\ldots,\xi_m$ are a basis for the dual of $V_1.$ The only caveat is that the variables $\xi_i$ are odd and anti-commute. These coordinates are called \textbf{fermionic}, and the coordinates $\left(x_1,\ldots,x_n\right)$ of $M$ are called \textbf{bosonic.} This is justified by the well known fact that morphisms of supermanifolds can be expressed locally in these coordinates. More precisely:

A smooth function $f:\sM \to \R,$ i.e. a morphism of ringed spaces $\left(M,\O_{\sM}\right) \to \left(\R,\Ci_\R\right)$ can be identified with a choice of an even element $f$ in $\Ci\left(\sM\right),$ and a smooth function $g:M \to \R^{0|1}$ with an odd element. Putting this together, it follows that a smooth function $f:\R^{n|m} \to \R^{l|k}$ is the same data as $l$ even elements $f^1,\ldots, f^l$ of $$\Ci\left(\R^{n|m}\right)=\Ci\left(\R^n\right) \underset{\R} \otimes \Lambda\left(\left(\R^m\right)^*\right)$$ and $k$ odd elements $f^{l+1},\ldots f^{l+k}.$ Let $x_1,\ldots,x_n$ be the standard coordinates on $\R^n$ and $y_1,\ldots,y_l$ the standard coordinates on $\R^l.$ Let $\xi_1,\ldots,\xi_m$ be a dual basis of $\R^m$ and $\eta_1,\ldots,\eta_k$ a dual basis for $\R^k.$ An even element of $\Ci\left(\R^{n|m}\right)$ is of the form
$$h_0\left(x_1,\ldots,x_n\right) + h_{12}\left(x_1,\ldots,x_n\right)\xi_1\xi_2+\ldots=\sum_{I \subset \{1,.\dots,m\}, |I| \in 2 \Z} h_I\left(x_1,\ldots,x_n\right)\xi^I,$$ with each $h_i$ a smooth function, and an odd element is of the same form, but with the sum over all subsets $I$ of odd cardinality. (Here we are regarding $I$ as a multi-index). Thus, using Einstein summation conventions, one gets expressions of the form
\begin{eqnarray*}
y_1\left(x_1,\ldots,x_n,\xi_1,\ldots,\xi_m\right)&=&f^1_{I}\left(x_1,\ldots,x_n\right)\xi^I\\
y_2\left(x_1,\ldots,x_n,\xi_1,\ldots,\xi_m\right)&=&f^2_{I}\left(x_1,\ldots,x_n\right)\xi^I\\
\ldots &&\\
y_l\left(x_1,\ldots,x_n,\xi_1,\ldots,\xi_m\right)&=& f^l_{I}\left(x_1,\ldots,x_n\right) \xi^I\\
\eta_1\left(x_1,\ldots,x_n,\xi_1,\ldots,\xi_m\right)&=& f^{l+1}_I\left(x_1,\ldots,x_n\right)\xi^I\\
\ldots &&\\
\eta_k\left(x_1,\ldots,x_n,\xi_1,\ldots,\xi_m\right) &=& f^{l+k}_I\left(x_1,\ldots,x_n\right) \xi^I.
\end{eqnarray*}
Most of differential geometry carries over with little modification using these local coordinates (with careful attention to signs!), with the remarkable exception of integration theory.
\end{remark}

It is well-known that supermanifolds are affine, in the sense that the functor
$$\SMfd \to \SCom^{op}$$ sending $\sM$ to $\Ci\left(\sM\right)$ is fully faithful. This follows readily from the existence of partitions of unity.

Vector fields on supermanifolds can be defined via derivations.

\begin{definition}
An \textbf{even vector field} on a supermanifold $\sM$ is a derivation $X$ of $\O_{\sM}$ with values in the $\O_{\sM}$-module $\O_{\sM},$ and we write $|X|=0.$ An \textbf{odd vector field} is a derivation $X$ of $\O_{\sM}$ with values in the $\O_{\sM}$-module $\Pi \O_{\sM},$ and we write $|X|=1.$
\end{definition}

Unwinding this, for an odd vector field $X$ and $f$ an element of $\Ci\left(\sM\right),$ $X\left(f\right) \in \Ci\left(\sM\right)$ is an element of the opposite parity of $f,$ i.e. if $f$ is even, $X\left(f\right)$ is odd etc. Moreover, one has
$$X\left(fg\right)=X\left(f\right)g+\left(-1\right)^{|X|}fX\left(g\right).$$

A vector field on $\R^{n|m}$ takes the form
$$f^i \frac{\partial}{\partial x_i} + g^j \frac{\partial}{\partial \xi_j},$$ where $\frac{\partial}{\partial \xi_j}$ is the \emph{left partial derivative}, e.g. 
\begin{eqnarray*}
\frac{\partial}{\partial \xi_1}\left(\xi_2\xi_1\right) &=& -\xi_2\frac{\partial}{\partial \xi_1}\xi_1\\
&=& -\xi_2.
\end{eqnarray*}
If the vector field is even, then the $f^i$ are even functions, and the $g^j$ odd, and vice versa if the vector field is odd. Vector fields on $\sM$ form a $\Z_2$-graded Lie algebra, (i.e. a Lie algebra object in $\SVect$), with Lie bracket defined as
$\left[X,Y\right]=XY-\left(-1\right)^{|X||Y|}YX.$

Vector bundles for supermanifolds are defined in the obvious way, namely, as a fiber bundle with typical fiber $\V$ for a $\Z_2$-graded vector space, whose transition functions are linear. Just as any natural (multi-)functor on finite dimensional real vector spaces extends to vector bundles (e.g. direct sum, tensor product etc.), the parity inversion endofunctor on $\SVect$ extends to vector bundles on supermanifolds, that is, given a vector bundle $E \to \sM$ with typical fiber $\V,$ one can construct in a canonical way a vector bundle $\Pi E \to \sM$ with typical fiber $\Pi \V.$ 

To talk about sections of a vector bundle $\mathcal{V} \to \sM$ over a supermanifold, one must take some care. For example, one would want that for a $\Z_2$-graded vector space $\V,$ sections of the trivial bundle $\sM \to \V$ should be $\V.$ But this doesn't make any sense, since $\V$ is a $\Z_2$-graded vector space, not a set. This can be rectified with the following definition:

\begin{definition}
Given a vector bundle $\mathcal{V} \to \sM$ over a supermanifold, the $\Z_2$-graded module $\underline{\Gamma}_{\sM}\left(\mathcal{V}\right)$ is defined to be the $\Ci\left(\sM\right)$-module in $\Z_2$-graded vector spaces with even component
$$\underline{\Gamma}_{\sM}\left(\mathcal{V}\right)_0=\Gamma_{\sM}\left(\mathcal{V}\right)$$
and odd component
$$\underline{\Gamma}_{\sM}\left(\mathcal{V}\right)_1=\Gamma_{\sM}\left(\Pi\mathcal{V}\right).$$
\end{definition}

The $\Z_2$-graded module of sections of a vector bundle form a locally free sheaf of $\O_\sM$-modules, and conversely, any locally free sheaf of $\O_\sM$-modules arises in this way. For example, one defines the tangent bundle $T \mathcal{M}$ of a supermanifold in the obvious way, and it corresponds to the locally free sheaf whose even component is even vector fields, and odd component is odd vector fields. That is to say, an odd vector field is actually a section of $\Pi T \mathcal{M},$ not $T\M.$

The construction of parity reversing the fibers of a vector bundle can be carried out in the special case where $\sM$ and $\V$ are purely even, i.e. when $E \to M$ is a ordinary vector bundle of smooth manifolds, in the classical sense. If $M$ is $n$-dimensional and $E$ is rank $k,$ then $\Pi E$ is locally isomorphic to $\R^{n|k}.$ Explicitly:
$$\Pi E=\left(M,\O_{\Pi_E}\right)$$ with
$$\O_{\Pi_E}=\Gamma\left(\Lambda E\right),$$ where for sections $\lambda$ and $\kappa$ we define $\lambda \cdot \kappa:=\lambda \wedge \kappa.$ A supermanifold of the form $\Pi E$ for $E$ a classical vector bundle is called \textbf{split}.

There is the following famous result of Batchelor:

\begin{theorem}[Batchelor's Theorem]\cite{batch}
Every smooth supermanifold (non-canonically) splits.
\end{theorem}

This splitting is non-canonical. However, there is a \emph{canonical} splitting as a fiber bundle with purely odd fiber, as we will explain. Let $\sM$ be supermanifold and let $U \subseteq |\sM|$ be an open submanifold of its core. Then there is a pullback diagram of the form
$$\xymatrix{\left(U,\O_{\sM}|_U\right) \ar[r] \ar[d]& \sM \ar[d]\\
U \ar[r] & |\sM|.}$$
By definition of a supermanifold, there exists an open cover $\left(U_\alpha \hookrightarrow |\sM|\right)_\alpha$ of $|\sM|$ by open submanifolds and isomorphisms of supermanifolds 
$$\varphi_\alpha:\left(U_\alpha,\O_{\sM}|_U\right) \stackrel{\sim}{\longrightarrow} \V,$$ for some $\Z_2$-graded vector space $\V.$ Since $\V=\left(V_0,\O_{\V}\right),$ this means that $|\varphi_\alpha|:U_\alpha \stackrel{\sim}{\longrightarrow} V_0$ is a diffeomorphism. Moreover, notice that $\V\cong V_0 \oplus \Pi V_1,$ from which it follows that $\V \cong V_0 \times \Pi V_1$ as a supermanifold. The canonical projection $\V=V_0\times \Pi V_1 \to V_0$ is the unit $\pi_\V$ exhibiting $\Mfd$ as a reflective subcategory, hence, by naturality, the following diagram commutes:
$$\xymatrix{\left(U_\alpha,\O_{\sM}|_U\right) \ar[d]_-{\pi_\sM} \ar[r]^-{\sim} & U_\alpha \times \Pi V_1. \ar[ld]^{pr}\\
U_\alpha &}$$
In summary, we have proven the following:
\begin{proposition}\label{eq:fiberbun}
For every supermanifold $\sM$ locally modeled on a $\Z_2$-graded vector space $\V,$ the reflector $\pi_\sM:\sM \to |\sM|$ exhibits $\sM$ as a fiber bundle over $|\sM|$ with fiber $\Pi V_1.$
\end{proposition}

In light of this, Batchelor's theorem can be rephrased as the statement that the structure group of the fiber bundle $\pi_\sM:\sM \to |\sM|$ can be (non-canonically) reduced to $\mathbf{GL}\left(\Pi V_1\right).$

\subsection{$\N$-graded geometry}
\begin{definition}
The category of \textbf{$\Z$-graded vector spaces} over $\R$ is $\Fun\left(\Z,\Vect\right),$ hence a $\Z$-graded vector space $\V$ is a collection $\left(V_n\right)_{n \in \Z},$ of vector spaces. We will sometimes also write $\V$ as $V_\bullet.$ We say that an element $v$ of $V_n$ is \textbf{even} (respectively \textbf{odd}) if $n$ is, and we write $|v|=n \in \Z.$ We define tensor products in the usual way. To boot
$$\left(\V \otimes \W\right)_n= \underset{l+k=n} \bigoplus V_l \otimes W_k.$$ The Koszul braiding is defined analogously to the $\Z_2$-graded case, producing a symmetric monoidal category $\ZVect.$
\end{definition}

\begin{definition}
Let $n \in \Z.$ Denote by $\R\left[n\right]$ the $\Z$-graded vector space with $\R$ in degree \textbf{-}$n$ and zero elsewhere. We define the \textbf{suspension} $$\Sigma:=\R\left[1\right] \otimes \left(\blank\right)$$ and \textbf{loop} $$\Omega:=\R\left[-1\right] \otimes \left(\blank\right)\Sigma^{-1}$$ functors. For $n \in \Z,$ and $V_\bullet$ a $\Z$-graded vector space, we define $$V\left[n\right]:=\Sigma^n\left(\V\right),$$ Concretely, we have $$\V\left[n\right]_i=V_{n+i}.$$ Thus $\V\left[n\right]$ is $\V$ shifted to the left by $n$-degrees.
\end{definition}

$\ZVect$ is symmetric monoidal closed with
$$\underline{\Hom}\left(\V,\W\right)_n=\Hom\left(\V,\W\left[n\right]\right).$$
We define the \textbf{dual} of a $\Z$-graded vector space $\V$ as
$$\V^*:=\underline{\Hom}\left(\V,\R\right).$$ Concretely, one has the formula
$$\V^*_n=\left(V_{-n}\right)^*.$$

There is a canonical functor $$\blank/2:\ZVect \to \SVect$$ which sends $\V$ to $\V/2:=\left(\bigoplus \limits_{n\in \Z} V_{2n},\bigoplus \limits_{n\in \Z} V_{2n+1}\right),$ i.e. it reduces the grading modulo $2.$

\begin{definition}
The category of \textbf{$\Z$-graded commutative algebras} over $\R$ is the category of commutative monoids in the symmetric monoidal category $\ZVect.$ They are a $\Z$-graded algebras with the property that $$a\cdot b=\left(-1\right)^{|a||b|}b \cdot a.$$ We define the category thereof as $\gCom$. There is an obvious forgetful functor to $\Z_2$-commutative algebras by reducing the $\Z$-grading modulo $2.$
\end{definition}

\begin{remark}
A $\Z$-graded commutative algebra $\A$ together with a derivation $d$ of $\A$ of degree $1$ (i.e. a derivation with values in $\A\left[1\right]$) whose graded commutator $\left[d,d\right]=2d^2=0,$ is the same as a commutative differential graded algebra over $\R,$ where $d$ is the differential. Here we are using \emph{cohomological grading conventions}.
\end{remark}

There is also a forgetful functor $\gCom \to \ZVect,$ which admits a left adjoint which we abusively denote again by $\Sym.$ It is the usual construction of the \textbf{graded symmetric algebra} of a graded vector space. In particular, for $V$ a vector space in degree $0,$ $$\Sym\left(V\left[-1\right]\right)=\Lambda\left(V\right),$$ where now $\Lambda\left(V\right)$ is equipped with its standard $\N$-grading.

We are now about to define what a \emph{graded manifold} is. However, we will restrict to the case of those graded manifolds modeled on $\Z$-graded vector spaces that are concentrated in non-negative degrees. The approach works equally well for $\Z$-graded manifolds in non-positive degrees, but is ill-suited for the study (or even definition) of general $\Z$-graded manifolds.

\begin{definition}
Let $\V=\left(V_0,V_1,\ldots\right)$ be an $\N$-graded vector space (regarded as a $\Z$-graded one) over $\R$ whose total dimension is finite, i.e. $\bigoplus \limits_{n=0}^{\i} V_n$ is finite dimensional. We define $\left(V_0,\O_{\V}\right)$ to be the ringed space whose structure sheaf assigns an open set $U$ of $V_0$ the $\Z$-graded commutative algebra
$$\O_{\V}\left(U\right)=\Ci\left(U\right) \underset{\Sym\left(\left(V_0\right)^*\right)} \otimes \Sym\left(\V^*\right).$$ (Notice that since $\V$ is concentrated in non-negative degrees, $\V^*$ is concentrated in non-positive degrees, as is $\O_{\V}.$)

A \textbf{graded domain} is a topological space $X$ together with a sheaf $\O_X$ of $\Z$-graded commutative $\R$-algebras (concentrated in non-positive degrees), such that is isomorphic as a $\Z$-graded commutative ringed space to $\left(V_0,\O_{\V}\right)$ for $\V=\left(V_0,V_1,\ldots\right)$ an $\N$-graded vector space of total finite dimension.
\end{definition}

\begin{definition}
A \textbf{graded manifold} $\M=\left(M,\O_{\M}\right)$ is a second countable Hausdorff space $M$ with a sheaf of $\Z$-graded commutative $\R$-algebras $\O_{\M},$ which is locally isomorphic to a graded domain. We write $\V=\left(V_0,\O_{\V}\right),$ meaning we regard $\left(V_0,\O_{\V}\right)$ as the geometric incarnation of $\V$ as a graded manifold. We denote the $\Z$-graded commutative $\R$-algebra $\Gamma_{M}\left(\O_{\M}\right)$ by $\Ci\left(\M\right).$ Smooth maps of graded manifolds are morphisms of ringed spaces.\footnote{There is no need to worry about local rings, as the points of a graded manifold correspond to maximal ideals with residue field $\R.$}

Let $\I_0$ be the ideal sheaf locally generated by elements of $\O_\M$ of negative degree. Then $|\M|:=\left(M,\O_\M/\I_0\right)$ is a smooth manifold, called the \textbf{core} of $\M.$ Notice that there is a canonical smooth map $\M \to |\M|.$
\end{definition}

\begin{remark}
Notice that there is a canonical smooth map $\M \to |\M|.$ In fact, this is the unit of an adjunction exhibiting the category $\Mfd$ as a reflective subcategory of graded manifolds, with the core-functor $|\blank|$ as a left adjoint.

In particular, any morphism $f:\M \to \cN$ of graded manifolds fits into a commutative square
$$\xymatrix{\M \ar[r]^-{f} \ar[d]_-{\pi_\M} & \cN \ar[d]^-{\pi_\cN}\\
|\M| \ar[r]_-{|f|} & |\cN|.}$$
\end{remark}

\begin{remark}
Another, perhaps simpler description of the core $|\M|$ of a graded manifold is that it is $\left(M,\left(\O_{\M}\right)_0\right).$ This is because no elements of negative degree can create any new elements of degree $0.$ This phenomenon breaks down when one considers unbounded $\Z$-graded manifolds. Notice that this is also in contrast to the $\Z_2$-graded case of supermanifolds. For example, if $\left(\eta_1,\eta_2\right)$ are local coordinates on $\R^{0|2},$ $f\left(\eta_1,\eta_2\right)=\eta_1\eta_2$ is a non-trivial even smooth function, but the core of $\R^{0|1}$ is just a single point $\R^0.$
\end{remark}

Notice that, in particular, for an $\N$-graded vector space $\V,$ $\left(V_0,\O_{\V}\right)$ is a graded manifold. Hence, there is a canonical functor
$\mathsf{Vect}_{\N} \to \Mfd_{\N}.$ Morphisms in the image of this functor are called \textbf{linear}. From here, it is straightforward to (geometrically) define vector bundles over a graded manifold $\M$, and their category is equivalent to that of locally free sheaves of $\O_{\M}$-modules concentrated in non-positive degrees. Just as the parity reversal functor $\Pi$ in the $\Z_2$-graded setting extends fiberwise to vector bundles over supermanifolds, the loop functor $\Omega=\left(\blank\right)\left[-1\right]$ extends to vector bundles over graded manifolds. \footnote{If we took the time to define genuine $\Z$-graded manifolds, we could also extend the suspension functor.} If $\M$ is an ordinary manifold $M,$ then there is a natural one-to-one correspondence between vector bundles over $M$ in the category $\Mfd_{\N}$ of graded manifolds, and \emph{graded vector bundles} over $M,$ that is, a sequence $V_0,V_2,V_3,\ldots$ of (classical) vector bundles over $M.$ To such a sequence we can associate the vector bundle
$$\underset{n \in \N}\bigoplus V_n\left[-n\right]$$ over $M.$ (This places $V_n$ in degree $n$.)

There is an analogue of Batchelor's theorem in the graded-context:

\begin{theorem}\label{thm:BatchN}\cite[Theorem 2.1]{ZBatch}
Every graded manifold $\M$ is (non-canonically) isomorphic to the total space of a vector bundle $\V$ over its core $|\M|,$ with $\V_0=0.$
\end{theorem}

\begin{remark}
The above theorem can be reformulated to be in closer analogy with the case of supermanifolds. It is equivalent to the statement that every graded manifold $\M$ is (non-canonically) isomorphic to the total space $\V\left[-1\right],$ where $\V$ arises from graded vector bundle $V_\bullet$ over $|\M|.$ Notice that the shift places $V_0$ in degree $1,$ so $\V\left[-1\right]_0=0.$
\end{remark}

In complete analogy with the proof of Proposition \ref{eq:fiberbun}, we have the following:

\begin{proposition}
Let $\M$ be a graded manifold locally modeled on an $\N$-graded vector space $\V=\left(V_0,V_1,\ldots\right).$ Then the reflector $\pi_\M:\M \to |\M|$ exhibits $\M$ as a fiber bundle over $|\M|$ with fiber $\V'=\left(0,V_1,V_2,\ldots\right).$
\end{proposition}


Suppose that $\V$ and $\W$ are $\N$-graded (total) finite dimensional vector space over $\R.$  Let's get a handle on what smooth maps $f:\V \to \W$ of graded manifolds are. We need the data of a smooth map $f_0:V_0 \to W_0$ of ordinary manifolds, and a morphism $\alpha:f_0^*\O_{\W} \to \O_{\V}$ of sheaves. Notice that any morphism $f:\V \to \W$ makes the following diagram commute
$$\xymatrix{\V \ar@{-->}[rr]^-{f} \ar[rd]_-{f_0 \circ \pi} & &\W. \ar[ld]\\
& W_0}$$ Let $\V':=\left(0,V_1,V_2,...\right)$ and define $\W'$ analogously.
It follows that $\alpha$ is the same as for each open subset $U$ of $W_0,$ having a commutative diagram
natural in $U.$ Such a commuting triangle is the same as a homomorphism $$\Sym\left(\left(\W'\right)^*\right) \to \left(f_0\right)_*\O_\V\left(U\right).$$ Unwinding the definitions, the need for naturality implies $\alpha$ is the same data as a morphism of sheaves or $\R$-algebras $\Delta\left(\Sym\left(\W'\right)^*\right) \to \left(f_0\right)_*\O_\V,$ where $\Delta\left(\Sym\left(\W'\right)^*\right)$ is the constant sheaf. By adjunction
and the fact that pushforward commutes with global sections, this is the same as a homomorphism of $\R$-algebras of the form
$$\Sym\left(\left(\W'\right)\right) \to \Gamma\left(\O_\V\right).$$ Since $\Sym$ is left adjoint, this is finally the same as a linear map of $\Z$-graded vector spaces $$\W' \to \Ci\left(\V\right)=\Ci\left(V_0\right)\underset{\R}\otimes \Sym\left(\left(\V'\right)^*\right).$$ Finally, if $\dim\left(W_k\right)=n_k,$ this is in bijection with the choice for each $k$ of $n_k$ elements of $\Ci\left(V_0\right)\underset{\R}\otimes \Sym\left(\left(\V'\right)^*\right)$ of degree $-k.$ Choosing a basis for each $V_i,$ such an element of degree $-k$ is a homogeneous polynomial expression in the dual basis vectors with coefficients in $\Ci\left(V_0\right).$

In analogy with the case of supermanifolds, given a vector bundle $\mathcal{V} \to \M$ over a graded manifold, we can associate a graded $\O_{\M}$-module of sections:

\begin{definition}
The graded $\O_\M$-module of sections $\mathcal{V}$ is defined as
$$\underline{\Gamma}_\M\left(\mathcal{V}\right)_n=\Gamma_\M\left(\mathcal{V}\left[n\right]\right).$$
\end{definition}


\begin{definition}
A \textbf{vector field of degree $n$} on a graded manifold $\M$ is a derivation $D:\O_{\M} \to \O_{\M}\left[n\right].$ Regarding $n$ as a grading, this is locally free sheaf of graded $\O_{\M}$-modules, which we denote by $\mathcal{T}\M$--- the \textbf{tangent sheaf}. Of course, this is the module of sections of the tangent bundle $T \M \to M.$ 
\end{definition}

\begin{remark}
The (graded) commutator Lie bracket is defined in the obvious way, with the same sign rule as in the $\Z_2$-graded case. This makes the tangent sheaf a sheaf of graded Lie algebras.
\end{remark}

For a vector field $X$ of even degree, $\left[X,X\right]=X\circ X - X \circ X=0,$ i.e. a $1$-dimensional even distribution is automatically integrable. However, for a vector field of \emph{odd} degree, this is no-longer the case, and integrability is a condition. 

\begin{definition}
A vector field of degree $+1$ whose graded commutator $\left[D,D\right]=2D^2=0$ is called a \textbf{cohomological vector field}.
\end{definition}

\begin{definition}
A \textbf{differential graded manifold} or \textbf{dg-manifold} is a pair $\left(\M,D\right)$ of a graded manifold $\M$ and a cohomological vector field $D.$ A morphism $\varphi:\left(\M,D\right) \to \left(\M',D'\right)$ of dg-manifolds is a morphism $\varphi:\M \to \M'$ of graded manifolds with the property that $D'$ is $\varphi$-related to $D,$ that is, for any $f$ in $\Ci\left(\M'\right),$ 
\begin{equation}\label{eq:related}
\varphi^*D'f=D\varphi^*f.
\end{equation}
We denote the resulting category as $\dgMan.$ (It is more precisely the category of dg-manifolds which are concentrated in non-positive degrees.)
\end{definition}

\begin{remark}
A cohomological vector field $D$ on $\M$ makes $\left(\O_\M,D\right)$ into a sheaf of commutative differential graded algebras over $\R.$ Indeed, $D$ becomes a differential turning $\left(\O_\M^\bullet,D\right)$ into a cochain complex, since $D$ has degree $1$ and $D^2=0,$ and since it is a derivation, the algebra structure is compatible with this differential. The above condition
(\ref{eq:related}) is the same as asking $f^*$ to be a map of cdgas. In analogy with the case of supermanifolds, the functor
$$\Ci:\dgMan \to \dga^{op}$$ is fully faithful.
\end{remark}

Given a dg-manifold $\left(\O_\M,D\right),$ taking the Lie bracket with the cohomological vector field $D$ defines a differential $\mathcal{L}_D:=\left[D,\blank\right]$ of degree $1$ on $\mathcal{T} \M.$ With this differential, $\mathcal{T} \M$ becomes a dg-module for $\left(\O_\M^\bullet,D\right).$ The differential corresponds to the canonical lift of $D$ to $T\left[n\right] \mathcal{M},$ for any $n,$ i.e. the tangent lift of $D.$

Any graded manifold $\M$ may be regarded as a dg-manifold with $D=0.$ In particular, every smooth manifold $M$ is a dg-manifold. Given a dg-manifold $\left(\M,D\right)$ the map $\pi_\M:\M \to |\M|$ is a morphism of dg-manifolds. At the level of algebras, this is just the morphism of differential graded algebras $$\Ci\left(|\M|\right)=\Ci\left(\M\right)_0 \to \left(\Ci\left(\M\right)_\bullet,D\right).$$ which is a chain map since $\Ci\left(\M\right)_1=0.$ However, this map is not a fiber bundle, unless the cohomological vector field $D$ depends only on the fiber coordinates, i.e. is fiber-wise constant.

\subsection{Dg-manifolds as equivariant supermanifolds}\label{sec:dgsup}
\subsubsection{Graded manifolds as supermanifolds}
\begin{definition}
Let $S$ be a finite set equipped with a map $\omega:S \to \N,$ called its weight function. Let $S_0$ and $S_1$ be the pre-images of $0$ and $1$ respectively of the composition $$S \stackrel{\omega}{\longrightarrow} \N \stackrel{\mod 2}{\longlongrightarrow} \Z_2.$$ The \textbf{standard $\N$-graded superdomain of \textbf{weight} $\omega$} is the supermanifold $\left(\R^{S_0|S_1},E_\omega\right),$ where $\R^{S_0|S_1}$ is $\R^{n|m},$ with $n=|S_0|,$ and $m=|S_1|,$ equipped with local coordinates $\left(x_i;\eta_j\right)_{i \in S_0,j \in S_1},$ and $E$ is the even vector field $$\omega\left(i\right)x_i \frac{\partial}{\partial x_i} + \omega\left(j\right)\eta_j\frac{\partial}{\partial \eta^j},$$ called the \textbf{Euler vector field}.

We will call a standard graded superdomain \textbf{trivial} if its weight function is constant with value $0.$

An \textbf{$\N$-graded superdomain} is a pair $\left(\sM,E\right)$ of a supermanifold $\sM$ with an even vector field $E,$ for which there exists a diffeomorphism $\varphi:\M \to \R^{p|q}$ such that $\varphi_*E=E_{\omega}$ for some $\N$-weight $\omega.$ On one hand, $\left(\M,E\right)$ will be called \textbf{trivializable} is $E=0,$ and on the other hand will be called \textbf{non-degenerate} if $\ker\left(E\right)=0.$

A morphism of $\N$-graded supermanifolds $f:\left(\sM,E\right) \to \left(\sM',E'\right)$ is a smooth map $\sM \to \sM'$ such that $E'$ is $f$-related to $E.$

Let $f$ be a function on a graded superdomain $\left(\sM,E\right)$ and let $n \in \N.$ We say that $f$ is \textbf{homogeneous of degree $n$} if $E\left(f\right)=nf.$ 
\end{definition}

Notice that every trivial graded superdomain is diffeomorphic to $\R^n$ for some $n,$ since odd elements can not have weight $0.$ \footnote{This could be relaxed, and one would have a theory of $\N$-graded supermanifolds.}
Moreover, every graded superdomain is isomorphic to a product of a trivializable graded superdomain and a non-degenerate superdomain:
\begin{equation}\label{ref:dec}
\sM\cong M_0 \times \sM'
\end{equation}
Moreover, the Euler field is a vertical vector field with respect to the projection $$\pi:\sM \to M_0.$$ 

\begin{remark}
$M_0$ is generally \emph{not} the core of $\sM.$ This is the case if and only if there are no coordinates of positive even weights.
\end{remark}

The decomposition (\ref{ref:dec}) should be viewed as exhibiting $\sM$ as a (trivial) bundle of non-degenerate graded superdomains over the ordinary manifold field $M_0.$

Moreover, $$\pi^*:\Ci\left(M_0\right) \to \Ci\left(\sM\right)$$ is a monomorphism identifying $\Ci\left(M_0\right)$ with those functions $f:\sM \to \R$ for which $E\left(f\right)=0.$ It follows that given any map $$\varphi:\left(\sM,E_\sM\right) \to \left(\sN,E_\sN\right)$$ must restrict to a smooth map $M_0 \to N_0$ of ordinary manifolds. 

For example, if for $\left(x_1,x_2,\ldots x_n,y_1,y_2,\ldots y_k;\eta_1,\eta_2,\ldots \eta_l\right)$ coordinates on $\R^{n+k|l},$ with the $x_i$ of weight $0$, and the rest of the coordinate with positive weight. This is the same as the product of $\R^n$ (with its trivial Euler field) and $\R^{k|l}$ with $pr_*E$. The coordinates $\left(y_1,y_2,\ldots y_k;\eta_1,\eta_2,\ldots \eta_l\right)$ are called the \textbf{fiber coordinates}.

\begin{lemma}\label{ref:homog}\cite[Lemma 5.6]{monoids}
A function $f$ on a standard graded superdomain $\left(\R^{S_0|S_1},E_\omega\right)$ is $n$-homogeneous if and only if is a homogeneous polynomial in the fiber-coordinates with coefficients in the bases coordinates, and the total weight of the polynomial is $n$.
\end{lemma}

It follows that any morphism $f:\left(\sM,E_\sM\right) \to \left(\sN,E_\sN\right)$ of graded superdomains is a morphism of bundles of non-degenerate graded superdomains over the map $f_0=f|_{M_0}$
$$\xymatrix{\sM \ar[r]^-{f} \ar[d]_-{\pi_\sM} & \sN \ar[d]^-{\pi_\sN}\\
M_0 \ar[r]_-{f_0} & N_0.}$$ This is equivalent data to a morphism of bundles $\varphi:\sM \to f_0^*\sN$ over non-degenerate graded superdomains over $M_0.$ Choose fiber coordinates $\left(x_1,\ldots,x_n;\eta_1,\ldots,\eta_m\right)$ and $\left(y_1,\ldots,y_k; \xi_1,\ldots,\xi_k\right)$ of $\sM$ and $\sN$ respectively. Then $\varphi$ is uniquely determined in local coordinates, and each each coordinate $y_i$ must be a function on $\sM$ for which $E_\sM\left(y_i\right)=\omega\left(i\right)y_i,$ similarly for each $\xi_j.$ It follows from Lemma \ref{ref:homog} that this is if and only if each $y_i$ is a homogeneous polynomial of degree $\omega\left(i\right)$ in the fibered coordinates $\left(x_1,\ldots,x_n;\eta_1,\ldots,\eta_m\right)$ with coefficients in $\Ci\left(M_0\right).$ This the same description of a map of graded manifolds between $\Z$-graded vector spaces. That is to say there is a canonical fully faithful functor 
$$F:\mathsf{SDom_{\N}} \to \Mfd_\N$$
from the category of standard graded superdomains whose essential image is the subcategory $\mathsf{Cart}\Mfd_\N$ of graded manifolds spanned by $\N$-graded vector spaces. The functor $F$ is defined on object by 
$$F\left(\left(\R^{S_0|S_1},E_\omega\right)\right)=\underset{n \in \N} \bigoplus \left[\left(\R\left[-2n\right]\right)^{\omega^{-1}\left(2n\right)} \oplus \left(\R^{0|1}\left[-2n-1\right]\right)^{\omega^{-1}\left(2n+1\right)}\right].$$

\begin{remark}
There is some fiddling with signs going on here, in order to produce a graded manifolds with coordinates in \emph{non-positive} rather than non-negative degree. There is due to our unwillingness to use homological grading conventions.
\end{remark}

\begin{definition}
Let $M$ be an ordinary manifold. A \textbf{bundle of graded superdomains} over $M$ with fiber $\left(\R^{S_0|S_1},E_\omega\right)$ is a fiber bundle with fiber $\R^{S_0|S_1},$ whose transition functions are maps of graded superdomains. It is called a bundle of non-degenerate graded superdomains, if the typical fiber is non-degenerate.
\end{definition}

\begin{remark}
Notice that such a bundle of non-degenerate graded superdomains $\sM \to M$ carries a canonical vector field $E$ glued together over its fibers, called its \textbf{Euler vector field}. If the bundle is non-degenerate, $M$ can be identified with the zero locus of $E.$
\end{remark}
 
The following theorem is folklore, but we include a proof for convenience:
\begin{theorem}
There is an equivalence of categories 
$$\mathsf{GrSMfdBun}_0 \stackrel{\sim}{\longrightarrow}\Mfd_\N$$
between the category of bundles of non-degenerate graded superdomains over smooth manifolds, and the category of graded manifolds.
\end{theorem}

\begin{proof}
This follows formally from the equivalence between graded superdomains--- equivalently bundles of non-degenerate graded superdomains whose base is $\R^n$ for some $n$--- and the subcategory of the category of graded manifolds on the $\N$-graded vector spaces. Consider both of these equivalent categories as a Grothendieck site, where the covers in the graded manifold case are those induced by open covers of their cores. Then we have an equivalence between their $2$-categories of stacks of groupoids
\begin{equation*}
F^*:\St\left(\mathsf{Cart}\Mfd_\N\right) \stackrel{\sim}{\longrightarrow} \St\left(\mathsf{SDom_{\N}}\right).
\end{equation*}

Since every graded manifold has a cover by $\N$-graded vector spaces, it follows that the canonical functor
\begin{eqnarray*}
\Mfd_\N &\to& \Sh\left(\mathsf{Cart}\Mfd_\N\right)\\
\M &\mapsto & \Hom_{\Mfd_\N}\left(\blank,\M\right)
\end{eqnarray*}
is fully faithful. By abuse of notation, will will represent the functor of points for $\M$ again by $\M.$ Similarly, given a bundle of non-degenerate graded superdomains $\sM \to M,$ over a manifold $M,$ there is a cover over $M$ by open submanifolds diffeomorphic to $\R^n,$ over which it trivializes, and for an open $U$ in this cover, $\sM_U \to U$ is a graded superdomain. It follows, that the category of bundles of non-degenerate superdomains embeds fully faithfully into $\Sh\left(\mathsf{SDom_{\N}}\right),$ by the functor of points construction.

Let $\M=\left(M,\O_\M\right)$ be graded manifold of the form $M \times \V$ for some graded vector space concentrated in strictly positive degrees. Note that any automorphism $\varphi$ of $\M$ commutes over the core $|\M|=M,$ that is, is a fiber-bundle automorphism. More categorically, this says $\varphi \in \Aut_{M}\left(M \times \V,M \times \V\right).$ Let $\sN$ be a non-degenerate graded superdomain for which $F\left(\sN\right)=\V.$ It follows that $$F^* \underline{\Aut}\left(\V\right) \simeq \underline{\Aut}\left(\sN\right),$$ where $\underline{\Aut}\left(\V\right)$ is the automorphism sheaf of (the representable sheaf associated to) $\V.$

Now let $\M$ be an arbitrary graded manifold. Since $\M \to |\M|$ is a fiber bundle with the same above purely positive fiber $\V,$ it is arises via an associated bundle construction from an $\underline{\Aut}\left(\V\right)$-principal bundle, in the topos $\Sh\left(\mathsf{Cart}\Mfd_\N\right).$ Let $p:|\M| \to \mathsf{B}\underline{\Aut}\left(\V\right)$ be the map classifying this principal bundle in $\St\left(\mathsf{Cart}\Mfd_\N\right).$ Then there is a pullback diagram
$$\xymatrix{\M \ar[d]_-{\pi_\M} \ar[r] & \left[\V/\underline{\Aut}\left(\V\right)\right] \ar[d]\\
|\M| \ar[r]_-{p} &  \mathsf{B}\underline{\Aut}\left(\V\right).}$$ The right adjoint $F_*$ to the equivalence $F^*$ produces a pullback diagram
$$\xymatrix{F_\ast\left(\M\right) \ar[d]  \ar[r] & \left[\sN/\underline{\Aut}\left(\sN\right)\right] \ar[d]\\
|\M| \ar[r]_-{F_\ast p} &  \mathsf{B}\underline{\Aut}\left(\sN\right),}$$
exhibiting $F_\ast\left(\M\right)$ as the functor of points of a bundle of non-degenerate graded superdomains of $|\M|.$ Conversely, given a bundle $\sM \to M$ of non-degenerate superdomains, the same argument shows that $F^*\sM$ is the functor of points of a graded manifold. The result now follows.
\end{proof}

Let $\left(\R^{S_0|S_1},E_\omega\right)$ be a standard graded superdomain of weight $\omega,$ equipped with coordinates $\left(x_i;\eta_j\right)_{i\in S_0,j \in S_1}.$ Define an action $\rho$ of the multiplicative monoid $\left(\R,\cdot\right)$ on $\R^{S_0|S_1}$ by defining
$$t \star \left(x_i;\eta_j\right)=\left(t^{\omega\left(i\right)}x_i;t^{\omega\left(j\right)} \eta_j\right),$$ where for $\omega\left(i\right)=0,$ $t^{\omega\left(i\right)}=1$ (for any $t$). Notice that $-1$ acts by parity involution. The exponential map
$$\exp:\left(\R,+\right) \stackrel{\sim}{\longrightarrow} \left(\R_{> 0},\cdot\right)$$ provides an isomorphism of Lie groups, identifying $\rho|_{\R_{> 0}}$ with a $1$-parameter subgroup of diffeomorphisms. The associated vectorfield is 
\begin{eqnarray*}
\frac{d}{ds}|_{s=0}&=& e^{s}\star \left(x_i;\eta_j\right)\\
\frac{d}{ds}|_{s=0}&=& \left(e^{\omega\left(i\right)\cdot s}x_i;e^{\omega\left(j\right)\cdot s} \eta_j\right)\\
&=& \left(\omega(i)x_i;\omega(j)\eta_j\right)\\
=E_\omega.
\end{eqnarray*}
It follows that for any other graded superdomain 
$\left(\R^{T_0|T_1},E_\lambda\right),$ a smooth map of supermanifolds $f:\sM=\R^{S_0|S_1} \to \R^{T_0|T_1}=\sN$ is equivariant with respect to the canonical $\left(\R_{>0},\cdot\right)$-actions if and only if it is a map of graded superdomains. Moreover, clearly equivariance with respect to $\left(\R,\cdot\right)$-implies the preservation of the Euler fields, since they are the infinitesimal generators. Conversely, $\rho\left(0,\blank\right)$ can be identified with the canonical projection to $M_0,$ making $\sM$ into a fiber bundle, but all morphisms of graded superdomains are morphisms of such fiber bundles, hence $f\left(0\star\left(\blank\right)\right)=0\star f\left(\blank\right).$ Finally, \cite[Remark 4.1]{gbun} implies that $f$ is equivariant for the whole monoid action.

By the above argument, for any bundle $\sM \to M$ of non-degenerate superdomains, its transition functions must be $\left(\R,\cdot\right)$-equivariant, and hence there is a globally defined fiber-wise $\left(\R,\cdot\right)$-action on $\sM.$

\begin{definition}
An action $\rho$ of of $\left(\R,\cdot\right)$ on a supermanifold $\sM$ is called \textbf{even} if $-1$ acts by parity involution. That is, for all $f \in \Ci\left(\sM\right),$ 
$$f\circ \rho\left(-1,\blank\right)=\left(-1\right)^{|f|}\cdot f.$$
\end{definition}

Let $$\rho:\R \times \sM \to \sM,$$ be any smooth monoid action and let $\rho_0$ be the smooth endomorphism $\rho\left(0,\blank\right)$ of $\M.$ Notice that $\rho_0$ is idempotent, since $0 \cdot x$ is a fixed point. It follows by the constant rank theorem for supermanifolds that the image of $\rho_0$ is a sub-supermanifold, and $\rho_0$ factors as
$$\sM \stackrel{p_0}{\longrightarrow} \rho_0\left(\sM\right):=M_0 \stackrel{i_0}\longhookrightarrow \sM,$$ with $p_0$ a submersion and $i_0$ an embedding, and $r_0 \circ i_0=id.$ In particular, $p_0$ is $\left(\R,\cdot\right)$-invariant.

Suppose that the action if \emph{even}. Then $M_0$ is purely even. To see this, notice that all coordinate functions on $M_0$ need to be fixed.

It turns out that one can always find coordinates on $\sM$ putting it into the local form of a graded superdomain:

\begin{theorem}\cite[Theorem 5.8]{monoids}
Let $\rho$ be an even smooth action of $\left(\R,\cdot\right)$ on a supermanifold $\sM,$ and let $E$ be the infinitesimal generator $\rho_*\left(\frac{d}{dt}\right).$ Then $$p_0:\sM \to M_0$$ exhibits $\sM$ as a bundle of non-degenerate graded superdomains, with Euler field $E.$
\end{theorem}

\begin{corollary}
There is a canonical equivalence of categories 
$$\mathsf{GrSMfdBun}_0 \simeq \mathsf{\SMfd}_0^{\left(\R,\cdot\right)}$$
between the category of bundles of non-degenerate graded superdomains with purely even base, and the category of supermanifolds equipped with an even action of $\left(\R,\cdot\right)$, and equivariant maps.
\end{corollary}

In summary, we have
$$\Mfd_\N \simeq \mathsf{GrSMfdBun}_0 \simeq \mathsf{\SMfd}_0^{\left(\R,\cdot\right)}.$$

It is worth unwinding the equivalence $$\mathsf{\SMfd}_0^{\left(\R,\cdot\right)} \stackrel{\sim}{\longrightarrow} \Mfd_\N.$$
Given an even action of $\left(\R,\cdot\right)$ on a supermanifold $\sM,$ there is a submersion $p:\sM \to M$ onto the fixed point locus, which is a purely even manifold, and moreover, $\sM$ admits coordinates which are homogeneous, i.e. $E\left(x_i\right)=n_i \cdot x_i,$ with $n \in \N,$ and in particular, is a bundle of non-degenerate graded superdomains. Let $U \subseteq M$ be an open subset. Let $n$ be a non-negative integer. Write $\O^{gr}_\M\left(U\right)_{-n}$ for the subset of functions on $\sM_U$ which are homogeneous of degree $n,$ and write 
$$\O^{gr}_\M\left(U\right):=\underset{n \in \N} \bigoplus \O^{gr}_\M\left(U\right)_{-n}.$$ The supercommutative $\R$-algebra structure on $\Ci\left(\sM_U\right)$ descends to $\Z$-graded commutative $\R$-algebra structure on $\O^{gr}_\M\left(U\right),$ which is concentrated in non-positive degrees. This defines a sheaf $\O^{gr}_\M$ over $M$ of $\Z$-graded commutative algebras. Since over any trivialization, the homogeneous functions are homogeneous polynomials in the fiber-coordinates, it follows that if $\sM \to M$ trivializes over $U,$ and has typical fiber the non-degenerate graded superdomain $\left(\R^{S_0|S_1},E_\omega\right),$ then there is a canonical isomorphism
$$\O^{gr}_\M\left(U\right) \cong \Ci\left(U\right) \underset{\R} \otimes \Sym\left(F\left(\left(\R^{S_0|S_1},E_\omega\right)\right)^*\right).$$

We now turn to the case with non-zero differential.

\subsubsection{Dg-manifolds as supermanifolds}





The internal exponent $\End\left(\R^{1|1}\right)={\R^{0|1}}^{\R^{0|1}}$ exists in $\SMfd$ and has underlying supermanifold $\R^{1|1}.$ That is to say, for any supermanifold $\sM$, there is a canonical bijection
$$\Hom_{\SMfd}\left(\sM,\R^{1|1}\right) \cong \Hom_{\SMfd}\left(\sM \times \R^{0|1},\R^{0|1}\right).$$
To boot, it can be realized as the following string of natural identifications
\begin{eqnarray*}
\Hom_{\SMfd}\left(\sM \times \R^{0|1},\R^{0|1}\right) &\cong& \Ci\left(\sM \times \R^{0|1}\right)_1\\
&\cong& \left(\Ci\left(\sM\right) \underset{\R} \otimes \Lambda\left(\R\right)\right)_1\\
&\cong& \left(\Ci\left(\sM\right) \underset{\R} \otimes \R^{1|1}\right)_1\\
&\cong& \left(\Ci\left(\sM\right)_0\otimes \R\right) \oplus \left(\Ci\left(\sM\right)_1 \otimes \R\right)\\
&\cong& \Ci\left(\sM\right)_0 \times \Ci\left(\sM\right)_1\\
&\cong& \Hom_{\SMfd}\left(\sM,\R^{1|1}\right).
\end{eqnarray*}
$\End\left(\R^{1|1}\right)=\R^{1|1}$ therefore inherits the structure of a monoid object in $\SMfd.$ To understand how this works concretely, it is helpful to look at the evaluation map
$$ev:\R^{1|1} \times \R^{0|1} \to \R^{0|1}.$$
This is a faithful action of $\R^{1|1}$ on $\R^{0|1}.$ If $\left(x,\alpha\right)$ are coordinates for $\R^{1|1},$ and $\xi$ is the single odd coordinate of $\R^{0|1},$ the above map is
$$\left(t,\alpha\right)\cdot \xi=\xi \mapsto t\xi+\alpha.$$
It follows that $$\left(t,\alpha\right) \circ \left(y,\beta\right)=\left(ty,\alpha+x\beta\right),$$ and that we therefore have
$$\End\left(\R^{1|1}\right)\cong \left(\R,\cdot\right) \ltimes \left(\R^{0|1},+\right).$$
We also have a corresponding Lie supergroup $$\Aut\left(\R^{0|1}\right)\cong \R^* \ltimes \R^{0|1}.$$ Accordingly, its Lie superalgebra is equivalent to the Lie algebra of vector fields on $\R^{0|1},$ which is isomorphic to $\R^{1|1}$ as a $\Z_2$-graded vector space. It is spanned by the even vector field $$\varepsilon:=-\xi\frac{\partial}{\partial \xi},$$ and the odd vector field $$\delta:=\frac{\partial}{\partial \xi}.$$ Alternatively, this can be directly calculated as $$-\varepsilon=ev_*\frac{\partial}{\partial t},$$ and
$$\delta=ev_*\frac{\partial}{\partial \alpha}.$$
The minus sign is a matter of convention; it will induce a grading, and we wish it to be concentrated in non-positive degrees rather than non-negative degrees, so that we can use cohomological grading conventions.

They satisfy the commutation relations
\begin{equation}
\left[\varepsilon, \delta\right]=\delta
\end{equation}
and
\begin{equation}
\left[\delta,\delta\right]=0.
\end{equation}
Of course, $\left[\varepsilon,\varepsilon\right]=0,$ but this is automatic as it is an even vectorfield.

Therefore, a smooth action $\rho$ of $\Aut\left(\R^{0|1}\right)$ on a supermanifold $\sM,$ gives rise to an even vectorfield $E=\rho_*\varepsilon$ and an odd vectorfield $D=\rho_*\delta,$ such that $\left[E,D\right]=D$ and $\left[D,D\right]=2D^2=0.$ Chasing through the definitions, one can identify $-E$ with the Euler field of the induced action of $\left(\R_{>0},\cdot\right),$ and $D$ with the infinitesimal generator of the action of the odd additive group $\left(\R^{0|1},+\right).$ Nonetheless, we will abusively call $E$ the Euler field.

\begin{definition}
An $\Aut\left(\R^{0|1}\right)$-action on a supermanifold $\sM$ will be called \textbf{even} if $-1$ acts by parity involution on $\sM,$ i.e. the induced action of $\left(\R,\cdot\right)$ is even.
\end{definition}

In general, smooth $\Aut\left(\R^{0|1}\right)$-actions can be wild, however, those which extend to $\End\left(\R^{1|1}\right)$-actions are much better behaved. If $\rho$ is an even action of $\End\left(\R^{1|1}\right)$ on $\sM,$  the induced action of $\left(\R,\cdot\right)$ is even and we get a submersion $p:\sM \to M$ onto the fixed point locus, which is a purely even manifold. The subsheaf $\O^{gr}_\M$ of $\O_\sM|_{M}$ spanned by homogeneous elements then makes $\left(M,\O^{gr}_\M\right)$ into an $-\N$-graded manifold. The minus sign shows up since we have negated the Euler field.

Let $f$ be a homogeneous function of degree $n$ on $\sM_U$ for some open subset of $M,$ with respect to the $\left(\R,\cdot\right)$-action. This means that $E\left(f\right)=-n,$ since we have let $E$ be the opposite of the infinitesimal generator. We can identify $f$ with a function on the graded manifold $\left(U,\O^{gr}_\M|_{U}\right)$ of degree $-n,$ and conversely. So we have
$$E\left(f\right)=|f| \cdot f,$$ where $|f|$ is the degree of $f$ as a function on $\left(U,\O^{gr}_\M|_{U}\right).$ The commutation relations for the Lie algebra of $\Aut\left(\R^{0|1}\right)$ imply that for any such $f,$ $ED\left(f\right)=\left(n+1\right)D\left(f\right).$ So $D$ restricts to a derivation of $\O^{gr}_\M$ of degree $+1.$ Moreover, since $\left[D,D\right]=0$ on $\sM,$ its restriction also satisfies this, so it induces a cohomological vector field (which we will abusively call $D$) on $\left(M,\O^{gr}_\M\right)$ making $\left(M,\O^{gr}_\M,D\right)$ into a differential graded manifold. Since  the vector field $D$ on $\sM$ is determined by what it does on local coordinate functions, and these may be chosen homogeneous, and hence elements of $\O^{gr}_\M,$ $D$ can be recovered from its restriction to $\left(M,\O^{gr}_\M\right).$
Conversely, given a dg-manifold $\left(\M,\O_{\M},Q\right),$ we know that we can associated to $\M$ a supermanifold $\sM$ with an even action of $\left(\R,\cdot\right)$ with fixed point locus $M,$ and the local coordinates on $\M$ become homogeneous coordinate systems on $\sM$ with respect to this action. We can then extend $Q$ as a derivation on all of $\Ci\left(\sM\right),$ as we know what it does on local coordinates.

The following theorem is also folklore:
\begin{theorem}
There is a canonical equivalence of categories
$$\SMfd_0^{\End\left(\R^{1|1}\right)} \simeq \dgMan,$$ between the category of supermanifolds with an even $\End\left(\R^{1|1}\right)$-action and equivariant maps, and the category of non-positively graded dg-manifolds.
\end{theorem}

\begin{proof}
We know the functor on objects, namely, if $\sM$ is a supermanifold with an even $\End\left(\R^{1|1}\right)$-action, with induced even and odd vector fields $E$ and $D$ as above, $\left(M,\O^{gr}_\M,D\right)$ is the associated dg-manifold, where $M$ is the fixed point locus. Since $D$ is uniquely determined by its restriction to $\left(M,\O^{gr}_\M\right),$ and we already have an equivalence of categories between supermanifolds with an even $\left(\R,\cdot\right)$-action and graded manifolds, it suffices to show that a morphism between two supermanifolds with an $\left(\R^{0|1},+\right)$-action is equivariant if and only if it intertwines the infinitesimal generator. But an $\left(\R^{0|1},+\right)$-action is an odd one-parameter subgroup, which is automatically complete, hence completely determined by its associated odd-vector field.
\end{proof}

\subsection{Dg-manifolds as bundles of curved $L_\i\left[1\right]$-algebras}
Finally, we will explain yet another equivalent approach to dg-manifolds. We will not invest as much time explaining, as we will not be using it in this paper. However, we will be relying on the results of \cite{dgder} which are cast in this language, so we shall review it briefly in order to be able to effortlessly translate.

\begin{definition}
A \textbf{bundle of curved $L_\i\left[1\right]$-algebras} on a smooth manifold $M$ is a triple $\left(M,L,\lambda\right)$ of a finite dimensional graded vector bundle $L=\bigoplus \limits_{i=1}^n L_i$ equipped with smooth multi-linear maps $$\lambda^k:L \times_M^k L= L \times_{M} L \times_M \ldots \times_M L \to L$$ of degree $+1,$ for each $k=0,\ldots n$ which makes each fiber of $L$ into a curved $L_\i\left[1\right]$-algebra.

A \textbf{morphism} in this category $$\left(M,L,\lambda\right) \to \left(N,E,\mu\right)$$ is  a pair $\left(f,\phi\right)$ where $f:M \to N$ is a smooth map, and $\phi$ is a collection of maps
$$\phi^k:L \times_M^k L= L \times_{M} L \times_M \ldots \times_M L \to E$$ covering $f$ which induce morphisms of curved $L_\i\left[1\right]$-algebras on each fiber.
\end{definition}

\begin{remark}
Hidden in the notation of this definition is that $\lambda_0$ is simply a section of $L_1.$
\end{remark}

Given such an $L,$ we can canonically associated the vector bundle in graded manifolds $$\mathbb{L}:=\bigoplus \limits_{i=1}^n L_i\left[-i\right]\to M,$$ which places each $L_i$ in degree $i.$ The total space of this bundle serves as the underlying graded manifold of the associated dg-manifold. Hidden in the definition of the $\lambda$ is that each $\lambda^k$ must be graded-symmetric, so we get a morphism of graded vector bundles $$\lambda:\underset{k} \bigoplus \Sym^k\left(\mathbb{L}\right) \to \mathbb{L}\left[1\right]$$ which dualizes to a map
$$\mathbb{L}^*\left[-1\right] \to \Sym\left(\mathbb{L}^*\right).$$
If each $L_i$ has typical fiber $V_i$ and $\V$ is the associated $\N$-graded vector space, locally on an open over which it trivializes we have a map of graded $\Ci\left(U\right)$-modules
$$\Ci\left(U\right) \underset{\R} \otimes \mathbb{V}^*\left[-1\right] \to \Ci\left(U\right) \underset{\R} \otimes \Sym\left(\V^*\right)=\Ci\left(\mathbb{L}|_{U}\right).$$ This map uniquely extends to a derivation $D_\lambda$ of degree $+1,$ and the $L_\i$-algebra condition is equivalent to $D_\lambda^2=0.$

\begin{remark}
Again, some care must be taken with the case $k=0.$ Then $\Sym^0\left(\mathbb{L}\right)=\underline{\R},$ the trivial line bundle over $M,$ and the map
$$\underline{\R} \to \mathbb{L}\left[1\right]$$ is the same as a (degree 0) global section, i.e. a section of $\mathbb{L}\left[1\right]_0=L_1,$ and can be identified with $\lambda_0.$ Unwinding the definitions, the contribution to the differential $D_\lambda$ coming from $\lambda_0$ is the insertion operator $\iota_{\lambda_0},$ which acts on sections of $L^*_1$ by pairing and then extends as a derivation.
\end{remark}

Conversely, given a dg-manifold $\left(M,\O_{\M},D\right),$ using the graded Batchelor's theorem \ref{thm:BatchN}, we can reduce the structure group of the fiber bundle $\M \to M$ to produce an associated vector bundle in graded manifolds $\mathbb{L} \to M,$ with typical fiber $\V$ which is (non-canonically) isomorphic to $\M$ as a graded manifold. The cohomological vector field $D$ is determined by what it does in local coordinates, so locally determines maps of modules
$$\Ci\left(U\right) \underset{\R} \otimes \V^*\left[-1\right] \to \Ci\left(U\right) \underset{\R} \otimes \Sym\left(\V^*\right),$$
inducing a vector bundle map $$\mathbb{L}^*\left[-1\right] \to \Sym\left(\mathbb{L}^*\right),$$ which dualizes to a define a bundle of curved $L_\i$-algebras.

This is promoted to an equivalence of categories in \cite[Proposition 1.18]{dgder}.

\section{$\Ci$-algebras} \label{sec:cialg}
The prototypical example of a $\Ci$-algebra (a.k.a $\Ci$-ring) is the ring of smooth functions $\Ci\left(M\right)$ on a smooth manifold $M.$ This ring has an $n$-ary operation  associated to each smooth function $f:\R^n \to \R$:
\begin{eqnarray*}
\Ci\left(M\right)^n\cong \Hom_{\Mfd}\left(M,\R^n\right) &\to& \Ci\left(M\right)\\
\left(\varphi_1,\ldots,\varphi_n\right)=\varphi:M \to \R^n &\mapsto& f\circ \varphi.
\end{eqnarray*}

The precise definition is as follows:

Let $\Ci$ denote the full subcategory of $\Mfd$ spanned by the Euclidean manifolds $\R^n,$ for $n \ge 0.$ Every object is a finite product of $\R.$ This makes $\Ci$ into an \emph{algebraic theory,} with of one \emph{sort}, i.e. $\R$ is its single generator.

\begin{definition}
An \textbf{algebraic theory}, is a category $\bT$ (or more generally an $\i$-category) with finite products. A morphism of algebraic theories $$\bT \to \bT'$$ is a finite product preserving functor.
\end{definition}

\begin{example}
Let $\bbS$ a set. Regarding this set as a discrete category, let $\mathbf{T}_\mathbb{S}.$ be the
free completion of $\bbS$ with respect to finite products. Concretely, the objects are finite families
$$\left(s_i \in \mathbb{S}\right)_{i\in I}$$
and morphisms
$$\left(s_i \in \mathbb{S}\right)_{i\in I} \to \left(t_j \in \mathbb{S}\right)_{j\in J}$$
are functions of finite sets $f:J\to I$ such that
$$s_{f\left(j\right)}= t_j$$
for all $j\in J$.
Then $\mathbf{T}_\mathbb{S}$ is the \textbf{algebraic theory of $\bbS$-sorts}.

If $\bbS$ is a singleton set, then $\mathbf{T}_\mathbb{S}\cong \mathbf{FinSet}^{op}$ be the opposite category of finite sets. We denote this algebraic theory by $\bT_{obj}.$
\end{example}

\begin{definition}\label{definition:algebra}
Let $\sC$ be an $\i$-category with finite products. The $\i$-category $\Alg_{\bT}\left(\sC\right)$ of \textbf{$\bT$-algebras in $\sC$} is the full subcategory of $\Fun\left(\bT,\sC\right)$ on those functors which preserve finite products.

Of particular importance is the cases when $\sC=\Set$ the category of sets, and when $\sC=\Spc$, the $\i$-category of spaces (a.k.a. $\i$-groupoids). We will call an object of $\Alg_{\bT}\left(\Set\right)$ a $\bT$-algebra, and an object of $\Alg_{\bT}\left(\Spc\right)$ a homotopical $\bT$-algebra.
\end{definition}

\begin{example}
For any $\i$-category $\sC$ with finite products, there is a canonical equivalence $\Alg_{\bT_{obj}.}\left(\sC\right)\simeq \sC$ given by evaluation at $*$. More generally, for any set $\bbS,$ there is a canonical equivalence $\Alg_{T_{\bbS}}\left(\sC\right)\simeq \sC^{\bbS}.$ For example, if $\bbS=\left\{1,2,\ldots,n\right\},$ a $\bT_{\bbS}$-algebra in $\sC$ is the same data as the choice of $n$-objects of $\sC.$
\end{example}

\begin{definition}
Given an algebraic theory $\bT,$ an object $r \in \bT_0,$ called a \textbf{generator} if every object in $\bT$ is equivalent to $r^n,$ for some $n \ge 0.$ More generally, a subset $\bbS \subseteq \bT_0$ is said to be a set of generators for $\bT$ if every object in $\bT$ is equivalent to a finite product of objects in $\bbS.$
\end{definition}



Most types of algebraic objects have an associated algebraic theory. For example:

\begin{example}\label{ex:Comk}
Let $k$ be a commutative ring and let $\Comk$ be the the opposite category of finitely generated free $k$-algebras. Making this a skeletal category, we may represent it as the category of affine $k$-schemes of the form $\bbA_k^n,$ for $n \ge 0.$ Then, $\Comk$ is a $1$-sorted Lawvere theory with $\bbA_k^1$ as a generator.

A $\Comk$-algebra in $\Set$ is an ordinary commutative $k$-algebra. More precisely, given a $\Comk$-algebra $$\mathcal{A}:\Comk \to \Set,$$ the set $\mathcal{A}\left(\bbA_k^1\right)$ has an induced structure of a commutative $k$-algebra. This is because $\bbA_k^1$ is a commutative $k$-algebra object in schemes (and hence in $\Comk$) and any finite product preserving functor preserves commutative $k$-algebra objects. For example, the maps
\begin{eqnarray*}
\bbA_k^1 \times \bbA_k^1 = \bbA_k^2 &\to & \bbA_k^1\\
\left(x,y\right) &\mapsto & x+y
\end{eqnarray*}
and 
\begin{eqnarray*}
\bbA_k^1 \times \bbA_k^1 = \bbA_k^2 &\to & \bbA_k^1\\
\left(x,y\right) &\mapsto & x\cdot y
\end{eqnarray*}
define addition and multiplication for the ring object $\bbA_k^1.$ Conversely, given a commutative $k$-algebra $B,$ the functor $$\Hom\left(\blank,B\right)=\Hom\left(\Spec\left(B\right),\blank\right)$$ preserves finite products, hence is a $\Comk$-algebra in $\Set$. These constructions are categorically inverse to one another.
\end{example}



\begin{definition}
Let $\Cart$ denote the full subcategory of the category of smooth manifolds $\Mfd$ on those manifolds of the form $\R^n$ for $n \ge 0$ Then $\Cart$ is an algebraic theory with $\R$ as a generator. $\Cart$-algebras in $\Set$ are also known as \emph{$\Ci$-rings.}
\end{definition}


\begin{remark}
If $k=\R,$ $\ComR$ may also be seen as the category of manifolds of the form $\R^n$ whose morphisms are given by polynomials. Hence, there is an evident functor $\tau:\ComR \to \Cart$ which is a morphism of $1$-sorted Lawvere theories. 
$$\tau:\SComR \to \SCart.$$
\end{remark}

A $\Ci$-algebra is a commutative $\R$-algebra with extra structure. Indeed, if $\A:\Ci \to \Set$ is a $\Ci$-algebra, then $$i\circ \A:\ComR \to \Set$$ is a $\ComR$-algebra. As we have seen in Example \ref{ex:Comk}, this defines a commutative $\R$-algebra with underlying set $\A\left(\R\right).$ For $M$ a smooth manifold, as an algebra for the algebraic theory $\Ci,$ $\Ci\left(M\right)$ is the functor
$$\Hom_{\Mfd}\left(M,\blank\right):\Ci \to \Set.$$ 

For any algebraic theory $\T,$ and a generator $r,$ $j\left(r^n\right)=\Hom_\T\left(r,\blank\right)$ is the free $\T$-algebra (in $\Set$) on $n$-generators (of sort $r$). That is, it follows by the Yoneda lemma (since $\Alg_\T\left(\Set\right)$ is a full subcategory of $\Psh\left(\T^{op}\right)$) that for any $\T$-algebra $\A:\T \to \Set,$ there is a canonical bijection $$\Hom_{\Alg_\T\left(\Set\right)}\left(j\left(r^n\right),\A\right) \cong \A\left(r\right)^n.$$

\begin{definition}
The \textbf{free $\Ci$-algebra on $n$-generators} is $j\left(r^n\right)=\Hom_{\Mfd}\left(\R^n,\blank\right),$ i.e. 
$$\Ci\{x_1,x_2,\ldots,x_n\} \cong \Ci\left(\R^n\right).$$
\end{definition}




One of the main reasons that $\Ci$-algebras are reasonable to work with is the following theorem:

\begin{theorem}\cite[Proposition 1.2]{MSIA}, \cite[Corollary 2.116]{dg1}
Let $I$ be an ideal of the underlying $\R$-algebra $\A_\sharp$ of a $\Ci$-algebra $\A.$ Then $I$ is a $\Ci$-congruence and $\A_\sharp/I$ carries the canonical structure of a $\Ci$-algebra making $\A \to \A_\sharp/I=:\A/I$ a $\Ci$-homomorphism.
\end{theorem}

The category of $\Ci$-algebras is presentable, so, in particular, it has coproducts. Coproducts in the category of commutative $k$-algebras are computed by tensor product: If $\A$ and $\B$ are commutative rings, their coproduct is $\A \underset{k} \otimes \B$--- the tensor product of $k$-modules equipped with an induced $k$-algebra structure. This is no longer the case for (classical) $\Ci$-algebras, and this is a good thing, since we have for $n,m > 0,$
$$\Ci\left(\R^n\right) \underset{\R} \otimes \Ci\left(\R^m\right) \subsetneq \Ci\left(\R^{n+m}\right).$$ But $\Ci\left(\R^n\right)$ and $\Ci\left(\R^m\right)$ are the free $\Ci$-algebras on $n$ and $m$ generators respectively, so their coproduct must be the free $\Ci$-algebra on $\left(n+m\right)$ generators, namely $\Ci\left(\R^{n+m}\right).$ We denote the coproduct in the category of classical $\Ci$-rings by $\oinfty.$ Then we have
$$\Ci\left(\R^n\right) \oinfty \Ci\left(\R^m\right) \cong \Ci\left(\R^{n+m}\right).$$ Similarly, we denote pushouts by
$$\xymatrix{\A \ar[d] \ar[r] & \cC \ar[d] \\ \B \ar[r] & \B \underset{\A} \oinfty \cC.}$$


The tensor product $\oinfty$ is remarkably well-behaved.

\begin{proposition}\cite[Chapter 1, Theorem 2.8]{MSIA}
If $f:M \to N$ and $g:L \to N$ are transverse maps of smooth manifolds, then
$$\Ci\left(M \times_N L\right) \cong \Ci\left(M\right) \underset{\Ci\left(N\right)} \oinfty \Ci\left(L\right).$$
\end{proposition}




The following fact is fundamental:

\begin{proposition}
The canonical functor
$$\Ci\left(\blank\right):\Mfd \to \Alg_{\Ci}\left(\Set\right)$$
is fully faithful.
\end{proposition}

\subsection{Homotopical $\Ci$-algebras}

By definition, a homotopical $\Ci$-algebra is a finite product preserving functor 
$$\A:\Ci \to \Spc.$$ In particular, it has an underlying \emph{space} $\A\left(\R\right)$ which carries the structure of a homotopical $\R$-algebra. The set of connected components $\pi_0\left(\A\left(\R\right)\right)$ inherits the structure of a classical $\Ci$-algebra. This is a direct consequence of the fact that $\pi_0$ preserves finite products. This defines a truncation functor
$$\pi_0:\Alg_{\Ci}\left(\Spc\right) \to \Alg_{\Ci}\left(\Set\right)$$ from homotopical $\Ci$-algebras to classical $\Ci$-algebras. Conversely, since the inclusion $i:\Set \hookrightarrow \Spc$ preserves \emph{all} limits, for any $\Ci$-algebra $\A,$ $i \circ \A$ is a homotopical $\Ci$-algebra. This defines a fully-faithful inclusion
$$i:\Alg_{\Ci}\left(\Set\right) \hookrightarrow \Alg_{\Ci}\left(\Spc\right),$$ right adjoint to $\pi_0,$ therefore exhibiting $\Alg_{\Ci}\left(\Set\right)$ as a reflective subcategory.

We already remarked that the canonical functor $M \mapsto \Ci\left(M\right)$ constitute a fully faithful embedding of $\Mfd$ into $\Alg_{\Ci}\left(\Set\right)^{op}$ which preserves transverse pullbacks, but as $i^{op}$ is a \emph{left} adjoint, the analogous result for $\Alg_{\Ci}\left(\Spc\right)^{op}$ does not follow. That is to say, if $f:M \to N$ and $g:L \to N$ are transverse smooth maps, then even though
$$\Ci\left(M \times_N L\right) \cong \Ci\left(M\right) \underset{\Ci\left(N\right)} \oinfty \Ci\left(L\right),$$ one must show that this tensor product (pushout) is a derived tensor product (homotopy pushout). Nonetheless, the result is true:

\begin{theorem}\label{thm:transverse}\cite[Theorem 4.51]{univ}
The functor $\Ci:\Mfd \hookrightarrow \Alg_{\Ci}\left(  \Spc\right)^{op}$ sending a manifold $M$ to $\Ci\left(M\right)$ is fully faithful, and preserves transverse pullbacks. 
\end{theorem}

\begin{theorem}\cite[Theorem 5.1]{Be},\cite[Corollary 5.5.9.2]{htt} \label{theorem:bergner}
For any algebraic theory $\bT,$ there is an equivalence of $\i$-categories $$N_{hc}\left(\Alg_{\bT}\left(\Set\right)^{\Delta^{op}}_{proj.}\right) \simeq \Alg_{\bT}\left(\Spc\right)$$
between the homotopy coherent nerve of the category of simplicial $\bT$-algebras, endowed with the projective model structure, and the $\i$-category of $\bT$-algebras in $\Spc.$
\end{theorem}

By Theorem \ref{theorem:bergner}, any homotopical $\Ci$-algebra can be modeled by a simplicial $\Ci$-algebra in sets. Since any $\Ci$-algebra has an underlying abelian group, such a simplicial algebra has an underlying simplicial abelian group. It follows that the underlying simplicial set is a Kan complex. The projective model structure on the category of $\Ci$-algebras in $\Set^{\Delta^{op}}$ is such that:
\begin{itemize}
\item[i)] $f:X \to Y$ is a \textbf{weak equivalence} if
$$f\left(\R\right):X\left(\R\right) \to Y\left(\R\right)$$ is a weak homotopy equivalences of simplicial sets and
\item[ii)] $f:X \to Y$ is a \textbf{fibration} if the above map is a Kan fibration.
\end{itemize}
The $\i$-category associated to this model category is equivalent to $\Alg_{\SCi}\left(\Spc\right).$

\subsection{Differential Graded $\Ci$-algebras}
Let a homotopical $\Ci$-algebra $\A$ be represented by a simplicial $\Ci$-algebra $A_\bullet.$ As discussed, $A_\bullet$ has the underlying structure of a simplicial abelian group. The classical Dold-Kan correspondence associates to this simplicial abelian group a cochain complex. In fact, this abelian group is an $\R$-module, so this becomes a cochain complex over $\R,$ concentrated in non-positive degrees. One may wonder if one can transfer the $\Ci$-algebra structure to this cochain complex.
\begin{definition}
A (non-positively graded) \textbf{differential graded $\Ci$-algebra} (dg-$\Ci$-algebra), is a commutative differential graded $\R$-algebra $\left(\cA^{\bullet},d\right),$ together with the additional structure of a lift of the induced commutative $\R$-algebra structure on $\cA_0$ to the structure of a $\Ci$-algebra. A morphism $f:\left(\cA,d\right)\to \left(\cA',d'\right)$ between two such algebras is a morphism of differential graded $\R$-algebras such that the morphism $f^0:\cA^0 \to \cA'^0$ is a morphism of $\Ci$-algebras.
\end{definition}

\begin{remark}
If $\A^\bullet$ is not concentrated in non-positive degrees, the $\Ci$-algebra structure is on the algebra of $0$-cycles.
\end{remark}

The prototypical example of a dg-$\Ci$-algebra is the dg-algebra of smooth functions on a dg-manifold $\left(\Ci\left(\M\right),D\right).$

\begin{theorem}\cite[Theorem 6.5]{dg2} \label{thm:dgmodel}
There exists a cofibrantly generated, almost simplicial, model category structure on the category of (non-positively graded) dg-$\Ci$-algebras, unique with the property that $$f:\left(\cA,d\right)\to \left(\cA',d'\right)$$ is a weak equivalence (respectively fibration) if and only if the induced map of underlying cochain complexes is, with respect to the projective model structure on cochain complexes $\textbf{Ch}\left(\R\right)_{\le 0}.$ 
\end{theorem}

\begin{definition}
Denote by $\underline{\mathbf{dg}\Ci\Alg}_{\le 0}$ the $1$-category of non-positively graded dg-$\Ci$-algebras and denote by $\dgci$ the $\i$-category associated to the above model category, and denote the associated $\i$-category of non-positively graded cochain complexes by $\ch.$
\end{definition}

The above model structure was constructed by transfer from the projective model structure on non-positively graded cochain complexes, so there is a natural Quillen adjunction. The right adjoint sends a dg-$\Ci$-algebra $\left(\A^\bullet,d\right),$ to its underlying cochain complex. The left adjoint sends a cochain complex $\left(V^\bullet,d\right)$ to the \emph{free dg-$\Ci$-algebra on $\left(V^\bullet,d\right).$} Explicitly, it is
$$\Sym_{\Ci}\left(\left(V^\bullet,d\right)\right):=  \Ci\left(V^0\right) \underset{\Sym\left(V^0\right)} \otimes \Sym\left(\left(V^\bullet,d\right)\right).$$ 

Note that for $\V$ an $\N$-graded finite dimensional vector space, regarded simultaneously as a cochain complex with zero differential, and a graded manifold, then we have an identification
$$\Ci\left(\V\right)\cong \Sym_{\Ci}\left(\V^*\right).$$

\begin{definition}
Let $\V$ be a $\Z$-graded vector space concentrated in non-positive degrees, which is of total finite dimension, let $m$ be the smallest $k$ for which $V_k \ne 0.$ For all $n \ge 0,$ let $x_{i,1},\ldots,x_{i,n_i}$ be a basis for $V_{-n}.$ Then we write
$$\Ci\{x_{1,1},x_{i,2},\ldots x_{1,n_1}, x_{2,1},x_{2,2},\ldots x_{m,n_m}\}:=\Sym_{\Ci}\left(\V\right),$$ for the \textbf{free $\Ci$-algebra on $n_1$-generators of degree $-1,$ $n_2$-generators of degree $-2$,..., and $n_m$-generators of degree $-m.$}
\end{definition}

Finally, we note that in the projective model structure, every cochain complex is cofibrant, so the left derived functor of $\Sym_{\Ci}$ is itself.

\begin{definition}
A morphism $f:\left(\A,d\right) \to \left(\B,d\right)$ between  dg-$\Ci$-algebras is \textbf{quasi-free} if there exists a $-\N$-graded super vector space $W^\bullet$ such that $f$ is isomorphic to the inclusion of $\A$ into $\A \underset{\R}\otimes \Sym_{\Ci}\left(W^\bullet\right),$ after forgetting the differentials.
\end{definition}

\begin{proposition}\cite{dg2}
A morphism $f:\left(\A,d\right) \to \left(\B,d\right)$ between non-positively graded differential $\Ci$-algebras is a cofibration if and only if it is a retract of a quasi-free extension.
\end{proposition}

This implies that the cofibrant objects of $\dgcni$ are retracts of quasi-free dg-$\Ci$-algebras by \cite[Corollary 6.20]{dg2}. In particular, each algebra $\Ci\left(\R^n\right),$ regarded as a dg-$\Ci$-algebra is cofibrant, and the collection of these algebras is closed under coproducts, and hence is also closed under derived $\Ci$-tensor products, as they coincide in this case. It follows that the canonical functor
\begin{eqnarray*}
\Ci & \to & \dgci\\
\R^{n} &\mapsto & \Ci\left(\R^{n}\right)
\end{eqnarray*}
preserves finite products. Hence, by \cite[Proposition 5.5.8.15]{htt}, we deduce that this extends to a unique colimit preserving functor of $\i$-categories $$N_{\Ci}:\Alg_{\Ci}\left(\Spc\right) \to \dgci$$ which sends each algebra of the form $\Ci\left(\R^{n}\right)$ to itself. By the adjoint functor theorem, it follows there exists a right adjoint $\Gamma^{\Ci}$. In fact, it follows from the Yoneda lemma that  $\Gamma^{\Ci}$ sends a dg-$\Ci$-algebra $\A$ to
\begin{eqnarray*}
\Gamma^{\Ci}\left(\A\right):\Ci &\to & \Spc\\
n &\mapsto & \Map_{\dgci}\left(\Ci\left(\R^{n}\right),\A\right).
\end{eqnarray*}

\begin{theorem}\label{thm:DK}[The $\Ci$ Dold-Kan Correspondence] \cite[Corollary 2.2.10]{Nuiten}
The above adjunction
is an adjoint equivalence of $\i$-categories.
\end{theorem}

\begin{corollary}\label{cor:dk}
For any dg-$\Ci$-algebra $\left(\A^\bullet,d\right),$
$\Map_{\dgci}\left(\Ci\left(\R\right),\A\right)\simeq |\Gamma\left(\A\right)|,$
where $|\Gamma\left(\A\right)|$ is the homotopy type of the simplicial set corresponding to the underlying cochain complex of $\A$ via the Dold-Kan correspondence.
In particular, for all $n$, we have
$$\pi_n \Map_{\dgci}\left(\Ci\left(\R\right),\A^\bullet\right) \cong H^{-n}\left(\A,d\right).$$
\end{corollary}

\subsection{Homotopically finitely presented dg-$\Ci$-algebras and derived manifolds}\label{sec:dmfd}
\begin{definition}
A dg-$\Ci$-algebra $\A$ is \textbf{homotopically finitely presented} if
$$\Map_{\dgci}\left(\A,\blank\right):\dgci \to \Spc$$ preserved filtered colimits, i.e. $\A$ is a compact object in the $\i$-category $\dgci.$
\end{definition}

\begin{proposition}\label{prop:hfpret}
A dg-$\Ci$-algebra is finitely presented if and only if it is a retract of a finite colimit of dg-$\Ci$-algebras of the form $\Ci\left(\R^n\right),$ i.e, of finitely generated free dg-$\Ci$-algebras.
\end{proposition}

\begin{proof}
This is a translation of \cite[Lemma 3.21]{univ} under the equivalence of Theorem \ref{thm:DK}.
\end{proof}

\begin{corollary}\label{cor:hfpdman}
There is a canonical equivalence
$$\left(\dgci^{\mathbf{fp}}\right)^{op}\simeq \DMfd$$ between the opposite of the $\i$-category of homotopically finitely presented dg-$\Ci$-algebras, and the $\i$-category of derived manifolds.
\end{corollary}

\begin{corollary}\label{cor:retdmfd}
For every derived manifold $\mathscr{X},$ there exists a derived manifold $\mathscr{X'}$ which can be constructed through a finite iteration of pullbacks from $\R,$ such that $\mathscr{X}$ is a retract of $\mathscr{X}'.$ 
\end{corollary}

\begin{proof}
This follows immediately from Corollary \ref{cor:hfpdman}, and Proposition \ref{prop:hfpret}, and of course that $\R^n=\left(\R\right)^n.$
\end{proof}

\subsection{$\Ci$-Localizations}

\begin{definition}
Let $\A$ be a homotopical $\Ci$-algebra and let $a\in \pi_0\left( \A\right)$. We say that a map $f:\A\rightarrow \B$ such that $f\left( a\right)\in \pi_0\left( B\right)$ is invertible is a \textbf{localization of $A$ with respect to $a$} if for each $\cC \in \cialgsp$, the map $\Map_{\cialgsp}\left( \B,\cC\right)\rightarrow\Map_{\cialgsp}\left(\A,\cC\right)$ given by composition with $f$ induces an equivalence
\[\Map_{\cialgsp}\left(\B,\cC\right)\overset{\simeq}{\longrightarrow} \Map^0_{\cialgsp}\left( \A,\cC\right),\]
where $\Map^0_{\cialgsp}\left( \A,\cC\right)$ is the union of those connected components of 
\end{definition}
In the case of an ordinary $\Ci$-algebra in set $\A,$ and some $a\in \A$, the above definition reduces to the usual $\Ci$ localization $\A\left[1/a\right]$ given up to equivalence by the pushout
\begin{equation*}
\begin{tikzcd}
\Ci\left( \R\right)\ar[r,"q_a"]\ar[d]& \A\ar[d]\\
\Ci\left( \R\setminus \{0\}\right)\ar[r] & \A\left[1/a\right]
\end{tikzcd}    
\end{equation*}
of $C^{\infty}$-rings. The localization of homotopical $C^{\infty}$-ring admits a similar characterization, for which we will need the following definition.
\begin{definition}\label{strongmap}
\begin{enumerate}
    \item A map $f:\A\rightarrow \B$ in $\Alg_{\ComR}\left(\Spc\right)$ is \textbf{strong} (in the sense of \cite[Definition 2.2.2.1]{ToeVez}) if the natural map
\[ \pi_n\left( \A\right)\underset{\pi_0\left( \A\right)}\otimes\pi_0\left( \B\right)\rightarrow \pi_n\left( \B\right)\]
is an isomorphism for all $n\geq 0$.
\item A map $f:\A\rightarrow \B$ of homotopical $C^{\infty}$-algebras is \textbf{strong} if $f^{\sharp}:\A^{\sharp}\rightarrow \B^{\sharp}$ is strong.
\end{enumerate}
\end{definition}
\begin{proposition}\label{localization}\label{prop:locz}
Let $\A$ be a homotopical $\Ci$-algebra and let $a\in \pi_0\left( \A\right)$, and let $f:A\rightarrow B$ a map of simplicial $C^{\infty}$-rings. The following are equivalent:
\begin{enumerate}
    \item The map $f:A\rightarrow B$ exhibits $B$ as a localization with respect to $a$.
    \item For every $n\geq 0$, the induced map 
    \[\pi_n\left( \A^{\sharp}\right)\underset{\pi_0\left( A^{alg}\right)}\otimes\left( \pi_0\left( \A\right)\left[1/a\right]\right)^{alg}\rightarrow \pi_n\left( B^{alg}\right) \]
    is an equivalence; that is, $f$ is strong and the map of $C^{\infty}$-schemes corresponding to $\pi_0\left( A\right)\rightarrow \pi_0\left( B\right)$ is an open immersion. 
    \item $B$ fits into a pushout diagram
\begin{equation*}
\begin{tikzcd}
\Ci\left( \R\right)\ar[r,"q_a"]\ar[d]& A\ar[d,"f"]\\
\Ci\left( \R\setminus \{0\}\right)\ar[r] & B
\end{tikzcd}    
\end{equation*}
where $q_a$ is the unique up to homotopy map associated to $a\in \pi_0\left( A\right)$ (note that as a consequence, localizations always exist).
\end{enumerate}
\end{proposition}

We denote the $\Ci$-localization of a homotopical $\Ci$-ring $\A$ by an element $a \in \pi_0 \A$ by $\A\left[1/a\right].$

\subsection{$\Ci$-Koszul complexes}

\begin{definition}
An even element $q$ of a supercommutative ring $\A$ is \textbf{regular} if it is not a zero-divisor.

An odd element $p$ of a supercommutative ring $\A$ is \textbf{regular} if the sequence $$\A \stackrel{\cdot p}{\longrightarrow} \A \stackrel{\cdot p}{\longrightarrow} \A$$ is exact.

A sequence $\left(a_1,a_2,...,a_n\right)$ of (even or odd) elements in a supercommutative ring $\A$ will be called a \textbf{regular sequence}, if  the image of $a_{i+1}$ is a regular element of $\A/\left(a_1,\ldots,a_i\right)$ for all $i.$

If $\A$ is a $\Z$-graded commutative algebra, a sequence $\left(a_1,a_2,...,a_n\right)$ of elements (possibly of different degrees), is \textbf{regular} if it is a regular sequence in the underlying $\Z_2$-graded commutative algebra.
\end{definition}

\begin{definition}
Let $\A$ be a dg-$\Ci$-algebra. Let $a_1,\ldots, a_k$ be a set of elements of $\A$ of degrees $n_1,\ldots,n_k$ such that each $a_i$ is closed, i.e. $da_i=0.$ The \textbf{Koszul algebra} of $\A$ associated to this set is the ($1$-categorical) pushout
$$\xymatrix{\left(\Ci\{x_1,\ldots,x_k\},|x_i|=|a_i|,0\right) \ar[d]_-{\left(a_1,\ldots,a_k\right)} \ar[r] & \left(\Ci\{x_1,\ldots,x_k,\xi_1,\ldots,\xi_k\},d=x_i \frac{\partial}{\partial \xi_i},|\xi_i|=n_i -1\right) \ar[d]\\
\A \ar[r] & K\left(\A,a_1,\ldots, a_k\right),}$$
in $\dgcni.$
If $\A=\left(\A,d\right),$ then the underlying $\Z$-graded commutative algebra of $K\left(\A,a_1,\ldots, a_k\right)$ is
$K\left(\A,a_1,\ldots, a_k\right) \cong \A\{\xi_1,\ldots,\xi_k\},$ the free $\A$-algebra (in $\Ci$-algebras) on the generators $\xi_1,\ldots,\xi_k.$ Since all of these generators are strictly of negative degrees, we have that
$$\A\{\xi_1,\ldots,\xi_k\} \cong \A \underset{\R} \otimes \mathbb{R}\left[\xi_{1},\ldots,\xi_{k}\right],$$ where $\mathbb{R}\left[\xi_{1},\ldots,\xi_{k}\right]$ is the free commutative $\R$-algebra on the (graded) generators $\xi_{1},\ldots,\xi_k.$
Using Einstein summation conventions, the differential takes the form
$$D=D_{\A} + a^i \frac{\partial}{\partial \xi_i}.$$
\end{definition}




\begin{remark}
The Koszul algebra of $a_1,\ldots, a_k$ can be constructed inductively, so that $$K\left(\A,a_1,\ldots,a_n\right)=K\left(K\left(\A,a_1\right),a_2,\ldots,a_n\right)=\ldots= K\left(K\ldots\left(K\left(\A,a_1\right),a_2\right),\ldots,a_n\right).$$
\end{remark}

\begin{lemma}\label{lem:Kosz1}
Suppose that $\A$ is a dg-$\Ci$-algebra with zero differential, and let $a_i$ be regular element of $\A$ of degree $i.$ Then there is a quasi-isomorphism of dg-$\Ci$-algebras
$$K\left(\A,a_i\right) \stackrel{\sim}{\longrightarrow} \A/\left(a_i\right),$$
where $\left(a_i\right)$ is the homogeneous ideal generated by $a_i.$
\end{lemma}

\begin{proof}
\underline{Case 1}: $i$ is even. Then we have that $a_i$ is not a zero-divisor. Let $\xi_{i-1}$ denote the added generator of degree $i-1.$ Then $\xi_{i-1}$ is odd so squares to zero. Therefore, as a graded vector space, we have
$$K\left(\A,a\right) \cong \A^\bullet \oplus \xi_{i-1} \cdot \A^{\bullet}\left[1-i\right].$$ We have
$$d_{k}:\A^{k} \oplus \xi_{i-1} \cdot  \A^{k-i+1} \to \A^{k+1} \oplus   \xi_{i-1} \cdot \A^{k-i}$$ is zero for $k=0,\ldots,i,$ and for $k <i,$
$$d_k\left(b_{k}+\xi_{i-1}c_{k-i+1}\right)=a_i \cdot c_{k-i+1}.$$ Since $a_i$ is not a zero-divisor, we have $$d_k\left(b_{k}+\xi_{i-1}c_{k-i+1}\right) = 0 \iff c_{k-1+i}= 0.$$ This means that $\ker\left(d_k\right)=\A^k.$ It's also clear that $\mathbf{Im}\left(d_{k-1}\right)=\left(a_i\right)_{k},$ that is, the $k^{th}$ component of the principal ideal $\left(a\right).$ It follows that the canonical map $K\left(\A,a\right) \to \A/\left(a\right)$ is a quasi-isomorphism.

\underline{Case 2}: $i$ is odd. Then we have that $a_i x = 0$ iff $x=a_i b$ for some $b.$ So we have the underlying graded algebra of $K\left(\A,a\right)$ is the polynomial algebra $\A\left[\xi_{i-1}\right]$ on a generator of \emph{even} degree $i-1.$ An element of degree $k$ is polynomial of the form 
$x=c_k+\underset{j \ge 1} \sum c_{k-j\cdot\left(i-1\right)}\xi^j_{i-1},$ with each $c_{r} \in \A^{r}.$
So $$dx=\underset{j \ge 1} \sum c_{k-j\cdot\left(i-1\right)}\cdot a_i\cdot j\xi^{j-1}_{i-1}.$$ This means that $dx=0$ if and only if for all $j \ge 1,$ 
 $$c_{k-j\cdot\left(i-1\right)}\cdot a_i=0,$$
 but since $a_i$ is regular, this means that each $c_{k-j\cdot\left(i-1\right)}=0$ for $j \ge 1,$ i.e. $x=c_k,$ so $\ker\left(d_k\right)=\A_k.$ Again, we clearly have that $\mathbf{Im}\left(d_{k-1}\right)=\left(a\right)_{k},$ so it follows once more that the canonical map $K\left(\A,a\right) \to \A/\left(a\right)$ is a quasi-isomorphism.
 \end{proof}
 
\begin{proposition}\label{prop:kosz}
Suppose that $\A$ is a dg-$\Ci$-algebra with zero differential, and let $\left(a_1,a_2,\ldots a_n\right)$ be a regular sequence of elements of $\A.$ Then
$$K\left(\A,a_i\right) \stackrel{\sim}{\longrightarrow} \A/\left(a_1,a_2,\ldots,a_n\right).$$
\end{proposition}

\begin{proof}
We will prove this by induction on the number of elements in the sequence. We have already established this for $n=1.$ Suppose now that it is true for all sequences of $n-1$ regular elements. We wish to show it also holds for those of length $n.$

Let $\left(a_1,a_2,\ldots a_n\right)$ be a regular sequence of elements of $\A.$ Firstly, notice that since we are not allowing elements of degree $1,$ the Koszul construction can never add generators of degree $0.$ Therefore, for any dg-$\Ci$-algebra $\cB$ one can identify the underlying cdga of the Koszul algebra of an element $b_i \in \cB$ of degree $i$ as the pushout
$$\xymatrix{\R\left[x_i\right] \ar[r] \ar[d]_-{b_i} & \R\left[x_i,\xi_{i-1}\right] \ar[d]\\
\cB \ar[r] & K\left(\cB,b_i\right).}$$
Applying this to the case that $\cB=K\left(\A,a_1,\ldots,a_{n-1}\right),$ we can consider the stack of pushout diagrams
$$\xymatrix{\R\left[x_i\right] \ar[r] \ar[d]_-{a_n} & \R\left[x_i,\xi_{i-1}\right] \ar[d]\\
K\left(\A,a_1,\ldots,a_{n-1}\right) \ar[r] \ar[d]_-{ \rotatebox[origin=c]{90}{$\sim$}} & K\left(\A/\left(\A,a_1,\ldots,a_{n-1}\right),b_i\right) \ar[d]\\
\A/\left(\A,a_1,\ldots,a_{n-1}\right) \ar[r] & K\left(\A/\left(a_1,\ldots,a_{n-1}\right),a_{n}\right),}$$
where $$K\left(A,a_1,\ldots,a_{n-1}\right) \to \A/\left(\A,a_1,\ldots,a_{n-1}\right)$$ is a quasi-isomorphism. Since the projective model structure on non-positively graded differential algebras over $\R$ is left proper, and the top horizontal map is a cofibration, it follows that the map $$K\left(a_1,\ldots,a_n\right) \to K\left(\A/\left(a_1,\ldots,a_{n-1}\right),a_{n}\right)$$ is also a quasi-isomorphism. Since the outer square is also a pushout diagram, and $a_n$ is regular in $\A/\left(a_1,\ldots,a_{n-1}\right),$ Lemma \ref{lem:Kosz1} implies that $K\left(\A/\left(a_1,\ldots,a_{n-1}\right),a_{n}\right) \to \A/\left(a_1,\ldots,a_{n}\right)$ is a quasi-isomorphism. This completes the proof.
\end{proof}

\section{Derived $\Ci$-schemes}\label{sec:sch}


\begin{definition}
A \textbf{homotopically $\Ci$-ringed space} is a pair $\left(X,\cO_X\right)$ with $X$ a topological space and $\cO_X$ a sheaf with values in the $\i$-category $\cialgsp.$ Morphisms are pairs $$\left(  f,\alpha\right):\left(  X,\cO_X\right) \to \left(  Y,\cO_Y\right),$$ with $f$ a continuous map and $$\alpha:\cO_Y \to f_{*}\cO_X.$$ \footnote{For simplicity, we circumvent demanding $\pi_0 \O_X$ have local stalks, since we will only be concerned with examples where each such stalk has residue field $\R,$ so any morphism is automatically local.} More formally, the $\i$-category $\Loc$ of homotopically $\Ci$-ringed spaces is the full subcategory on such spaces of the total space of the coCartesian fibration associated to to the functor
$$\Shv\left(  \blank,\cialgsp\right):\Top \to \widehat{\mathbf{Cat}}_\i$$ induced by push-forward of sheaves, where $\widehat{\mathbf{Cat}}_\i$ is the $\i$-category of large $\i$-categories, and $\Top$ is the $1$-category of (small) topological spaces.
\end{definition}

Consider the canonical functor $$\Gamma:\Loc \to \cialgsp$$ induced by taking global sections of each structure sheaf. This functor has a right adjoint $$\Speci:\cialgsp^{op} \to \Loc$$  \cite[Proposition 5.1.8]{Nuiten} which can be described as follows:

Let us first describe the underlined space $\underline{\Speci\left(\A\right)}$ of $\Speci\left(\A\right).$ Consider the canonical functor $\Ci \to \Top$ sending $\R^{n}$ to itself. This functor preserves finite products, hence, by \cite[Proposition 5.5.8.15]{htt} induces a unique limit-preserving functor
$$\underline{\Speci}:\cialgsp^{op} \to \Top.$$ Concretely, the underlying set of $\underline{\Speci}\left(\A\right)$ is the set of homomorphisms $$\Hom\left(\pi_0\left(\A\right),\R\right).$$ Given an element $a \in \pi_0\left(\A\right),$ consider the function
\begin{eqnarray*}
f_a:\Hom\left(\pi_0\left(\A\right)_0,\R\right) &\to &\R\\
\varphi & \mapsto & \varphi\left(a\right).
\end{eqnarray*}
Then the subsets of the form $U_a:=f_a^{-1}\left(\R\setminus \left\{0\right\}\right)$ form a basis for a topology on $\Hom\left(\pi_0\left(\A\right),\R\right)$ which is moreover closed under finite intersections--- see \cite[Section 4.3.1]{univ}. Since this basis is closed under finite intersections, to describe a sheaf on $\underline{\Speci}\left(\A\right),$ it suffices to define a sheaf on the subcategory of all open subsets of $\underline{\Speci}\left(\A\right)$ on those of the form $U_a.$ Hence, there is a unique sheaf $\cO_\A$ arising as the sheafification of the presheaf $$U_a \mapsto \A\left[a^{-1}\right],$$ and we have
$$\Speci\left(\A\right)=\left(\underline{\Speci}\left(\A\right),\cO_\A\right).$$

\begin{definition}
An \textbf{affine derived $\Ci$-schemes} is a homotopically $\Ci$-ringed space in the image of $\Speci.$ Denote the full subcategory of $\Loc$ on these objects by $\Aff_{\Ci}.$
\end{definition}

Despite this, unlike in classical algebraic geometry, the spectrum functor $\Speci$ is not fully faithful in general. This was already realized for $\Ci$-rings in $\Set$ by Dubuc \cite{cinfsch}. It is fully faithful however when restricted to germ-determined finitely presented $\Ci$-rings by loc. cit.

\begin{definition}
A homotopical $\Ci$-algebra $\A$ is \textbf{complete} if the co-unit $$\varepsilon:\A \to \Gamma\left(\Speci\left(\A\right)\right)$$ is an equivalence. Denote the full subcategory on the complete algebras by $\dgc^{\wedge}.$
\end{definition}

\begin{proposition}\label{prop:comp}\cite[Example 5.1.30]{Nuiten}, \cite[Corollary 4.45]{univ}
If $\A \in \cialgsp$ is finitely presented, then it is complete.
\end{proposition}

Note that the adjunction $\Gamma \dashv \Speci$ restricts to an adjunction
$$\xymatrix@C=2cm{\cialgsp^{op} \ar@<-0.5ex>[r]_-{\Speci} & \Aff_{\Ci \ar@<-0.5ex>[l]_-{\Gamma}}}$$

\begin{proposition}\cite[Proposition 5.1.26]{Nuiten}
The unit $$\eta:id_{\Aff_{\Ci}} \Rightarrow \Speci \circ \Gamma$$ is an equivalence.
\end{proposition}

The following corollary follows immediately:

\begin{corollary}
The composite
$$\left(\cialgsp^{\wedge}\right)^{op} \hookrightarrow \cialgsp^{op} \stackrel{\Speci}{\longlonglongrightarrow} \Loc$$ is fully faithful, with essential image $\Aff_{\Ci}.$ In particular $$\left(\cialgsp^{\wedge}\right)^{op} \simeq \Aff_{\Ci}.$$ Moreover, $\cialgsp^{\wedge}$ is a reflective subcategory of $\cialgsp$ with the left adjoint to the inclusion $i^{\wedge}$ given by $L^{\wedge}:=\Gamma \circ \Speci.$
\end{corollary}

\part{New Results}\label{part:new}
\section{Differential Graded Geometry, Part II}\label{sec:ii}
\subsection{From dg-manifolds to derived $\Ci$-schemes}\label{sec:dgsch}
\begin{definition}\label{dfn:locus}
Let $\M=\left(M,\O_\M,D\right)$ be a dg-manifold. The $\Ci$-scheme $\pi_0\left(\M\right):=\Speci\left(H^0\left(\Ci\left(\M\right)\right)\right)$ is called the \textbf{classical locus} of $\M.$
\end{definition}

\begin{remark}
As a topological space, $\pi_0\left(\M\right)$ is the subspace of $M$ on which the cohomological vectorfield vanishes.
\end{remark}

\begin{lemma}
Let $\left(M,\O_\M,\D\right)$ be a dg-manifold. Then $\O_\M$ is a homotopy sheaf of dg-$\Ci$-algebras.
\end{lemma}

\begin{proof}
It suffices to prove that the underlying sheaf of cochain complexes is a homotopy sheaf. Since any open subset of $M$ is automatically a dg-manifold, it suffices to prove descent for a cover of $M.$ Let $\left(U_\alpha \hookrightarrow M\right)_\alpha$ be an open cover, and denote by $\mathcal{U}_\alpha$ the dg-manifold $\left(U_\alpha,\O_{\M}|_{U_\alpha},D\right)$. We need to show that the canonical map
$$\Ci\left(\M\right) \to  \lim\left[ \underset{\alpha} \prod \Ci\left(\mathcal{U}_\alpha\right) \rightrightarrows \underset{\alpha,\beta} \prod\Ci\left(\mathcal{U}_{\alpha\beta}\right)  \rrrarrow \ldots\right]$$
is a quasi-isomorphism. For each $n < 0$ let $\mbox{\v{C}}^\bullet\left(\mathcal{U},\O^n_{\M}\right)$ be the \v{C}ech complex of the abelian sheaf $\O^n_{\M}$ with respect to the above cover. Then these assemble into a double complex $\mbox{\v{C}}^\bullet\left(\mathcal{U},\O^\bullet_{\M}\right).$ Since $\O^\bullet_{\M}$ is bounded below, the above homotopy limit can be computed as the totalization of this double complex. Notice that for each $n,$ $\O^n_{\M}$ is a locally free sheaf associated to a vector bundle over $M,$ and hence is soft. It follows that its sheaf cohomology vanishes. Thus the double complex spectral sequence collapses on the $E_2$-page, and we obtain our result.
\end{proof}

\begin{lemma}\label{lem:acycpi}
Let $U \subseteq M$ be an open subset of a dg-manifold $\M=\left(M,\O_\M,D\right),$ and suppose that $U \cap \pi_0\left(\M\right)\ne 0,$ then $\O_{\cM}|_U$ is acyclic.
\end{lemma}

\begin{proof}
Since $\O_\M$ is a homotopy sheaf, it suffices to prove this for sufficiently small $U.$ Hence, we can assume that $U\cong \R^n,$ and $\O_\M\left(U\right)$ is a quasi-free dg-$\Ci$-algebra. We can therefore find graded coordinates $\left(x_1,\ldots,x_n;\eta_{ij}\right)$ such that $\left(x_1,\ldots,x_n\right)$ is a coordinate system on $\R^n,$ and for each $i,$ $\left(\eta_{i1},\eta_{i2},\ldots,\eta_{im_i}\right)$ are of degree $-i,$ with $i$ running from $i=1$ to $N,$ for some finite $N.$ The differential thus looks like
$$D=f^1_i \frac{\partial}{\partial \eta_{1i}} + f^2_j \frac{\partial}{\partial \eta_{2j}}+ \ldots,$$ with $f^i_j$ of degree $-i+1$ for all $i.$ It follows that we can construct $\O_\M\left(U\right)$ iteratively as $\A^{(0)}=\Ci\left(U\right),$ $$\A^{(k)}=K\left(\A^{(k-1)},f^k_1,f^k_2,\ldots,f^k_{m_k}\right),$$ with $\O_\M\left(U\right) =  \A^{(N)}.$ Since $\O_\M\left(U\right)$ is quasi-free, it is cofibrant in $\dgcni,$ and since all the maps
$$\Ci\{y_1,\ldots,y_{m_{k}}\} \to \left(\Ci\{y_1,\ldots,y_{m_k},\eta_{i1},\ldots,\eta_{m_k}\},d=y_i \frac{\partial}{\partial \xi_{ki}},|\xi_i|=k \right)$$ are cofibrations (they are in fact the generating cofibrations), it follows that each pushout diagram 
$$\xymatrix{\Ci\{y_1,\ldots,y_{m_{k}}\}  \ar[d]_-{\left(f^k_1,\ldots,f^k_{m_k}\right)} \ar[r] & \left(\Ci\{y_1,\ldots,y_{m_k},\eta_{i1},\ldots,\eta_{m_k}\},d=y_i \frac{\partial}{\partial \xi_{ki}},|\xi_i|=k\right) \ar[d]\\
\A^{(k-1)} \ar[r] & \A^{(k)},}$$ is a homotopy pushout. Therefore, if $\A^{(k-1)}$ is acyclic, so is $\A^{(k)}.$ Therefore, to show that $\O_\M\left(U\right)$ is acyclic, it suffices to show that $\A^{(1)}$ is. Let $x \in U$. Then, by hypothesis, $x \in \pi_0\left(\M\right),$ which means that there exists a $j$ such that $f^1_j\left(x\right) \ne 0.$ Without loss of generality, to ease notation suppose $f^1_1\left(x\right) \ne 0.$ By shrinking $U$ if necessary, but keeping it diffeomorphic to $\R^n,$ we can arrange that $f^1_1\left(y\right) \ne 0,$ for all $y \in U.$ This means that $f^1_1$ is a unit, and in particular, a regular element. Notice that $$\A^{(1)}\cong K\left(K\left(\Ci\left(U\right),f^1_1\right),f^1_2,\ldots,f^1_{m_1}\right).$$ By an analogous argument to the one just given, the pushout defining the outer Koszul algebra is a homotopy pushout, so it suffices to prove that $K\left(\Ci\left(U\right),f^1_1\right)$ is acyclic. Since $f^1_1$ is not a zero divisor, we know that $$K\left(\Ci\left(U\right),f^1_1\right) \stackrel{\sim}{\longrightarrow} \Ci\left(U\right)/\left(f^1_1\right),$$ and since $f^1_1$ is a unit, this quotient is $0.$ This completes the proof.
\end{proof}

\begin{proposition}\label{prop:qfcpt}
Let $\left(\A,d\right)$ be a finitely generated quasi-free dg-$\Ci$-algebra. Then $\A$ is homotopically finitely presented, i.e. a compact object in $\dgci.$
\end{proposition}

\begin{proof}
Compact objects are closed under finite colimits. The proof of Lemma \ref{lem:acycpi} shows that any such algebra can be obtained by taking finitely many pushouts of finitely generated free $\Ci$-rings. The result now follows.
\end{proof}

\begin{lemma}\label{lem:piok}
Let $\left(M,\O_\M,\D\right)$ be a dg-manifold. Then there is a canonical quasi-isomorphism
$$\Ci\left(\M\right) \simeq \Gamma_{\pi_0\left(\M\right)}\left(\O_{\M}|_{\pi_0\left(\M\right)}\right).$$
\end{lemma}

\begin{proof}
Let $\mathcal{J}$ be the category of open neighborhoods of $\pi_0\left(\M\right)$ in $M,$ ordered by inclusion. Then
$$\Gamma_{\pi_0\left(\M\right)}\left(\O_{\M}|_{\pi_0\left(\M\right)}\right)\cong \underset{U \in \mathcal{J}} \colim \O_{\M}\left(U\right).$$ 
Let $V$ be the compliment of $\pi_0\left(\M\right)$ in $M.$ Then for any open $U \supset K,$ $U \cup V$ is a $2$-term cover of $M.$ By Lemma \ref{lem:acycpi}, $\O_{\M}\left(V\right)$ and $\O_{\M}\left(U \cap V\right)$ are acyclic. Since $\O_{\M}$ is a homotopy sheaf, it follows that the canonical map $\O_\M\left(M\right)=\Ci\left(\M\right) \to \O_\M\left(U\right)$ is a quasi-isomorphism. These canonical maps are given by $\O_{\M}$ regarded as a functor, hence constitute a natural transformation $$\Delta_{\Ci\left(\M\right)} \Rightarrow \O_{\M}|_{\mathcal{J}},$$ which is component-wise a quasi-isomorphism. This implies that $\O_{\M}|_{\mathcal{J}}$ is equivalent to $\Delta_{\Ci\left(\M\right)},$ in the $\i$-category of functors $\Fun\left(\J^{op},\dgci\right).$ Recall that we have an equivalence of $\i$-categories $\Gamma^{\Ci}:\dgci \stackrel{\sim}{\longrightarrow} \Alg_{\Ci}\left(\Spc\right).$ It therefore suffices to show that for any homotopical $\Ci$-algebra $\A,$ $$\A \simeq \underset{U \in \mathcal{J}} \colim \A.$$ Since $\Ci$ is an algebraic theory, the forgetful functor $$\Alg_{\Ci}\left(\Spc\right) \to \Spc$$ is conservative and preserves filtered colimits, so it suffices to prove that for any space $X$ in $\Spc,$ 
$$X \simeq \underset{U \in \mathcal{J}} \colim X.$$
But $\underset{U \in \mathcal{J}} \colim X=\underset{\mathcal{J}} \colim \Delta_X \simeq X \times B\mathcal{J}.$ Since we are dealing with a filtered colimit, $B\mathcal{J}$ is contractible. The result follows.
\end{proof}

Theorem \ref{thm:DK} gives an identification $\Alg_{\Ci}\left(\Spc\right) \simeq \dgci.$ It follow that the $\i$-category $\Loc$ has an alternative description, where the objects are pairs $\left(X,\O_X\right),$ with $O_X$ a homotopy sheaf of dg-$\Ci$-algebras. We will therefore interchangeably view homotopically $\Ci$-ringed spaces either as a space with a sheaf with values in the $\i$-category $\Alg_{\Ci}\left(\Spc\right),$ or a sheaf with values in the (equivalent) $\i$-category $\dgci$--- i.e. a homotopy sheaf of dg-$\Ci$-algebras. 

\begin{theorem}\label{thm:sch}
Let $\left(M,\O_\M,\D\right)$ be a dg-manifold. Then
$$\Speci\left(\Ci\left(\M\right),D\right) \simeq \left(\pi_0\left(\M\right),\O_\M|_{\pi_0\left(M\right)}\right).$$
\end{theorem}

\begin{proof}
Firstly, notice that $$\pi_0 \Gamma^{\Ci} \Ci\left(\M\right)\cong H^0\Ci\left(\M\right),$$ so the underlying classical $\Ci$-scheme $\pi_0\left(\M\right)=\Speci\left(H^0\Ci\left(\M\right)\right)$ agrees with the underlying classical $\Ci$-scheme $\pi_0\left(\Speci\left(\pi_0 \Gamma^{\Ci} \right)\right).$ In particular, their underlying topological spaces are homeomorphic. It therefore suffices to showing an equivalence of structure sheaves. It suffices to show they agree on sufficiently small $U,$ so we can therefore reduce to the case when  $\Ci\left(\M\right)$ is quasi-free. This implies that $\Gamma^{\Ci}\Ci\left(\M\right)$ is a compact object of $\Alg_{\Ci}\left(\Spc\right).$

An arbitrary open subset of $\pi_0\left(\M\right)$ is of the form $U=f^{-1}\left(\R\setminus\{0\}\right)\cap \pi_0\left(\M\right),$ for a smooth function $f:M \to \R.$ By \cite[Corollary 4.45]{univ}, it follows that $\O_{\Speci\left(\Ci\left(\M\right)\right)}\left(U\right)$ is the homotopical localization of $\Ci\left(\M\right)$ at the element $\left[f\right]\in H^0\Ci\left(\M\right).$ By Lemma \ref{lem:piok}, it therefore suffices to show that $\O_{\M}\left(U\right)$ is the homotopy localization of $\Ci\left(\M\right)$ at $f.$ By Proposition \ref{prop:locz}, it suffices to prove that the canonical map
$$H^n\left(\Ci\left(\M\right)\right) \underset{H^0\left(\Ci\left(\M\right)\right)} \otimes \left(H^0\Ci\left(\M\right)\right)\left[1/a\right] \to H^n\left(\O_{\M}\left(U\right)\right)$$ is an isomorphism for all $n,$ where $a$ is the image of $f$ in $0^{th}$ cohomology.

Notice that we have $1$-categorical pushout diagrams stacked together:
$$\xymatrix{\Ci\left(\R\right) \ar[r] \ar[d] & \Ci\left(\R\setminus \{0\}\right) \ar[d]\\
\Ci\left(M\right) \ar[d]_-{p} \ar[r] & \Ci\left(U\right) \ar[d]\\
H^0\Ci\left(\M\right) \ar[r] & \O_{\Speci\left(\Ci\left(\M\right)\right)}\left(U\right)}$$
Since $p$ is an effective epimorphism, we deduce by \cite[Lemma 4.6]{univ} that 
$$H^0\Ci\left(\M\right)\left[1/a\right] \cong H^0\Ci\left(\M\right) \underset{\Ci\left(M\right)} \otimes \Ci\left(U\right).$$
So we have $$H^n\left(\Ci\left(\M\right)\right) \underset{H^0\left(\Ci\left(\M\right)\right)} \otimes \left(H^0\Ci\left(\M\right)\right)\left[1/a\right] \cong H^n\left(\Ci\left(\M\right)\right) \underset{\Ci\left(M\right)} \otimes \Ci\left(U\right).$$
It therefore suffices to prove that $$H^n\left(\Ci\left(\M\right)\right) \underset{\Ci\left(M\right)} \otimes \Ci\left(U\right) \cong H^n \left(\O_\M\left(U\right)\right).$$
Denote by $\left(H^n\O_{\M}\right)$ the sheafificaton of presheaf
$$W \mapsto H^n\left(\O_{\M}\left(W\right)\right).$$
Consider the $\Ci\left(M\right)$-module $H^n\left(\Ci\left(M\right)\right).$ Since it is finitely presented, \cite[Example 5.28]{joycesch} implies that we can deduce that $$\left(H^n\O_{\M}\right)\left(W\right)\cong \Ci\left(W\right) \underset{\Ci\left(M\right)} \otimes H^n\left(\Ci\left(M\right)\right).$$
Notice that we have an exact sequence of $\O_M$-modules:
$$0 \to \ker\left(d_{n}\right) \to \O^{n}_{\M} \to \ker\left(d_{n+1}\right) \to \left(H^n\O_{\M}\right) \to 0.$$ It follows from \cite[Proposition 5.11]{joycesch} that 
$$0 \to \ker\left(d_{n}\left(U\right)\right) \to \O^{n}_{\M}\left(U\right) \to \ker\left(d_{n+1}\left(U\right)\right) \to \left(H^n\O_{\M}\right)\left(U\right)=\Ci\left(U\right) \underset{\Ci\left(M\right)} \otimes  H^n\left(\Ci\left(M\right)\right) \to 0$$
is also exact. The result now follows.
\end{proof}

\begin{corollary}
Regard a dg-manifold $\mathcal{M}=\left(M,\O_{\M}\right)$ as a topological space with a homotopy sheaf of dg-$\Ci$-algebras. Then there is a canonical map
$$\psi_\M:\Speci\left(\Ci\left(\M\right)\right)\simeq \left(\pi_0\left(\M\right),\O_\M|_{\pi_0\left(M\right)}\right) \to \mathcal{M}$$ of homotopically ringed spaces which induces a quasi-isomorphism on global sections.
\end{corollary}

\subsection{The category of fibrant objects structure}\label{sec:catfib}

As mentioned in the introduction, the authors of \cite{dgder} prove that there exists the structure of a category of fibrant objects on the category $\dgMan$ of non-positively graded dg-manifolds, or more precisely, on the equivalent category of bundles of curved $L_\i\left[1\right]$-algebras.

\subsubsection{Categories of fibrant objects}
\begin{definition}
A category of fibrant objects is a category $\sC$ with finite products and two classes of morphisms called the \emph{fibrations} and \emph{weak equivalences} respectively, satisfying the following axioms:

\begin{itemize}
\item[0.] Every object $C$ is fibrant (i.e. $C \to 1$ is a fibration).
\item[1.] Pullbacks against fibrations exist, and are again fibrations.
\item[2.] Acyclic fibrations (morphisms that are both fibrations and weak equivalences) are preserved under pullback.
\item[3.] The weak equivalences satisfy the $2$-out-of-$3$ property.
\item[4.] The following \textbf{factorization property} is satisfied: if $f:C \to D$ is any morphism, there exists a trivial fibration $\pi:C' \to C$ with a section $\lambda,$ and a fibration $f':C' \to D$ such that $f' \circ \lambda =f.$ 
\end{itemize}
\end{definition}

The tautological class of examples is the full subcategory of a model category on its fibrant objects, but there are many which are not of that form.

\begin{definition}
A pullback diagram in a category of fibrant objects in which one leg is a fibration is called a \textbf{homotopy pullback}.

Given any diagram
$$\xymatrix{ & C \ar[d]^-{f}\\
B \ar[r]^-{g} & D,}$$
there exists a trivial fibration $\pi:C' \to C$ with a section $\lambda,$ and a fibration $f':C' \to D$ such that $f' \circ \lambda =f.$ In such a situation, the
fibered product $B\times_{D} C'$ will be called a \textbf{homotopy pullback} of the above diagram.
\end{definition}

Given any category $\cC$ and a set $\mathcal{W}$ of morphisms in $\cC$ one can invert them up to homotopy in a universal way, to construct an $\i$-category $\cC\left[\cW^{-1}\right]_\i.$ There are many classical constructions for this (e.g. the Dwyer-Kan simplicial localization), all of which are equivalent, but the one which most clearly satisfies the desired universal property, is as follows:
Let $y:\cC \hookrightarrow \Psh_\i\left(\cC\right)$ denote the Yoneda embedding into the $\i$-category of presheaves of spaces. Let $y\left(\cW\right)$ be the image of the morphisms in $\cW.$ A presheaf $F:\cC^{op} \to \Spc$ is \textbf{$\cW$-local} if for all $w:C \to D$ in $\cW,$ $$F\left(w\right):F\left(D\right) \to F\left(C\right)$$ is an equivalence in $\Spc$ (i.e. a weak equivalence of Kan complexes.) Since $\cW$ is a set, the full subcategory spanned by the $\cW$-local presheaves are a reflective subcategory \cite[Proposition 5.5.4.15]{htt}, i.e. the inclusion admits a left adjoint $L_\cW.$

\begin{definition}
In the situation above, we let $\cC\left[\cW^{-1}\right]_\i:=L_\cW\left(\cC\right),$ the essential image of $\cC$ under the composite $L_{\cW}\circ y,$ and let $h_{\cW}:\cC \to \cC\left[\cW^{-1}\right]_\i$ be the canonical functor.
\end{definition}

Let $\cD$ be any $\i$-category, and let $\Fun^{\cW}\left(\cC,\cD\right)$ be the full subcategory of the $\i$-category of functors from $\cC$ to $\cD$ on those which send morphisms of $\cW$ to equivalences. Since the image of $\cW$ under the reflector $L_\cW$ consists of equivalences, given any functor $$F:\cC\left[\cW^{-1}\right]_\i \to \cD,$$ $F \circ h_{\cW}$ is in $\Fun^{\cW}\left(\cC,\cD\right).$ Hence there is an induced functor
$$h_{\cW}^*:\Fun\left(\cC\left[\cW^{-1}\right]_\i,\cD\right) \to \Fun^{\cW}\left(\cC,\cD\right).$$

The following proposition is very-well known:
\begin{proposition}
In the situation above, $h_{\cW}^*$ is an equivalence of $\i$-categories.
\end{proposition}

\begin{theorem}\label{thm:pbok}
If $\cW$ is the set of weak equivalences in a category of fibrant objects, then $h_{\cW}$ sends homotopy pullbacks to pullbacks in the $\i$-category $\cC\left[\cW^{-1}\right]_\i.$
\end{theorem}

\begin{proof}
This follows immediately by combining \cite[Proposition 3.6 and Proposition 3.23]{Denis}, since the Bousfield localization of the projective model structure on simplicial presheaves at the image of $\cW$ is the model category presenting the $\i$-category of $\cW$-local presheaves.
\end{proof}

\subsubsection{The category of fibrant objects structure on dg-manifolds.}

We will now review the category of fibrant objects structure that the authors of \cite{dgder} define on the category of bundles of curved $L_\i$-algebras, and then translate what it means into the language of dg-manifolds. We will also prove a conjecture from the same paper giving an alternative characterization of the weak equivalences. 

Given $\left(M,L,\lambda\right)$ a bundle of curved $L_\i$-algebras, they define the \emph{classical locus} to be the vanishing set of $\lambda_0,$ which is a section of $L_1.$ Recall that the first term of the differential $D_\lambda$ on the associated dg-manifold $\left(\mathbb{L},D_{\lambda}\right)$ is given by $\iota_{\lambda_0},$ which, if $\xi_1,\ldots,\xi_k$ are the local coordinates around a point $p$ of degree $-1$ corresponding to a local trivialization of $L_1$ (of rank $k$) is the term $\left(\lambda_0\right)_i\frac{\partial}{\partial \xi_i}.$ Hence we see that this agrees with the Definition \ref{dfn:locus}, and we have that the vanishing locus is $\pi_0\left(\mathbb{L},D_{\lambda}\right),$ except the authors of \cite{dgder} only consider the underlying set of points, rather than that structure of a $\Ci$-scheme.

\begin{definition}\label{dfn:catfib}
A morphism $f:\M=\left(M,L,\lambda\right) \to \left(N,E,\mu\right)=\cN$ is a \textbf{fibration} if:
\begin{itemize}
\item[i)] $f:M \to N$ is a submersion of smooth manifolds.
\item[ii)] $\phi_1:L \to E$ is a submersion of smooth manifolds.
\end{itemize}
A morphism $f:\M \to \cN$ is a \textbf{weak equivalence} if
\begin{itemize}
\item[i)] $f$ induces a bijection of underlying sets on classical loci.
\item[ii)] For each point $p \in \pi_0\left(\M\right),$
the induced map on tangent complexes $T_p \M \to T_{f(p)} \cN$ is a quasi isomorphism.
\end{itemize}
See \cite[Definitions 1.9, 1.12, 1.21]{dgder}.
\end{definition}

Firstly, suppose that $f:\M=\left(M,L,\lambda\right) \to \left(N,E,\mu\right)=\cN$ is a fibration. Denote by $\End\left(\R^{0|1}\right) \acts \mathfrak{M}$ the associated equivariant supermanifold. We have that
$\mathfrak{M}$ is the total space of the vector bundle
$$\bigoplus\limits_{n=1}^{\i} L_{2n} \oplus \bigoplus\limits_{n=1}^{\i} \Pi L_{2n+1}.$$ From here, it is easy to see that the condition translates into the induced map $\mathfrak{M} \to \mathfrak{N}$ between associated supermanifolds is a submersion of supermanifolds.

Secondly, given a bundle of curved $L_\i\left[1\right]$-algebras $\left(M,L,\lambda\right)$, and a point $p$ in the classical locus, the authors of \cite{dgder} define the \emph{tangent complex} at $p$ to be the complex
$$TM_p \stackrel{D_p \lambda_0}{\longlongrightarrow} L^1_p \stackrel{\lambda_1|_p}{\longlongrightarrow} L^2_p \stackrel{\lambda_1|_p}{\longlongrightarrow} \ldots$$
Note that, in the world of graded manifolds we have
$$T_p \mathbb{L} \cong T_p M \oplus \mathbb{L}_p \cong T_pM \oplus \underset{i} \bigoplus \left(L_i\right)_p\left[-i\right],$$ which becomes the complex
$$TM_p \stackrel{\mathcal{L}_D(p)}{\longlongrightarrow} L^1_p \stackrel{\mathcal{L}_D(p)}{\longlongrightarrow} L^2_p \stackrel{\mathcal{L}_D(p)}{\longlongrightarrow} \ldots$$
It a matter of simple calculation that the differential agrees. 

Here is an alternative description. Let $\mathfrak{M}$ be the corresponding supermanifold with an action of $\End\left(\R^{0|1}\right),$ and denote by $\hat D$ the canonical extension of $D.$ $M$ is the fixed points of the $\left(\R,\cdot\right)$-action and $\pi_0\left(\M\right)$ is the fixed points of the $\R^{0|1}$-action, i.e. the odd-flow of $\hat D.$ For any $x \in M,$ since $x$ is a fixed point for the $\left(\R,\cdot\right)$-action, there is an induced linear action of $\left(\R,\cdot\right)$ on $T_x \mathfrak{M}.$ This induces an $\N$-grading on $T_x \mathfrak{M},$ which since the action is even, is a lift of the $\Z_2$-grading coming from the supermanifold $\mathfrak{M}.$ Explicitly, a tangent vector $v \in T_x \mathfrak{M}$ is of degree $n$ if $t \star v=t^n \cdot v,$ for all $t \in \R,$ where $\star$ denotes the linearized action. Equivalently, $$\frac{d}{dt}|_{t=0} t \star v=n\cdot v.$$ If $x$ is additionally in $\pi_0\left(\M\right),$ then $\hat D_x=0.$ $\hat D,$ or rather its odd-flow, therefore induces a linear map $d_x:T_x \mathfrak{M} \to \Pi T_x \mathfrak{M}.$ But this exactly $\left(\mathcal{L}_D\right)_x.$ Moreover, with respect to the grading induced by $E,$ $\mathcal{L}_D$ becomes a degree $1$-map, and since $\left[\hat D,\hat D\right]=0,$ $\mathcal{L}_D$ becomes a differential. In summary, $\left(T_x \mathfrak{M},\left(\mathcal{L}_D\right)_x\right)$ becomes a cochain complex concentrated in non-negative degrees. It follows that $$T\mathfrak{M}|{\pi_0\left(\M\right)}$$ is a complex of topological vector bundles\footnote{More precisely, has a canonically associated complex of topological vector bundles.} over the topological space $\pi_0\left(\M\right).$ So condition $ii)$ of a weak equivalence states that  $$T\mathfrak{M}|{\pi_0\left(\M\right)} \to f^*T\mathfrak{N}|{\pi_0\left(\N\right)}$$ is a fiber-wise quasi-isomorphism of this complex of topological vector bundles.

Consider the inclusions
$$\pi_0\left(\M\right) \stackrel{j}{\longhookrightarrow} M \stackrel{i_0}{\longhookrightarrow} \mathfrak{M}.$$ 
Then one gets a canonical identification of $\left(i_0\right)^* T \mathfrak{M}$ with the tangent sheaf $\mathcal{T}\M$ of the $dg$-manifold $\M.$ 

\begin{proposition}\label{prop:tansame}
Recall the canonical map
$$\psi_\M:\Speci\left(\Ci\left(\M\right)\right)\simeq \left(\pi_0\left(\M\right),\O_\M|_{\pi_0\left(M\right)}\right) \to \mathcal{M}$$ There is a canonical identification $\psi_\M^*\mathcal{T} \M \simeq \mathcal{T}\left(\Speci\left(\Ci\left(\M\right)\right)\right),$ 
where $\mathcal{T}\left(\Speci\left(\Ci\left(\M\right)\right)\right)$ is the dual of the cotangent complex of the affine derived $\Ci$-scheme $\Speci\left(\Ci\left(\M\right)\right),$ as defined in \cite[Chapter 5]{pelle}.
\end{proposition}

\begin{proof}
As this is a question about sheaves, it suffices to prove it locally, so we can reduce to the case that $\Ci\left(\M\right)$ is quasi-free. Passing to local coordinates, this follows from the concrete description of the cotangent complex of a quasi-free dg-$\Ci$-algebra in \cite[Remark 5.1.018]{pelle}.
\end{proof}

\begin{lemma}\label{lem:ptwise}\cite{dimatalk}
Let $F:\left(V^\bullet,d\right) \to \left(W^\bullet,d\right)$ be morphism of bounded cochain complexes of finite dimensional real vector bundles over a paracompact Hausdorff space $X.$ Suppose that $F$ induces a fiber-wise quasi-isomorphism. Then $F$ induces a quasi-isomorphism on global sections.
\end{lemma}

\begin{proof}
Firstly, by passing to the mapping cone of $F,$ one can reduce to proving that if $\left(V^\bullet,d\right)$ is fiber-wise acyclic, $\Gamma\left(V^\bullet\right)$ is an acyclic complex. Let $x \in X.$ Choose some $k \in \Z,$ and let $V^k$ and $V^{k-1}$ be of rank $n$ and $m$ respectively. Let $x \in X.$ Then
$$\dim\left(V^k_x\right)=\dim\left(\ker\left(d_k\left(x\right)\right)\right) + \dim\left(\mbox{Im}\left(d_k\left(x\right)\right)\right).$$
But the dimension of the image is the rank, and since $\left(V_x^\bullet,d_x\right)$ is acyclic, $\ker d_k\left(x\right)=\mbox{Im}\left(d_{k-1}\left(x\right)\right).$ 
Hence 
\begin{equation}\label{eq:rank}
\mbox{rank}\left(d_k\left(x\right)\right)=n-\mbox{rank}\left(d_{k-1}\left(x\right)\right).
\end{equation}
Choose an open set $U \ni x$ over which both $V^k$ and and $V^{k-1}$ trivialize. Then we can identify $d_k$ and $d_{k-1}$ with functions $$U \to \End\left(\R^n\right)$$ and $$U \to \End\left(\R^m\right).$$ Recall that for any $n,$ the function $$\mbox{rank}:\End\left(\R^n\right) \to \N$$ is lower-semi-continuous. From this, combined with the equality (\ref{eq:rank}), it follows that $\mbox{rank}\circ d_k$ is also upper-semi continuous, and hence continuous, and therefore locally constant. So the dimension of the image of $d_k$ is locally constant. It follows that for all $k,$ 
$$B_k:=\mbox{Im}\left(d_{k-1}\right)$$ is a vector bundle. Since $\left(V^\bullet,d\right)$ is fiber-wise acyclic, we have a short exact sequence of cochain complexes of vector bundles
$$\xymatrix{\tvdots \ar[d] & \tvdots \ar[d]_-{0} & \tvdots \ar[d]_-{d_{k-1}}  & \tvdots \ar[d]_-{0} & \tvdots \ar[d] \\
0 \ar[d] \ar[r] & B^k=\ker\left(d_k\right) \ar[d]_-{0} \ar[r] & V^k \ar[d]_-{d_k} \ar[r]^-{d_k} & B^{k+1} \ar[r] \ar[d]_-{0} & 0 \ar[d]\\
0 \ar[d] \ar[r] & B^{k+1} \ar[r] \ar[d]_-{0} & V^{k+1} \ar[r]_-{d_{k+1}}  \ar[d]_-{d_{k+1}} & B^{k+2} \ar[r] \ar[d]_-{0} & 0 \ar[d] \\
\tvdots & \tvdots & \tvdots & \tvdots & \tvdots }$$
Since $X$ is paracompact Hausdorff, and all the complexes are bounded, it follows from \cite[Lemma 17.13]{duggerK} that there exists a splitting as a cochain complex of vector bundles. Hence one can choose a non-canonical isomorphism of $\left(V^\bullet,d\right)$ with the mapping cone of
$$id:\left(B^\bullet,0\right) \to \left(B^\bullet,0\right),$$ and one concludes it is acyclic.
\end{proof}

\begin{corollary}\label{cor:cotcpx}
Let $f:\left(M,L,\lambda\right) \to \left(N,E,\mu\right)$ be a weak equivalence of bundles of curved $L_\i\left[1\right]$-algebras, and let $\M$ and $\cN$ be the associated dg-manifolds with associated $\End\left(\R^{0|1}\right)$-equivariant supermanifolds $\mathfrak{M}$ and $\mathfrak{N}$ respectively. Then the relative cotangent complex of the induced map $$\Speci\left(\Ci\left(\M\right)\right) \to \Speci\left(\Ci\left(\cN\right)\right)$$ of derived $\Ci$-schemes vanishes.
\end{corollary}

\begin{proof}
By dualizing, it suffices to show that
$$\mathcal{T}\Speci\left(\Ci\left(\M\right)\right) \to f^*\Speci\left(\Ci\left(\cN\right)\right)$$ is quasi-isomorphism. By Proposition \ref{prop:tansame}, it suffices to show that
$$T\mathfrak{M}|{\pi_0\left(\M\right)} \to f^*T\mathfrak{N}|{\pi_0\left(\cN\right)}$$ is a quasi-isomorphism of complexes of vector bundles over the topological space $\pi_0\left(\M\right).$ It is so fiber-wise so, and $\pi_0\left(\M\right)$ is a closed subspace of $M,$ hence paracompact Hausdorff. Moreover, both complexes are bounded as both supermanifolds are finite dimensional. The result now follows from Lemma \ref{lem:ptwise}.
\end{proof}

Recall that given any supermanifold $\mathfrak{M},$ one has a canonical identification $\mathfrak{M}^{\R^{0|1}}\cong \Pi T \mathfrak{M}.$ There is hence a canonical action of of $\End\left(\R^{0|1}\right)^{op}.$ Notice the $\emph{op},$ that is, this is now a right-action of the monoid as it is given by precomposition. Since it is a group, there is a canonical isomorphism $\Aut\left(\R^{0|1}\right) \cong  \Aut\left(\R^{0|1}\right)^{op},$ given by taken inverses. The affect on infinitesimal generators is by changing their sign. If $\varepsilon'=-\varepsilon$ and $\delta'=-\delta,$ then we get $\left[\varepsilon',\delta'\right]=-\delta',$ meaning the differential points in the opposite direction as before! It is a well-known and celebrated fact that the associated dg-$\Ci$-algebra is differential forms with the de Rham differential being the cohomological vector field \cite{SevGorms}. In slightly more detail, if $\mathfrak{M}=\R^{p|q},$ with coordinates
$\left(x_1,\ldots,x_p;\xi_1,\ldots,\xi_q\right),$ then 
$\Pi T \mathfrak{M}$ has coordinates $\left(x_1,\ldots,x_p,\bar \xi_1,\ldots, \bar \xi_q;\xi_1,\ldots,\xi_q,\bar x_1,\ldots,\bar x_p\right),$ where the barred coordinates have opposite parity. The cohomological vector field induced by the canonical $\End\left(\R^{0|1}\right)^{op}$-action takes the form
$$d=\bar x_i \frac{\partial}{\partial x_i} + \bar \xi_j \frac{\partial}{\partial \xi_j}.$$ To identify this with the de Rham differential, make the substitution $dx_i:=\bar x_i,$ and $d\xi_i:= \bar \xi.$

\begin{theorem}\label{thm:normal}\textbf{Morse lemma for cohomological vector fields.}\cite{Vaintrob,normal2}
Let $\mathfrak{M}$ be a supermanifold, with cohomological vector field $Q,$ and $p$ be a point in the core for which $Q_p=0.$ Then there exists a neighborhood $U$ around $p$ with corresponding open supersubmanifold $\mathfrak{U}=\left(U,\O_{\mathfrak{M}}\right),$ and a pointed diffeomorphism intertwining cohomological vector fields
$$\varphi:\left(\mathfrak{U},Q,p\right) \stackrel{\sim}{\longrightarrow}\left(\left(\Pi T \R^{p|q},d\right)\times \left(V,D\right),\left(0,0\right)\right)$$
where $d$ is the de Rham differential, $V$ is a superdomain, and $\mathcal{L}_D\left(p\right)=0.$ Moreover, if $Q$ arises from an $\End\left(\R^{0|1}\right)$-action on $\mathfrak{M},$ this can be made compatible with the grading induced by the Euler field.
\end{theorem}

\begin{corollary}\label{cor:schiso}\cite{dimatalk}
Let $f:\left(M,L,\lambda\right) \to \left(N,E,\mu\right)$ be a weak equivalence of bundles of curved $L_\i\left[1\right]$-algebras. Then the induced map
$$\pi_0\left(\left(M,L,\lambda\right)\right) \to \pi_0\left(\left(N,E,\mu\right)\right)$$ is an isomorphism \underline{of $\Ci$-schemes}.
\end{corollary}

\begin{proof}
Let $\left(\mathfrak{M},Q\right)$ and $\left(\mathfrak{N},Q'\right)$ be the corresponding supermanifolds with cohomological vector fields. Let $p \in \pi_0\left(\left(M,L,\lambda\right)\right).$ Then we can identify $p$ with a point in $|\mathfrak{M}|$ for which $Q_p=0.$ Since $Q'$ is $f$-related to $Q,$ we also have $Q'_{f\left(p\right)}=0.$ By Theorem \ref{thm:normal}, we can find 
open subsupermanifolds $\mathfrak{U}$ and $\mathfrak{U}'$ around $p$ and $f\left(p\right)$ and pointed diffeomorphisms intertwining cohomological vector fields
$$\xymatrix@C=2.5cm@R=1.5cm{\left(\mathfrak{U},Q\right) \ar[d]_-{\mis} \ar[r]^{f|_{\mathfrak{U}}} & \left(\mathfrak{U}',Q'\right) \ar[d]^-{\mis}\\
\left(\left(\Pi T \R^{p|q},d\right)\times \left(V,D\right),\left(0,0\right)\right) \ar[r]^-{\mbox{$\underset{{\rotatebox[origin=c]{270}{$:=$}}} \varphi$}} & \left(\left(\Pi T \R^{p'|q'},d'\right)\times \left(V',D'\right),\left(0,0\right)\right),}$$
for which $\mathcal{L}_D\left(p\right)$ and $\mathcal{L}_{D'}\left(f\left(p\right)\right)=0.$
Consider the composite
$$\xymatrix{\left(\Pi T \R^{p|q},d,0\right) \times \left(V,D,0\right) \ar[r]^-{\varphi} & \left(\Pi T \R^{p'|q'},d,0\right) \times \left(V',D',0\right) \ar[d]^-{pr_{V'}}\\
\left(V,D,0\right) \ar[u]^-{\{0\} \times id} & \left(V',D',0\right),}$$
and denote this by $\theta.$
The vertical arrows are quasi-isomorphisms since $\left(\Ci\left(\Pi T \R^{p|q}\right),d\right)\simeq \R.$ It follows that the composite induces a quasi-isomorphism on tangent complexes at $0.$ But the differential of the tangent complexes are $\mathcal{L}_D$ and $\mathcal{L}_D',$ which vanish. It follows that $\theta$ induces an \emph{isomorphism} on tangent spaces of supermanifolds
$$\theta_*:T_0 V \to T_0 V',$$ so shrinking $V'$ if necessary, we can assume that it is a diffeomorphism. Being quasi-isomorphisms, the vertical maps also induced isomorphisms of $\Ci$-rings on $H^0,$ and hence isomorphisms of of classical loci \emph{as $\Ci$-schemes}. Since $\theta$ is a diffeomorphism, it follows that $\varphi$ induces an isomorphism of $\Ci$-schemes between classical loci. It follows that $\pi_0\left(f\right)$ is locally an isomorphism of $\Ci$-schemes around any point in $\pi_0\left(\mathfrak{M}\right).$ Since it is also a bijection of sets, it must be globally an isomorphism.
\end{proof}

\begin{theorem}\label{thm:weqi}
Let $f:\left(M,L,\lambda\right) \to \left(N,E,\mu\right)$ a morphism of bundles of curved $L_\i\left[1\right]$-algebras with corresponding dg-manifolds $\left(\M,D\right)$ and $\left(\cN,D'\right)$ respectively. Then $f$ is a weak equivalence if and only if the induced map $\Ci\left(\cN\right) \to \Ci\left(\M\right)$ is a quasi-isomorphism.
\end{theorem}

\begin{proof}
Suppose that $f$ is a weak equivalence. Since $\pi_0\left(f\right)$ is surjective, by \cite[Proposition 5.1.0.16]{pelle}, we can identify the relative cotangent complex of
$$\Ci\left(\cN\right) \to \Ci\left(\M\right)$$ 
with its relative cotangent complex as a map of commutative dg-algebras over $\R,$ i.e. its \emph{algebraic} relative cotangent complex. Hence, by Corollary \ref{cor:cotcpx}, its algebraic cotangent complex vanishes. Since $H^0\left(\Ci\left(\M\right)\right)$ and $H^0\left(\Ci\left(\cN\right)\right)$ are finitely presented $\Ci$-algebras, it follows from Corollary \ref{cor:schiso} that the induced homomorphism between these $\Ci$-algebras is an isomorphism. It now follows from \cite[Corollary 7.4.3.4]{LurieHA} that the induced map 
$$\Ci\left(\cN\right) \to \Ci\left(\M\right)$$ 
is a quasi-isomorphism.

Conversely, suppose that the above map is a quasi-isomorphism. Then the induced map 
$$\Speci\left(\Ci\left(\M\right)\right) \to \Speci\left(\Ci\left(\M\right)\right)$$ is an equivalence of derived $\Ci$-schemes. In particular, it a bijection on the level of underlying sets, so $$\pi_0\left(\M\right) \to \pi_0\left(\cN\right)$$ is a bijection, and the relative cotangent complex vanishes. So, by Proposition \ref{prop:tansame}, we conclude that the map on point-wise tangent complexes is a quasi-isomorphism.
\end{proof}

\subsection{Derived fibered products}\label{sec:fibprod}
The central construction of study in derived differential geometry is derived fibered products. The category of fibrant objects structure on the category $\dgMan$ of dg-manifolds allows one to make this precise, by stating

\begin{center}
\emph{derived fibered products}   =   \emph{homotopy pullbacks of dg-manifolds}
\end{center}

We will not present the whole theory of how to construct derived fibered products here. (There is for example a construction explained in \cite{dgder} using ideas from AKSZ theory). However, there is special type of homotopy pullback of dg-manifolds that is very easy to describe. Suppose that $\pi:\mathbb{E} \to \M$ is a graded vector bundle over a graded manifold (both equipped with trivial cohomological vector field), and that $s$ is a section. We will explain how to construct the derived intersection of $s$ with the zero-section. The most important case of this is of course the case of a section of a classical vector bundle over a smooth manifold.

Consider the underlying graded manifold of the shifted vector bundle $\mathbb{E}\left[-1\right].$ Let's see what this looks like on a local trivialization. Then 
a section can be identified with a smooth map $\M \to \mathbb{V}.$ Since $\Ci\left(\mathbb{V}\right)$ is the free graded $\Ci$-algebra on $\mathbb{V}^*,$ this is the same data as collection of elements $s^{(k)}_i \in \Ci\left(\M\right),$ each $s^{(k)}_i$ being of degree $-k,$ with $i$ running from $1$ to $\dim V_k.$ Choose for each $k$ a basis $\xi^{(k)}_1,\ldots,\xi^{(k)}_{n_k}$ of $\left(V_k\right)^*$ so that we can identify 
$\Ci\left(\mathbb{E}\left[-1\right]\right)$ as the graded polynomial algebra
$$\Ci\left(\M\right)\left[\xi^{(k)}_{i}\right],$$ with each $\xi^{(k)}_{i}$ of degree $-k-1.$ On one hand, the expression $$s^{(k)}_i \frac{\partial}{\partial \xi^{(k)}_{i}}$$ is a cohomological vector field on $\mathbb{E}\left[-1\right],$ but on the other hand it's actually just the Koszul differential: $$K\left(\Ci\left(\M\right),s^{(k)}_i\right)=\left(\Ci\left(\M\right)\left[\xi^{(k)}_{i}\right],s^{(k)}_i \frac{\partial}{\partial \xi^{(k)}_{i}}\right).$$ This cohomological vector field however is \emph{globally defined}. This is easiest to see in the case of a classical vector bundle $E \to M$ over a manifold. Then, $$\O^{-\bullet}_{E\left[-1\right]}=\Gamma\left(\wedge^\bullet E^*\right).$$
The cohomological vector field is the derivation $\iota_s$ induced by insertion of $s$ into a wedge product of sections of $E^*.$ For the general case, since the degree $0$ fiber of $\mathbb{E}\left[-1\right]$ is $0$-dimensional, one has $$\O_{\mathbb{E}}\cong \Sym_{\O_{\M}}\left(\underline{\Gamma}\left(\mathbb{E}^*\right)\right),$$ so one can again use insertion.

\begin{proposition}\label{prop:zpb}
In the situation above $\left(\mathbb{E}\left[-1\right],\iota_{s}\right)$ is the derived fibered product of $s$ with the zero section.
\end{proposition}

\begin{proof}
Let $\mathbb{E} \to \M$ be any graded vector bundle, and then consider the pullback bundle $\pi^*\mathbb{E} \to \mathbb{E}.$ Since as a graded manifold, $\pi^*\mathbb{E} \cong \mathbb{E} \times_{\M} \mathbb{E},$ there is a canonical section $\Delta$ of $\pi^*\mathbb{E}$ induced by the identify map of $\mathbb{E}.$ Notice that if $s$ is a section of $\mathbb{E},$ that the following is a pullback diagram
$$\xymatrix{\left(\mathbb{E}\left[-1\right],\iota_{s}\right) \ar[r] \ar[d] & \left(\pi^*\mathbb{E}\left[-1\right],\iota_{\Delta}\right) \ar[d]\\
\M \ar[r]_-{s} & \mathbb{E}.}$$
Denote by $Z$ the zero section of $\mathbb{E},$ and by $Z_{-1}$ the zero section of $\pi^*\mathbb{E}\left[-1\right].$ Then we can factor
$$Z=\pi_{-1} \circ \left(Z_{-1} \circ Z\right),$$ with $\pi_{-1}$ the bundle projection. We claim that $Z_{-1} \circ Z$ is a weak equivalence. Firstly, notice that the vanishing locus of $\iota_{\Delta}$ is exactly the image of $|\M|=:M$ under $Z_{-1} \circ Z,$ so the induced map on classical loci is a bijection. It suffices to prove that the map locally induces a quasi-isomorphism of structure sheaves. It suffices to work locally, and assume that $\mathbb{E}$ is trivial. If $e_1,\ldots,e_m$ are (dual to) a local frame for $\mathbb{E}$ (of varying degree), introduce $\overline{e}_1,\ldots,\overline{e}_m$ as the corresponding frame for $\mathbb{E}\left[-1\right],$ so $|\overline{e}_i|=|e_1|-1,$ then $\iota_{\Delta}$ takes the form $e_i \frac{\partial}{\partial \overline{e}_i}.$ So we can identify the value of the structure sheaf locally with the Koszul algebra  of $\Ci\left(\M\right)\{e_1,\ldots,e_m\}$ with respect to the regular sequence $e_1,\ldots, e_m,$ and hence by Proposition \ref{prop:kosz}, it is quasi-isomorphic to $\Ci\left(\M\right).$ Since $Z_{-1} \circ Z$ is a section of the composite fibration $$\pi^*\mathbb{E}\left[-1\right] \to \mathbb{E} \to \M,$$ and $Z_{-1} \circ Z$ is a section, the above fibration is a trivial fibration. It follows that the pullback diagram above is a homotopy pullback for the diagram
$$\xymatrix{ & \M \ar[d]^-{Z}\\
\M\ar[r]^-{s} & \mathbb{E}.}$$
\end{proof}

Suppose that $E \to M$ is a vector bundle over a manifold with a section $s.$ Consider the special case where $s$ is the zero-section itself. Then $\iota_s=0.$ This means the derived intersection of the zero-section with itself is $E\left[-1\right],$ as graded manifold with $0$ as its differential. Because the differential is zero, the zero section of $E\left[-1\right]$ is a dg-map, and we can continue the process and take the derived intersection of the zero section with itself and get $E\left[-2\right],$ and so on. Moreover, if $E_1$ and $E_2$ are vector bundles, then, as a graded manifolds (not as a vector bundles), we have $$E_1\left[-1\right] \oplus E_2\left[-2\right]\cong E_1\left[-1\right] \times_M E_2\left[-2\right].$$ Continuing in this way, in light of Batchelor's theorem, we can conclude that every graded manifold can be obtained through homotopy pullbacks from manifolds.

\begin{remark}
A special case of the above is when $M=*,$ and $E$ is just a vector space. The zero section then is just the inclusion of the zero vector $* \to E.$ So the derived fibered product $* \times^h_E *=E\left[-1\right].$ One may also consider this a \emph{derived loop space.} Here, we see that on vector space, the derived loop space is the loop functor $\Omega=\left(\blank\right)\left[-1\right],$ (and more generally for graded vector spaces.)
\end{remark}

\begin{definition}
A \textbf{dg-vector bundle} over a dg-manifold $\left(\M,D\right)$ is a graded vector bundle $\mathbb{E} \to \M$ equipped with a cohomological vector field $Q_E$ making the bundle projection a dg-map.
\end{definition}

\begin{definition}
Let $\M$ be a graded manifold. By Batchelor's theorem for graded manifolds, Theorem \ref{thm:BatchN}, there are vector bundles $E_1,E_2,\ldots,E_m$ over the core $M=|\M|$ such that $$\M \cong \bigoplus\limits_{i=1}^m E_i\left[-i\right],$$ 
where each $\pi_k:E_k \to M$ is a vector bundle of positive rank. The integer $m$ is called the \text{amplitude} of $\M$. (Traditionally, one says $\M$ is of amplitude $\left[-m,0\right].$)
\end{definition}

\begin{remark}
A dg-manifold of amplitude $0$ is just a manifold.
\end{remark}

\begin{theorem}\label{thm:decomp}
Given any dg-manifold $\M$ of amplitude $m,$ there exists a canonical dg-manifold $\cN$ of amplitude $m-1,$ a dg-vector bundle $\mathbb{W} \to \cN$ with $\mathbb{W}$ also of amplitude $m-1$ and a dg-section $\lambda,$ such that $\M$ is the derived fibered product of $\lambda$ and the zero-section.

In particular, every dg-manifold can be constructed from manifolds through a finite iteration of derived fibered products.
\end{theorem}

\begin{proof}
To fix notation, by Batchelor's theorem there is an isomorphism of graded manifolds
$$\xymatrix{\psi:\M \ar[r]^-{\sim} \ar[d] & \mathbb{E}=\bigoplus\limits_{i=1}^m E_i\left[-i\right] \ar[ld]^-{\pi}\\
M.}$$
where each $\pi_k:E_k \to M$ is a vector bundle of rank $n_k>0.$ 

Clearly $\left(\M,D\right) \cong \left(\mathbb{E},\psi_*D\right).$ Write $Q=\psi_*D.$ For all $k=0,\ldots,m,$ let $\mathbb{E}_{\le k}$ be the graded vector bundle
$$\bigoplus\limits_{i=1}^k E_i\left[-i\right] \stackrel{\pi_{\le k}}{\longlongrightarrow} M.$$
Denote its underlying graded manifold by $\M_k$. Notice that there is a canonical isomorphism of underlying graded manifolds 
$$\M_{k+1} \cong \pi_{\le k}^* E_{k+1}\left[-k-1\right]$$ (although the vector bundle structures are different), under which the canonical map of graded vector bundles over $M$
$$\mathbb{E}_{\le k+1} \to \mathbb{E}_{ \le k},$$
regarded only as a smooth map of graded manifolds, can be identified with $$\pi_{\le j}^* E_{k+1}\left[-k-1\right] \to \M_{k}.$$
Since $Q$ must increase degrees of functions by $1,$ it follows that for all $k,$ $Q$ restricts to a cohomological vector field $Q_k$ on $\M_k,$ 
producing a tower of fibrations of dg-manifolds (corresponding to submersions of corresponding supermanifolds)
$$\left(\M,D\right) \cong \left(\M_{m},Q_m\right) \stackrel{p_m}{\longlongrightarrow} \left(\M_{m-1},Q_{m-1}\right) \stackrel{p_{m-1}}{\longlongrightarrow} \ldots \left(\M_1,Q_1\right) \stackrel{\pi_1}{\longrightarrow} \left(\M_0,Q_0\right)=\left(M,0\right).$$

Choose an open cover $\left(U_\alpha \hookrightarrow M\right)_\alpha$ over which each $E_k$ trivializes as $$\varphi^{(k)}_\alpha:E_k|_{U_\alpha} \stackrel{\sim}{\longrightarrow} U_\alpha \times V_k,$$ for some fixed vector space $V_k.$ For all $k,$ let $$\mathbb{V}_{\le k}:=\bigoplus\limits_{i=1}^k V_i\left[-i\right],$$ and denote by $$\varphi_\alpha:\mathbb{E}|_{U_\alpha} \stackrel{\sim}{\longrightarrow} U_\alpha \times \mathbb{V}\mbox{ and }\varphi^{\le k}_\alpha:\mathbb{E}_{\le k}|_{U_\alpha} \stackrel{\sim}{\longrightarrow} U_\alpha \times \mathbb{V}_{\le k}$$ the induced trivializations. 

For each $V_k$ choose a basis $\xi^{(k)}_1,\ldots,\xi^{(k)}_{n_k}$ of $\left(V_k\right)^*.$ Then in the induced coordinate system on $U_\alpha \times \mathbb{V},$
$$\left(\varphi_\alpha\right)_*Q=f^{\left(-k+1\right)}_{j,\alpha} \frac{\partial}{\partial \xi^{(k)}_j},$$ with $f^{\left(-k+1\right)}_{j,\alpha}$  a $\left(\varphi_\alpha\right)_*Q$-exact element of degree $-k+1$ of $$\Ci\left(U_\alpha \times \mathbb{V}\right)\cong \Ci\left(U_\alpha\right) \underset{\R} \otimes \Sym\left(\mathbb{V^*}\right).$$
Fixing $k,$ notice that for all $j=1,\ldots,n_{k+1}$ $$f^{\left(-k\right)}_{j,\alpha}=\left(\varphi^{\le k+1}_\alpha\right)_*Q_{k+1}\left(\xi^{k+1}_j\right),$$ is of degree $-k,$ so it cannot depend on the fiber coordinates of $p_{k+1}.$ Putting these together forms a map of graded manifolds $$f^{\left(k\right)}_\alpha:U_{\alpha} \times \mathbb{V}_{\le k} \to V_{k+1}\left[-k\right],$$ which since each component of $f^{\left(k\right)}_\alpha$ is $Q_{k+1}$-closed, becomes a map of dg-manifolds
$$\left(U_{\alpha} \times \mathbb{V}_{\le k},\left(\varphi^{\le k}_\alpha\right)_* Q_k\right) \to \left(V_{k+1}\left[-k\right],0\right),$$
and hence a section of the trivial dg-vector bundle $$\left(U_{\alpha} \times \mathbb{V}_{\le k}\right) \times V_{k+1}\left[-k\right] \to \left(U_{\alpha} \times \mathbb{V}_{\le k}\right).$$ Denote by $\lambda_\alpha$ the canonically induced section of the restriction of $pr_{k+1}^*E_{k+1}\left[-k\right]$ to the graded open submanifold $\mathbb{E}_{\le k}|_{U_\alpha}$ of $\M_{k}.$ Let 
$$\mathcal{U}^k_{\alpha\beta}:=\mathbb{E}_{\le k}|_{U_\alpha} \cap \mathbb{E}_{\le k}|_{U_\beta}.$$
We claim that
\begin{equation}\label{eq:glue}
\lambda_\alpha|_{\mathcal{U}^k_{\alpha\beta}}=\lambda_{\beta}|_{\mathcal{U}^k_{\alpha\beta}}.
\end{equation}
Since $E_{k+1}$ is a vector bundle over $M,$ we can write the induced transition functions in the form
\begin{eqnarray*}
U_{\alpha\beta} \times V_{k+1} &\stackrel{\left(\varphi^{\left(k\right)}_{\beta}\right)^{-1}}{\longlonglongrightarrow} E_{k+1}|_{U_{\alpha\beta}} \stackrel{{\varphi^{\left(k\right)}_{\alpha}}}{\longlonglongrightarrow}& U_{\alpha\beta} \times V_{k+1} \\
\left(x,v\right) &\mapsto& \left(x,g^{(k+1)}_{\alpha\beta}\left(x\right)v\right).
\end{eqnarray*}
To verify (\ref{eq:glue}), it suffices to prove that
$$f_{i,\alpha}^{\left(k\right)}=\left(g^{(k+1)}_{\alpha\beta}\right)_{il}f_{l,\beta}^{\left(k\right)}.$$
Firstly, note that
\begin{eqnarray*}
f^{\left(-k\right)}_{j,\alpha}&=&\left(\varphi^{\le k+1}_\alpha\right)_*Q_{k+1}\left(\xi^{k+1}_j\right)\\
&=& Q_{k+1}\left(\xi_{i}^{(k+1)} \circ \varphi^{\le k+1}_\alpha\right)\\
&=& Q_{k+1}\left(\xi_{i}^{(k+1)} \circ \varphi^{\le k+1}_\alpha \circ \left(\varphi^{\le k+1}_\beta\right)^{-1} \circ \varphi^{\le k+1}_\beta\right)\\
&=& Q_{k+1}\left(\xi_{i}^{(k+1)} \circ \varphi^{\le k+1}_{\alpha\beta} \circ \varphi^{\le k+1}_\beta\right)\\
&=& \left(\varphi^{\le k+1}_\beta\right)_*Q_{k+1}\left(\xi_{i}^{(k+1)} \circ \varphi^{\le k+1}_{\alpha\beta}\right)\\
&=& \left(\varphi^{\le k+1}_\beta\right)_*Q_{k+1}\left(\left(g^{(k+1)}_{\alpha\beta}\right)_{il}\xi^{(k+1)}_l\right).
\end{eqnarray*}
Notice that $g^{(k+1)}_{\alpha\beta}$ are functions on $M,$ so $Q\left(g^{(k+1)}_{\alpha\beta}\right)_{il}=0,$ and hence
\begin{eqnarray*}
f^{\left(-k\right)}_{j,\alpha}&=& \left(g^{(k+1)}_{\alpha\beta}\right)_{il}\left(\varphi^{\le k+1}_\beta\right)_*Q_{k+1}\left(\xi^{(k+1)}_l\right)\\
&=& \left(g^{(k+1)}_{\alpha\beta}\right)_{il}f_{l,\beta}^{\left(k\right)}.\
\end{eqnarray*}
Thus, we have constructed a global section $\lambda$ of $\pi_{\le k}^*E_{k+1}\left[-k\right].$ 

Let $$r_k:\pi_{\le k}^*E_{k+1}\left[-k\right] \to \M_{k}$$ be the bundle projection. The pullback bundle $r_k^*\pi_{\le k}^*E_{k+1}\left[-k\right] \to \pi_{\le k}^*E_{k+1}\left[-k\right]$ has a canonical section $\Delta$ coming from the diagonal map. The insertion operator $\iota_{\Delta}$ then can be identified with a canonical cohomological vector field on $r_k^*\pi_{\le k}^*E_{k+1}\left[-k-1\right],$ which in local coordinates can be identified with the Koszul complex of the fiber-coordinates of the bundle $\pi_{\le k}^*E_{k+1}\left[-k\right].$ Moreover, $\iota_{\Delta}$ commutes with natural cohomological vector field $\tilde Q_{k}$ induced from $Q_k.$ Therefore $\iota_{\Delta}+\tilde Q_{k}$ is a cohomological vector field. It is in fact nothing but the natural cohomological vector field coming from writing $\pi_{\le k}^*E_{k+1}\left[-k\right]$ as a pullback of dg-manifolds
$$\xymatrix{\pi_{\le k}^*E_{k+1}\left[-k\right] \ar[r] \ar[d] & \left(E_{k+1}\left[-k\right],\iota_{\Delta}\right) \ar[d] \\
\M_{k} \ar[r]^-{\pi_{\le k}} & M.}$$
By the proof of Proposition \ref{prop:zpb}, the composite
$$\left(\pi_{k+1}^*E_{k+1}\left[k-1\right],\iota_{\Delta}\right) \to E_{k+1}\left[-k\right] \to M$$ is a trivial fibration. Since trivial fibrations are stable under pullback,
$$\left(r_k^*\pi_{\le k}^*E_{k+1}\left[-k-1\right],\iota_{\Delta}\right) \to \pi_{\le k}^* E_{k+1}\left[-k\right] \stackrel{\rho}{\longrightarrow} \mathbb{E}_{\le k}$$ is as well. Moreover, the zero section of 
$\pi_{\le k}^*E_{k+1}\left[-k\right]$ factors a section of this map, followed by the map $\rho$ above.
By construction, we have that
$$\xymatrix{\left(\M_{k+1},Q_{k+1}\right) \ar[d]_-{p_{k+1}} \ar[r] & \left(r_k^*\pi_{\le k}^*E_{k+1}\left[-k-1\right],\iota_{\Delta}\right) \ar[d]\\
\left(\M_{k},Q_k\right) \ar[r]_-{\lambda} &  \left(\pi_{\le k}^*E_{k+1}\left[-k\right],0\right)}$$ is a pullback diagram, and since the right vertical arrow is a fibration, it is therefore a homotopy pullback, and by above can be identified with the derived fibered product of $\lambda$ with the zero-section. As this construction works for all $k,$ it in particular holds for $m,$ proving the theorem.
\end{proof}

\section{The equivalence of $\i$-categories}\label{sec:eq}
\subsection{Fully faithfulness}
Denote by $\dgMani$ the $\i$-category obtained from $\dgMan$ by formally inverting the weak equivalences.

Consider the canonical functor of $1$-categories
\begin{eqnarray*}
\Ci:\dgMan &\to& \left(\dgcni\right)^{op}\\
\M  &\mapsto& \Ci\left(\M\right)
\end{eqnarray*}

By Theorem \ref{thm:weqi}, it sends weak equivalences to quasi-isomorphisms, hence there is an induced functor between $\i$-categories
$$\Ci:\dgMani \to \dgci^{op}.$$
We will prove it is fully faithful.

\begin{lemma}\label{lem:hfp}
Let $\M$ be a dg-manifold. Then $\Ci\left(\M\right)$ is homotopically finitely presented.
\end{lemma}

\begin{proof}
Let $\M$ be a dg-manifold. Choose an embedding $M \hookrightarrow \R^n$ as a closed submanifold. Let $M \subset U_i$ for $i=1,2$ be tubular neighborhoods such that $U_1$ is a proper subset of $U_2.$ Let $f_i:M \hookrightarrow U_i$ be the inclusion with retraction $p_i.$ Let $\mathcal{U}_i$ be the pullback of dg-manifold
$$\xymatrix{ U_i \times_{M} \M \ar[r] \ar[d] & \M \ar[d]\\
U \ar[r]_-{p_i} & M.}$$ Then $\M$ is a retract of both $\mathcal{U}_1$ and $\mathcal{U}_2$. Let $Q$ be the cohomological vector field on $\mathcal{U}_2.$ Let $\lambda$ be a smooth function on $\R^n$ such that is identical to $1$ on $U_1$ and vanishes outside of $U_2.$
Let $E_1,\ldots,E_m$ be a sequence of vector bundles over $U_2$ such that $$\mathcal{U}_2\cong \bigoplus\limits_{i=1}^{m} E_i\left[-i\right]$$ as a graded manifold. For each $i,$ choose a vector bundle $F_i$ such that $E_i \oplus F_i$ is trivializable to $U_2 \times \R^{N_i}.$ Denote by $V_i$ the trivial vector bundle over $\R^n$ of rank $N_i.$ Consider the graded manifold
$$\cN=\bigoplus\limits_{i=1}^{m} V_i\left[-i\right].$$ Equip if with the odd vector field $\lambda \cdot Q,$ which is cohomological since $Q \lambda = 0.$ The core of this dg-manifold is $\R^n,$ and $\cN \times_{\R^n} U_2$ is a dg-manifold of which $\mathcal{U}_2$ is a retract. Notice that $\Ci\left(\cN\right)$ is quasi-free and finitely generated. By Proposition \ref{prop:qfcpt}, $\Ci\left(\cN\right)$ is homotopically finitely presented. Since $\Ci\left(\cN \times_{\R^n} U_2\right)$ is a localization of $\Ci\left(\cN\right),$ it is too (since $\Ci\left(\R\right)$ and $\Ci\left(\R\setminus\{0\}\right)$ are and compact objects are stable under pushouts). It follows that $\Ci\left(\cN \times_{\R^n} U_2\right)$ is a compact object. Hence is any retract of it is as well. Therefore $\Ci\left(\M\right)$ is a compact object, i.e. homotopically finitely presented.
\end{proof}

\begin{corollary}
The functor $\Ci:\dgMani \to \dgci^{op}$ induces a functor
$$\Psi:\dgMani \to \DMfd,$$ from the $\i$-category of dg-manifolds to the $\i$-category of derived manifolds.
\end{corollary}

\begin{proof}
This follows immediately from Corollary \ref{cor:hfpdman}.
\end{proof}

\begin{proposition}\label{prop:pbspres}
$$\Ci:\dgMani \to \dgci^{op}$$
preserves pullbacks.
\end{proposition}

\begin{proof}
It suffices to prove that if
$$\xymatrix{\mathfrak{P} \ar[r] \ar[d] & \mathfrak{M} \ar[d]^-{p} \\
\mathfrak{L} \ar[r]_-{f} & \mathfrak{N}}$$
is a pullback diagram of $\End\left(\R^{0|1}\right)$-supermanifolds, with $\cL,$ $\M,$ $\cN,$ and $\cP$ the corresponding dg-manifolds, and $p$ a submersion, then
$\Ci\left(\cP\right)$ is the homotopy pushout $$\Ci\left(\cN\right) \underset{\Ci\left(\cL\right)} \oinfty^{\mathbb{L}} \Ci\left(\M\right)$$ in $\dgcni.$
From Theorem \ref{thm:sch}, we know that
$$\Speci\left(\Ci\left(\cP\right)\right) \simeq \left(\pi_0\left(\cP\right),\O_{\cP}|_{\pi_0\left(\cP\right)}\right).$$
There is a canonical isomorphism
$$\pi_0\left(\cP\right) \cong \pi_0\left(\cL\right)\times_{\pi_0\left(\cN\right)} \pi_0\left(\M\right).$$
Since $$\Speci:\dgci^{op} \to \Loc$$ preserves limits, we also have that
$$\pi_0\left(\cL\right)\times_{\pi_0\left(\cN\right)} \pi_0\left(\M\right) \cong \Speci\left(H^0\left(\Ci\left(\cN\right) \underset{\Ci\left(\cL\right)} \oinfty^{\mathbb{L}} \Ci\left(\M\right)\right)\right).$$ Hence we have an identification of underlying spaces between $\Speci\left(\Ci\left(\cP\right)\right)$ and $$\Speci\left(\Ci\left(\cN\right) \underset{\Ci\left(\cL\right)} \oinfty^{\mathbb{L}} \Ci\left(\M\right)\right).$$ Let $\left(z,x\right) \in \pi_0\left(\cL\right) \times_{\pi_0\left(\cN\right)} \pi_0\left(\M\right)$. Then $x$ can be identified with a point in $\mathfrak{M}.$ Since $p$ is a submersion, there exists an open neighborhood $\mathcal{U}$ of $x,$ and a corresponding open neighborhood $\mathcal{W}$ of $p\left(x\right)$ and diffeomorphisms making the following diagram commute
$$\xymatrix{\mathcal{U} \ar[r]^-{\sim} \ar[d]_-{p} & \R^{n|m} \times \R^{l|k} \ar[d] \\
\mathcal{W} \ar[r]^-{\sim} & \R^{n|m}},$$
where we are identifying the open subsets with their super submanifolds, and we are assuming that $\mathcal{U} \subseteq p^{-1}\left(\cW\right).$ Let $\mathcal{V}:=\left(f^{-1}\left(\cW \right) \times \mathcal{U}\right)\cap \cP.$ Notice that
$$\mathcal{V} \cap \pi_0\left(\cP\right)$$ is an open neighborhood of $\left(z,x\right).$
By Lemma \ref{lem:piok} and Theorem \ref{thm:sch}, it follows that
$$\O_{\Speci\left(\Ci\left(\cP\right)\right)}\left(\mathcal{V}\right) \cap \pi_0\left(\cP\right) \simeq \O_{\cP}\left(\mathcal{V}\right).$$ 
Also, we have that
$$\Speci\left(\Ci\left(\cN\right) \underset{\Ci\left(\cL\right)} \oinfty^{\mathbb{L}} \Ci\left(\M\right)\right)|_{\mathcal{V}} \simeq \Speci\left(\Ci\left(f^{-1}\cW\right)\right) \times_{\Speci\left(\Ci\left(\cW\right)\right)} \Speci\left(\Ci\left(\mathcal{U}\right)\right).$$ So the value of the structure sheaf assigns $\cV$ is the homotopy pushout
$$\Ci\left(\left(f^{-1} \cW\right)\cap \cL\right) \underset{\Ci\left(\cW\cap \cN\right)}\oinfty^{\mathbb{L}} \Ci\left(\mathcal{U}\cap \M\right).$$ But $\Ci\left(\cW \cap \cN\right) \to \Ci\left(\mathcal{U} \cap \M\right)$ is a quasi-free extension since the corresponding map of supermanifolds $\mathcal{U} \to \cW$ is just projecting away the $\R^{l|k}$ factor. Hence the homotopy pushout is the ordinary pushout. Notice that we can write $$\mathcal{U} \cong \cW \times \R^{l|k},$$ and identify  
$$\left(f^{-1}\left(\cW \right) \times \mathcal{U}\right) \cap \mathfrak{P} \cong f^{-1}\left(\cW\right) \times \R^{l|k}.$$ 
Since at the level of algebras a pushout of a free extension is always a free extension, we can identify
$$\Ci\left(\left(f^{-1} \cW\right)\cap \cL\right)  \to \Ci\left(\left(f^{-1} \cW\right)\cap \cL\right) \underset{\Ci\left(\cW\cap \cN\right)}\oinfty \Ci\left(\mathcal{U}\cap \M\right)$$ with $$\Ci\left(f^{-1}\cW \cap \cL\right) \to \Ci\left(\mathcal{V} \cap \cP\right).$$ It follows that $$\Speci\left(\Ci\left(\cP\right)\right) \simeq \Speci \left(\Ci\left(\cN\right) \underset{\Ci\left(\cL\right)} \oinfty^{\mathbb{L}} \Ci\left(\M\right)\right).$$
By Lemma \ref{lem:hfp}, $\Ci\left(\cL\right)$,$\Ci\left(\M\right)$ and $\Ci\left(\cN\right)$ are compact objects in $\dgci$. Since pushouts of compact objects are compact, it follows that $$\Ci\left(\cN\right) \underset{\Ci\left(\cL\right)} \oinfty^{\mathbb{L}} \Ci\left(\M\right)$$ is compact, i.e. homotopically finitely presented. Proposition \ref{prop:comp} then implies that it is equal to global sections of its structure sheaf. Since the same is true for $\Ci\left(\cP\right),$ we are done.
\end{proof}

\begin{theorem}\label{thm:ff}
$$\Ci:\dgMani \to \dgci^{op}$$ is fully faithful.
\end{theorem}

\begin{proof}
By Theorem \ref{thm:decomp}, every dg-manifold can be obtained through a finite iteration of pullbacks in $\dgMani$ starting with manifolds. This reduces the task to showing that for any dg-manifold $\M$
$$\Map_{\dgMani}\left(\M,N\right) \simeq \Map_{\dgci}\left(\Ci\left(N\right),\Ci\left(\M\right)\right)$$ for any manifold $N.$ In the same vein, every manifold is a retract of a transverse pullback of Cartesian manifolds, i.e. ones of the form $\R^n,$ and every $\R^n$ is of course a finite product of $\R.$ So this further reduces to showing that 
$$\Map_{\dgMani}\left(\M,\R\right) \simeq \Map_{\dgci}\left(\Ci\left(\R\right),\Ci\left(\M\right)\right).$$
Notice that for all $k,$ $\left(\R\left[-k\right],+\right)$ is an abelian group object in both $\dgMan$ and $\dgMani,$ since products are homotopy fibered products. Moreover, 
$\Ci$ is a functor which preserves finite products, hence the canonical map 
$$\Map_{\dgMani}\left(\M,\R\left[-k\right]\right) \to \Map_{\dgci}\left(\Ci\left(\R\left[-k\right]\right),\Ci\left(\M\right)\right)$$ is a homomorphism of abelian group objects in $\Spc.$ Applying $\pi_0,$ we have simply a homomorphism of ordinary abelian groups
$$F:\pi_0\Map_{\dgMani}\left(\M,\R\left[-k\right]\right) \to \pi_0\Map_{\dgci}\left(\Ci\left(\R\left[-k\right]\right),\Ci\left(\M\right)\right).$$ Notice that in the $1$-category case, we have
$\Hom_{\dgMan}\left(\M,\R\left[-k\right]\right)\cong Z^{-k} \Ci\left(\M\right),$ the abelian group of elements $f$ of degree $-k$ in $\Ci\left(M\right)$ such that $Df=0.$ The functor
$$h_{\cW}:\dgMan \to \dgMani$$ hence induces a homomorphism
$$Z^{-k}\left(\Ci\left(\M\right)\right) \to \pi_0\Map_{\dgMani}\left(\M,\R\left[-k\right]\right).$$ Moreover, as a pointed dg-manifold, $\R\left[-k\right]\cong \Omega^k \R.$ By this we mean that $\R\left[-1\right]$ is the derived intersection of the origin with itself in $\R,$ $\R\left[-2\right]$ is the derived intersection of the origin with itself in $\R\left[-1\right],$ and so on. From this it follows that $$\Map_{\dgMani}\left(\M,\R\left[-k\right]\right)\simeq \Omega^k \Map_{\dgMani}\left(\M,\R\right).$$ Applying $\Ci,$ we also have that $\Ci\left(\R\left[-k\right]\right)$ is the $k$-fold looping of $\Ci\left(\R\right)$ in $\dgci^{op}.$ Hence
$$\Map_{\dgci}\left(\Ci\left(\R\left[-k\right]\right),\Ci\left(\M\right)\right) \simeq \Omega^k  \Map_{\dgci}\left(\Ci\left(\R\right),\Ci\left(\M\right)\right).$$
So, by Corollary \ref{cor:dk}, we have
\begin{eqnarray*}
\pi_0 \Map_{\dgci}\left(\Ci\left(\R\left[-k\right]\right),\Ci\left(\M\right)\right) &\cong & \pi_0 \Omega^k  \Map_{\dgci}\left(\Ci\left(\R\right),\Ci\left(\M\right)\right)\\
&\cong& \pi_0 \Omega^k \Ci\left(\M\right)\\
&\cong& \pi_k \Gamma\left(\Ci\left(\M\right)\right)\\
&\cong& H^{-k}\left(\Ci\left(\M\right)\right),
\end{eqnarray*}
where $\Gamma\left(\Ci\left(\M\right)\right)$ is the corresponding simplicial abelian group. Since the composite homomorphism
$$\xymatrix{Z^{-k}\left(\Ci\left(\M\right)\right) \cong  \Hom_{\dgMan}\left(\M,\R\left[-k\right]\right)\ar[d]\\
\pi_0\Map_{\dgMani}\left(\M,\R\left[-k\right]\right) \ar[d]_-{F}\\
\pi_0\Map_{\dgci}\left(\Ci\left(\R\left[-k\right]\right),\Ci\left(\M\right)\right) \cong H^{-k}\left(\Ci\left(\M\right)\right)}$$
is surjective, so is $F.$ Moreover, by \cite[Proposition 3.23]{Denis}, the mapping space $$\Map_{\dgMani}\left(\M,\R\left[-k\right]\right)$$ is the geometric realization of a cocycle category, i.e., a category of spans $\M \stackrel{\sim}{\leftarrow} \M' \to \R\left[-k\right],$ so every element $\pi_0\Map_{\dgMani}\left(\M,\R\left[-k\right]\right)$ is of the form $\left[f\right],$ for $f:\M' \to \R\left[-k\right],$ with $\M'$ quasi-isomorphic to $\M.$ If $F\left(\left[f\right]\right)=0,$ by repeating the same argument with $\M',$ this means that the $-k^{th}$-cohomology class of $f$ in $\Ci\left(\M'\right)$ is zero. This means that there exists $g \in \Ci\left(\M\right)$ of degree $-k-1$ such that $Dg=f.$ So, there exists a dotted homomorphism $\varphi$ as in
$$\xymatrix{\left(\Ci\{x,\epsilon\},x \frac{\partial}{\partial \epsilon}\right) \ar@{-->}[r]^-{\varphi} & \Ci\left(\M'\right)\\
\Ci\{x\} \ar[u] \ar[ru]_-{f} & }$$
where $|x|=-k$ and $|\epsilon|=-k-1.$ On one hand, the algebra $\left(\Ci\{x,\epsilon\},x \frac{\partial}{\partial \epsilon}\right)$ is the Koszul algebra $K\left(\Ci\{x\},x\right),$ and $x$ is not a zero divisor, so 
$\left(\Ci\{x,\epsilon\},x \frac{\partial}{\partial \epsilon}\right)$ is quasi-isomorphic to $\Ci\{x\}/\left(x\right)=\R.$
But $\left(\Ci\{x,\epsilon\},x \frac{\partial}{\partial \epsilon}\right)$ is also $\Ci\left(D\right),$ where $D$ is the dg-manifold with underlying graded manifold $\R\left[-k\right]\times\R\left[-k-1\right],$ with cohomological vector field $x \frac{\partial}{\partial \epsilon}.$ So we have a factorization of $f:\M' \to \R\left[-k\right]$ as
$$\M' \stackrel{\tilde g}{\longrightarrow} D \stackrel{f'}{\longrightarrow} \R\left[-k\right],$$ so $f$ is in the essential image of
$$\Map_{\dgMani}\left(\M,D\right) \to \Map_{\dgMani}\left(\M,\R\left[-k\right]\right).$$
But $D$ is weakly equivalent to the point, so $\Map_{\dgMani}\left(\M,D\right)$ is contractible, so this means that $f$ is in the same connected component as the constant map $\M \to \R\left[-k\right]$ at $0,$ and hence $\left[f\right]=0$ in the abelian group 
$$\pi_0\Map_{\dgMani}\left(\M,\R\left[-k\right]\right).$$ So $F$ is also injective, and hence an isomorphism. But since
\begin{eqnarray*}
\pi_0\Map_{\dgMani}\left(\M,\R\left[-k\right]\right)&\cong& \pi_0 \Omega^k\Map_{\dgMani}\left(\M,\R\right)\\
&\cong& \pi_k \Map_{\dgMani}\left(\M,\R\right)
\end{eqnarray*}
and similarly
$$\pi_0\Map_{\dgci}\left(\Ci\left(\R\left[-k\right]\right),\M\right)\cong \pi_k \Map_{\dgci}\left(\Ci\left(\R\right),\Ci\left(\M\right)\right),$$
it follows that the induced map of Kan complexes
$$\Map_{\dgMani}\left(\M,\R\right) \simeq \Map_{\dgci}\left(\Ci\left(\R\right),\Ci\left(\M\right)\right)$$
is a bijection on $\pi_0$ and induces an isomorphism on all homotopy groups with basepoint at the constant map with value $0,$ but since the spaces are abelian group objects in $\Spc,$ they are group-like $E_\i$-spaces, so this suffices to prove the map is a weak equivalence.
\end{proof}

\begin{corollary}
$$\Psi:\dgMani \to \DMfd$$ is fully faithful.
\end{corollary}

\subsection{Essential surjectivity}

If $\sC$ is any category of fibrant objects, since it has finite homotopy pullbacks and a terminal object and every object is fibrant, it follows that $\sC\left[\cW^{-1}\right]_\i$ has finite limits. Since $\Psi$ is fully faithful, and by Proposition \ref{prop:pbspres} it preserves finite limits (as it clearly preserves the terminal object) it follows that the essential image of $\Psi$ is closed under finite limits in $\DMfd.$ By Corollary \ref{cor:retdmfd}, this implies that every derived manifold is a retract of a derived manifold in the essential image of $\Psi.$ We wish to show that $\Psi$ is in fact essentially surjective, i.e. closed under these retracts. We will need a series of preliminary results.

Denote by $\Idem$ the free $1$-category on one idempotent $e:C \to C.$ Unwinding the definitions, one can see that for an $\i$-category $\cC,$ a functor $E:\Idem \to \cC$ is the same thing as as an idempotent in the sense of Definition \cite[Definition 4.4.5.4]{htt}--- this is just a homotopy-coherent idempotent.

\begin{lemma}\label{lem:cofproj}
Let $\sC$ be a combinatorial model category. Let an idempotent $e:C \to C$ in $\sC$ be regarded as a diagram $E:\Idem \to \sC.$ If $C$ is a cofibrant, then $e$ is projectively cofibrant.
\end{lemma}

\begin{proof}
Let $b:B \to B$ and $a:A \to A$ be idempotents. By definition of the projective model structure, a map $\pi:A \to B$ intertwining the idempotents is a trivial fibration if and only if $\pi$ is in $\sC.$ It suffices to show that given such a $\pi,$ if $B$ is cofibrant, there is a section $s$ of $\pi$ such that $sb=as.$ Since $\sC$ is idempotent complete, we can factor $b$ as
$$B \stackrel{r}{\longrightarrow} C \stackrel{i}{\longhookrightarrow} B,$$ with $r\circ i=id_{C},$ and factor $a$ as
$$A \stackrel{p}{\longrightarrow} K \stackrel{j}{\longhookrightarrow} A,$$ with $jp=a.$
 Since $C$ is a retract of $B,$ $B$ is also cofibrant. Therefore, the pullback trivial fibration $$C \times_{B} A \to C$$ has a section $\lambda.$ Let $s=j\lambda r.$ Then, since $b=ir,$ we have that $sb=s.$ But also, since $a=jp,$ $as=s.$ This implies that $\lambda$ is a morphism $\left(B,b\right) \to \left(A,a\right),$ which is a section of $\pi.$ It follows that $\left(B,b\right)$ is projectively cofibrant.  
\end{proof}

\begin{definition}
Let $\Qf$ denote the subcategory of $\dgcni$ on finitely generated quasi-free dg-$\Ci$-algebras, and let $\overset{\circ} \Qf$ denote its associated subcategory of the $\i$-category $\dgci.$
\end{definition}

\begin{remark}\label{rmk:qf}
The functor $\Ci:\dgMan \to \dgcni$ restricts to an equivalence between the full subcategory of $\dgMan$ on those dg-manifolds with Euclidean base, and $\Qf^{op}.$
\end{remark}

\begin{lemma}\label{lem:factor}
Let $p:\A \to \B$ be a map in $\Qf.$ Then there exists a canonical factorization
$\A \to \overline{B} \to B,$ with $\overline{B} \in \Qf,$ such that $A \to \overline{B}$ is a cofibration and $\overline{B} \to B$ is a trivial fibration.
\end{lemma}

\begin{proof}
Let $\A=\left(\Ci\{x_1,\ldots,x_n\},d=f^i\frac{\partial}{\partial x_i}\right),$ with $|x_i|=n_i,$ and let $\B=\left(\Ci\{y_1,\ldots,y_m\},d=g^j\frac{\partial}{\partial y_j}\right),$ with $|y_i|=m_i.$ For all $i,$ let $p_i:=p\left(x_i\right) \in \cB.$ Let the underlying graded algebra of $$\overline{B}=\Ci\{x_1,\ldots,x_n,\overline{x_1},\ldots,\overline{x_n},y_1,\ldots,y_m\},$$ with $|\overline{x_i}|=n_i-1,$ and it equip it with the differential
$$D=f^i\frac{\partial}{\partial x_i} + g^j\frac{\partial}{\partial y_j} + \left(p_i-x_i -\overline{x_k}\frac{\partial f^i}{\partial x_k}\right)\frac{\partial}{\partial \overline{x_i}}.$$
It is routine to check that this squares to zero.
Let $q:\A \to \overline{B}$ be given y $q\left(x_i\right)=x_i,$ and $r:\overline{B} \to  \B$ be given by $r\left(x_i\right)=p_i,$ $r\left(\overline{x_i}\right)=0,$ and $r\left(y_j\right)=y_j.$ It is easy to verify that both of these are homomorphisms of dg-algebras. Note that there is also a homomorphism $\varphi:\cB \to \overline{\cB}$ defined by $\varphi\left(y_j\right)=y_j,$ which is a section of $r.$
We have a factorization $p=r \circ q,$ and $q$ is a quasi-free extension, hence a cofibration. Moreover, $\overline{B}$ is in $\Qf.$ The morphism $r$ is surjective, hence a fibration. It suffices to prove it is a quasi-isomorphism. Let $\Gamma:=\overline{x_i} \frac{\partial}{\partial x_i}.$ Then $\Gamma$ is a degree $-1$ derivation of $\overline{\cB}.$ One can check that
$$\left[D,\Gamma\right]=\left(D \circ \Gamma + \Gamma \circ D\right)=jr-id.$$ It follows that $jr$ is chain homotopic to the identify. Since $rj=id,$ this implies in particular that $r$ is a quasi-isomorphism, hence a trivial fibration.
\end{proof}

\begin{proposition}\label{prop:fincol}
$\overset{\circ} \Qf$ is closed under finite colimits in $\dgci.$
\end{proposition}

\begin{proof}
As it contains the initial object, it suffices to prove that $\Qf$ is closed under homotopy pushouts in $\dgcni.$ By Lemma \ref{lem:factor}, one can form homotopy pushouts explicitly by replacing one map with a quasi-free extension. This means that this pushout will also be a quasi-free extension of a quasi-free algebra, and hence quasi-free (and clearly finitely generated).
\end{proof}




\begin{theorem}\label{thm:main}
The functor $\Psi:\dgMani \to \DMfd$ is an equivalence of $\i$-categories.
\end{theorem}

\begin{proof}
The functor is fully faithful by Theorem \ref{thm:ff}. It suffices to prove it is essentially surjective. For this, it suffices to show that every homotopically finitely presented dg-$\Ci$-algebra is in the essential image of $\Ci.$ By Proposition \ref{prop:hfpret}, it follows that such an algebra is a retract of a colimit of free $\Ci$-algebras, i.e. ones of the form $\Ci\left(\R^n\right),$ for some $n.$ But, each of these is in $\Qf,$ so by Proposition \ref{prop:fincol}, it follows that the colimit is in $\overset{\circ}\Qf.$ Let $$E:\Idem \to \dgci$$ be the homotopy idempotent that lies in $\overset{\circ}\Qf.$ By \cite[Proposition A.3.4.12]{htt}, we can model such a diagram as an actual retract in $\dgcni$ 
$$e:\A \to \A,$$ with $\A$ connected to $\B$ in $\Qf$ through a zig-zag of quasi-isomorphisms. Since $\cB$ is cofibrant, and all objects are fibrant, we can choose a legitimate quasi-morphism $f:\B \to \A.$ We can then split the idempotent $e,$ and then pull it back via $f$ to get an induced-retract diagram, and hence an induced idempotent  $e'$ on $\B.$ We manifestly have an induced weak-equivalence between these two idempotents in the projective model structure on $\Fun\left(\Idem,\dgcni\right).$ Moreover, since $\B$ is cofibrant, it follows by Lemma \ref{lem:cofproj} that $e'$ is projectively cofibrant. Therefore its splitting is a homotopy splitting. However, since $\cB$ is equivalent to the algebra of functions on a dg-manifold, so is any of its retracts. This shows that the retract of a finite colimit of free $\Ci$-algebras is quasi-isomorphic to functions on a dg-manifold. Since every homotopically finitely generated dg-$\Ci$-algebra is of this form, this finishes the proof.
\end{proof}

\begin{corollary}
A dg-$\Ci$-algebra is homotopically finite generated if and only if it is quasi-isomorphic to $\Ci\left(\M\right)$ for a dg-manifold $\M.$
\end{corollary}

\bibliography{research}
\bibliographystyle{hplain}

\end{document}